\documentclass[a4paper,10pt]{article}

\usepackage[english]{babel}
\usepackage{amsmath,amsfonts,amssymb,amsthm}
\usepackage{mathrsfs}
\usepackage{bbm}
\usepackage{indentfirst}

\usepackage{graphicx}

\newtheorem{theorem}{Theorem}[section]

\newtheorem{lemma}{Lemma}[section]
\newtheorem{proposition}{Proposition}[section]
\newtheorem{corollary}{Corollary}[section]

\theoremstyle{definition}
\newtheorem{remark}{Remark}[section]

\numberwithin{equation}{section}

\newcommand\blfootnote[1]{\begingroup\renewcommand\thefootnote{}\footnote{#1}\addtocounter{footnote}{-1}\endgroup}

\begin{document}

\title{
{\bf\Large Multiplicity of positive periodic solutions in the superlinear indefinite case via coincidence degree}
\footnote{Work supported by the Grup\-po Na\-zio\-na\-le per l'Anali\-si Ma\-te\-ma\-ti\-ca, la Pro\-ba\-bi\-li\-t\`{a} e le lo\-ro
Appli\-ca\-zio\-ni (GNAMPA) of the Isti\-tu\-to Na\-zio\-na\-le di Al\-ta Ma\-te\-ma\-ti\-ca (INdAM).
Progetto di Ricerca 2015: Problemi al contorno associati ad alcune classi di equazioni differenziali non lineari.}}
\author{{\bf\large Guglielmo Feltrin}
\vspace{1mm}\\
{\it\small SISSA - International School for Advanced Studies}\\
{\it\small via Bonomea 265}, {\it\small 34136 Trieste, Italy}\\
{\it\small e-mail: guglielmo.feltrin@sissa.it}\vspace{1mm}\\
\vspace{1mm}\\
{\bf\large Fabio Zanolin}
\vspace{1mm}\\
{\it\small Department of Mathematics and Computer Science, University of Udine}\\
{\it\small via delle Scienze 206},
{\it\small 33100 Udine, Italy}\\
{\it\small e-mail: fabio.zanolin@uniud.it}\vspace{1mm}}

\date{}

\maketitle

\vspace{-2mm}

\begin{abstract}
\noindent
We study the periodic boundary value problem associated with the second order nonlinear differential equation
\begin{equation*}
u'' + c u' + \bigl{(}a^{+}(t) - \mu a^{-}(t)\bigr{)} g(u) = 0,
\end{equation*}
where $g(u)$ has superlinear growth at zero and at infinity, $a(t)$ is a periodic sign-changing weight, $c\in\mathbb{R}$ and $\mu>0$ is a real parameter.
We prove the existence of $2^{m}-1$ positive solutions when $a(t)$ has $m$ positive humps separated by $m$ negative ones
(in a periodicity interval) and $\mu$ is sufficiently large.
The proof is based on the extension of Mawhin's coincidence degree defined in open (possibly unbounded) sets and applies also to Neumann boundary conditions.
Our method also provides a topological approach to detect subharmonic solutions.
\blfootnote{\textit{AMS Subject Classification:} 34B18, 34B15, 34C25, 47H11.}
\blfootnote{\textit{Keywords:} superlinear indefinite problems, positive periodic solutions,
multiplicity results, subharmonic solutions, Neumann boundary value problems, coincidence degree.}
\end{abstract}

\section{Introduction}\label{section-1}

Let ${\mathbb{R}}^{+}:=\mathopen{[}0,+\infty\mathclose{[}$ denote the set of non-negative real numbers and let
$g\colon {\mathbb{R}}^{+} \to {\mathbb{R}}^{+}$ be a continuous function such that
\begin{equation*}
g(0) = 0, \qquad g(s) > 0 \quad \text{for } \; s > 0.
\leqno{(g_{*})}
\end{equation*}
In the present paper we study the problem of existence and multiplicity of positive $T$-periodic solutions
to the second order nonlinear differential equation
\begin{equation}\label{eq-1.1}
u'' + c u' + w(t) g(u) = 0,
\end{equation}
where $c\in\mathbb{R}$ and $w \colon \mathbb{R}\to\mathbb{R}$ is a $T$-periodic locally integrable weight function.
Solutions to \eqref{eq-1.1} are meant in the Carath\'{e}odory sense. A \textit{positive solution} is a solution
such that $u(t) > 0$ for all $t\in\mathbb{R}$. As is well known, solving the $T$-periodic problem for \eqref{eq-1.1}
is equivalent to find a solution of \eqref{eq-1.1} satisfying the boundary conditions
\begin{equation*}
u(0) = u(T), \qquad u'(0) = u'(T),
\end{equation*}
on $\mathopen{[}0,T\mathclose{]}$ (any other time-interval of length $T$ can be equivalently chosen).

Our aim is to consider a nonlinear vector field
\begin{equation*}
f(t,s) := w(t) g(s)
\end{equation*}
satisfying suitable assumptions which cover the classical superlinear indefinite case, namely $g(s) = s^{p}$, with $p>1$,
and $w(t)$ a sign-changing coefficient. Our main result guarantees the existence of at least $2^{m}-1$
positive $T$-periodic solutions provided that, in a time-interval of length $T$, the weight function presents
$m$ positive humps separated by negative ones and the negative parts of $w(t)$ are sufficiently large.
To be more precise, it is convenient to express $w(t)$ as depending on a parameter $\mu > 0$ in this manner:
\begin{equation*}
w(t) = a_{\mu}(t) := a^{+}(t) - \mu a^{-}(t),
\end{equation*}
where $a \colon \mathbb{R}\to\mathbb{R}$ is a $T$-periodic locally integrable function.
As usual, we denote by
\begin{equation*}
a^{+}(t):= \dfrac{a(t)+|a(t)|}{2} \quad \text{ and } \quad a^{-}(t):= \dfrac{-a(t)+|a(t)|}{2}
\end{equation*}
the \textit{positive part} and the \textit{negative part} of $a(t)$, respectively.
Then, a typical corollary of our main result (cf.~Theorem~\ref{MainTheorem}) reads as follows.

\begin{theorem}\label{th-1.1}
Suppose that there exist $2m+1$ points
\begin{equation*}
\sigma_{1} < \tau_{1} < \ldots < \sigma_{i} < \tau_{i} < \ldots < \sigma_{m} <
\tau_{m} < \sigma_{m+1}, \quad \text{with } \; \sigma_{m+1} - \sigma_{1} = T,
\end{equation*}
such that $w(t) \succ 0$ on $\mathopen{[}\sigma_{i},\tau_{i}\mathclose{]}$ and $w(t) \prec 0$ on $\mathopen{[}\tau_{i},\sigma_{i+1}\mathclose{]}$.
Let $g(s)$ be a continuous function satisfying $(g_{*})$ and such that
\begin{equation*}
g_{0}:=\lim_{s\to 0^{+}} \dfrac{g(s)}{s} = 0
\quad \text{ and } \quad
g_{\infty}:=\liminf_{s\to +\infty} \dfrac{g(s)}{s} > 0.
\end{equation*}
If $g_{\infty}$ is sufficiently large, then there exists $\mu^{*} > 0$ such that for each $\mu > \mu^{*}$
equation
\begin{equation*}
u'' + c u' + \bigl{(} a^{+}(t) - \mu a^{-}(t) \bigr{)} g(u) = 0
\end{equation*}
has at least $2^{m}-1$ positive $T$-periodic solutions.
\end{theorem}

In the statement of the theorem we have employed the usual notation $w(t) \succ 0$ in a given interval, to express the fact that
$w(t)\geq 0$ almost everywhere with $w\not\equiv 0$ on that interval. Moreover, $w(t) \prec 0$ stands for $-w(t) \succ 0$.

\begin{remark}\label{rem-1.1}
For the application of Theorem~\ref{th-1.1}, a lower bound for $g_{\infty}$ can be explicitly provided (see Theorem~\ref{MainTheorem}).
It depends only on $a^{+}(t)$ (and not on $\mu$ or on $a^{-}(t)$).
In any case, our result holds for
\begin{equation}\label{eq-1.2}
g_{0} = 0 \quad \text{ and } \quad g_{\infty} = +\infty,
\end{equation}
without any further assumption on $g(s)$ (see Corollary~\ref{cor-5.2}).
Therefore the power nonlinearity $g(s) = s^{p}$, with $p>1$, is covered.
Regarding \eqref{eq-1.2}, Theorem~\ref{th-1.1} can be seen as an extension to the
periodic case of the result by Erbe and Wang \cite{ErWa-1994}. In \cite{ErWa-1994} the
authors considered the two-point boundary value problem with a sign-definite weight
and proved the existence of at least one positive solution. In our situation, we allow the
weight coefficient to change its sign and obtain a multiplicity result.

We underline that the condition $\mu > \mu^{*}$ implies
\begin{equation*}
\int_{0}^{T} w(t) ~\!dt < 0,
\end{equation*}
which is a necessary condition for the existence of positive $T$-periodic solutions
when $g'(s) > 0$ on $\mathbb{R}^{+}_{0} := \mathopen{]}0,+\infty\mathclose{[}$ (cf.~\cite{Bo-2011,FeZa-2015ade}).

Concerning the condition at zero, we observe that it can be slightly improved to an hypothesis of the form
\begin{equation*}
\limsup_{s\to 0^{+}} \dfrac{g(s)}{s} < \lambda_{*},
\end{equation*}
where $\lambda_{*}$ is a positive constant which can be estimated from below (see Remark~\ref{rem-4.2}).

Besides the periodic problem, in the last part of our paper, we study also the Neumann boundary value problem for
radially symmetric solutions of the Laplace equation and obtain the same multiplicity results.
In this manner, we generalize \cite{Wa-1994}, where the existence of radially symmetric solutions
for the Dirichlet problem was obtained for a sign-definite weight.
$\hfill\lhd$
\end{remark}

A simple application of Theorem~\ref{th-1.1} is the following.
The equation
\begin{equation*}
u'' + c u' + \bigl{(}\sin^{+}(2t) - \mu \sin^{-}(2t)\bigr{)} u^{p} = 0, \quad p>1,
\end{equation*}
has at least three positive $2\pi$-periodic solutions if $\mu > 0$ is sufficiently large.
The same result obviously holds for
\begin{equation*}
u'' + c u' + \bigl{(}\cos^{+}(2t) - \mu \cos^{-}(2t)\bigr{)} u^{p} = 0, \quad p>1.
\end{equation*}
Moreover, we can also prove that for an equation like
\begin{equation*}
u'' + c u' + \bigl{(}\nu \sin^{+}(m t) - \mu \sin^{-}(m t)\bigr{)} \, u \arctan(u) = 0,
\end{equation*}
with $m$ a fixed positive integer, there are at least $2^m -1$ positive $2\pi$-periodic solutions provided
that $\nu > \nu^{*}$ and $\mu > \mu^{*}$ with $\mu^{*} = \mu^{*}(\nu)$ (see Corollary~\ref{cor-5.1} for more details).
In this latter example, $\nu$ large guarantees that $g_{\infty}$ (although finite) is large as well and
Theorem~\ref{th-1.1} can be applied.
A less immediate application of Theorem~\ref{th-1.1}, presented in Section~\ref{section-6},
ensures that in all the above three examples there are subharmonic solutions of an arbitrary order
and also bounded solutions which exhibit chaotic-like dynamics (in a sense that will be made more precise later).

\medskip

Our proofs make use of a topological degree argument in the sense that we define some open sets where an operator
associated with our boundary value problem has nonzero degree. Thus our results are stable with respect to small
perturbations. In this manner, for example, we can also extend Theorem~\ref{th-1.1} to equations like
\begin{equation*}
u'' + c u' + \varepsilon u + \bigl{(} a^{+}(t) - \mu a^{-}(t) \bigr{)} g(u) = 0,
\end{equation*}
with $|\varepsilon| < \varepsilon_{0}$, where $\varepsilon_{0}$ is a sufficiently small constant depending on $\mu$.
This latter equation is analogous to the one considered by Graef, Kong and Wang in \cite{GrKoWa-2008} of the form
\begin{equation*}
u'' - \rho^{2} u + \lambda w(t) g(u) = 0.
\end{equation*}
For this equation, the authors in \cite[Theorem~2.1~(a)]{GrKoWa-2008} obtained existence of positive periodic solutions,
for $\rho > 0$ (arbitrary) and $\lambda > 0$ large, when $w \succ 0$. With our result we can ensure
the multiplicity of solutions for a sign-changing weight when $\rho$ is small.

The stability with respect to small perturbations is not generally guaranteed when using different approaches, such as variational or
symplectic techniques which require some special structure (e.g.~an Hamiltonian). Concerning this aspect,
we stress that our results work equally well with respect to the presence or not of the friction term $c u'$
($c$ can be zero or nonzero, indifferently).

The technique we employ in this paper exploits and combines the approaches introduced in our recent works \cite{FeZa-2015ade,FeZa-2015jde}.
In \cite{FeZa-2015jde} we provide existence and multiplicity of positive solutions to a Dirichlet boundary value problem of the form
\begin{equation}\label{eq-1.3}
\begin{cases}
\, u'' + f(t,u) = 0 \\
\, u(0) = u(T) = 0,
\end{cases}
\end{equation}
using an extension of the classical Leray-Schauder degree to open and possibly unbounded sets (cf.~\cite{Nu-1985,Nu-1993}).
More in detail, since $u\mapsto -u''$ (with the Dirichlet boundary conditions) is invertible, following a standard procedure,
we write \eqref{eq-1.3} as an equivalent fixed point problem in a Banach space. Then we apply some degree theoretical arguments
on suitably chosen open sets which are carefully selected in order to discriminate between the solutions we are looking for.
With respect to the periodic (and Neumann) problem associated with \eqref{eq-1.1}, the linear differential operator $u\mapsto -u''-cu'$
has a nontrivial kernel made by constant functions, therefore it is not invertible and we can not operate as in \cite{FeZa-2015jde}.
For this reason, as in \cite{FeZa-2015ade}, we use the coincidence degree theory introduced by J.~Mawhin to deal with
a problem of the form $Lu = Nu$, where $L$ is a linear operator with nontrivial kernel and $N$ is a nonlinear one.
Unlike \cite{FeZa-2015ade}, where we prove the existence of at least a positive $T$-periodic solution to \eqref{eq-1.1},
in the present situation we need an extension of the coincidence degree to open not necessarily bounded sets (see Appendix~\ref{appendix-A}),
in order to take advantage of the degree argument proposed in \cite{FeZa-2015jde}. The possibility of obtaining the
existence of solutions in some special open sets (via coincidence degree theory) allows us also to prove the
presence of \textit{positive subharmonic solutions}, a fact which appears rather new in this framework.

\medskip

In the last twenty years the nonlinear indefinite problems received considerable attention,
especially in connection to the study of existence and multiplicity of solutions to boundary value problems
associated with equations of the form
\begin{equation}\label{eq-1.4}
-\Delta u = w(x)g(u).
\end{equation}
Equations of this type arise in many models concerning population dynamics, differential geometry and mathematical physics.

The study of superlinear ODEs with a sign-changing weight was initiated in 1965 by Waltman (cf.~\cite{Wa-1965}),
considering oscillatory solutions for
\begin{equation*}
u'' + w(t) u^{2n+1} = 0, \quad n\geq1.
\end{equation*}
Among the several authors who have continued this line of research, for the periodic problem,
we recall the relevant contributions of Butler \cite{Bu-1976}, Terracini and Verzini \cite{TeVe-2000}, and Papini \cite{Pa-2003},
who proved the existence of periodic solutions with a large number of zeros to superlinear indefinite equations of the form
\begin{equation}\label{eq-1.5}
u'' + w(t) g(u) = 0.
\end{equation}
The presence of chaotic dynamics for superlinear indefinite ODEs was discovered in \cite{TeVe-2000} in the study of
\begin{equation*}
u''+ K u + w(t) u^{3} = h(t)
\end{equation*}
(with the constant $K$ and the function $h(t)$ possibly equal to zero). In this framework we also mention
\cite{CaDaPa-2002}, where a more general case of \eqref{eq-1.5}, adding the friction term $c u'$, has been considered
(see \cite[\S~1]{PaZa-2004} for a brief historical survey about this subject). A typical feature of these results
lies on the fact that the solutions which have been obtained therein have a large number of zeros in the intervals where $w(t)>0$.
This fact was explicitly observed also by Butler in the proof of \cite[Corollary]{Bu-1976}, where the author pointed out that the equation
\begin{equation*}
u'' + w(t) |u|^{p-1} u = 0, \quad p > 1,
\end{equation*}
has infinitely many $T$-periodic solutions, assuming that $w(t)$ is a continuous $T$-periodic function with only
isolated zeros and which is somewhere positive. It was also noted that all these
solutions oscillate (have arbitrarily large zeros) if $\int_{0}^{T} w(t) ~\!dt\geq 0$.
Since condition
\begin{equation}\label{eq-1.6}
\int_{0}^{T} w(t) ~\!dt < 0
\end{equation}
implies the existence of non-oscillatory solutions (cf.~\cite{Bu-1976a}), it was raised the question (see \cite[p.~477]{Bu-1976}) whether
\textit{there can exist non-oscillatory periodic solutions} if \eqref{eq-1.6} holds. In the recent paper \cite{FeZa-2015ade}
we have provided a solution to Butler's question, by showing the \textit{existence} of positive (i.e.~non-oscillatory)
$T$-periodic solutions under the average condition \eqref{eq-1.6}. In this context, Theorem~\ref{th-1.1} can be seen
as a further investigation about Butler's problem, in the sense that we give evidence of the fact that if
the weight is negative enough, multiple positive periodic solutions appear (depending on the number
of positive humps in the weight function which are separated by negative ones).

In many applications, the solutions of \eqref{eq-1.4} represent the steady states associated with reaction-diffusion equations.
With this respect, a great deal of importance was given to existence and multiplicity results of positive solutions
in the superlinear indefinite case (cf.~\cite{AlTa-1996,BeCaDoNi-1994,BeCaDoNi-1995}).

In \cite{GaHaZa-2003}, a shooting method was applied to prove the existence of at least three
positive solutions for the two-point boundary value problem associated with
\begin{equation}\label{eq-1.7}
u'' + a_{\mu}(t) u^{p} = 0, \quad p>1,
\end{equation}
when $a_{\mu}(t)$ has two positive humps separated by a negative one and $\mu > 0$ is sufficiently large.
The same multiplicity result has been obtained by Boscaggin in \cite{Bo-2011} for the Neumann problem.
One of the contributions of the present paper is also that of showing the possibility of multiple positive solutions
without the assumption of superlinear growth at infinity, provided that $g_{\infty}$ is large enough.
Figure~\ref{fig-01} shows a possible example in this direction.

\begin{figure}[h!]
\centering
\includegraphics[width=0.45\textwidth]{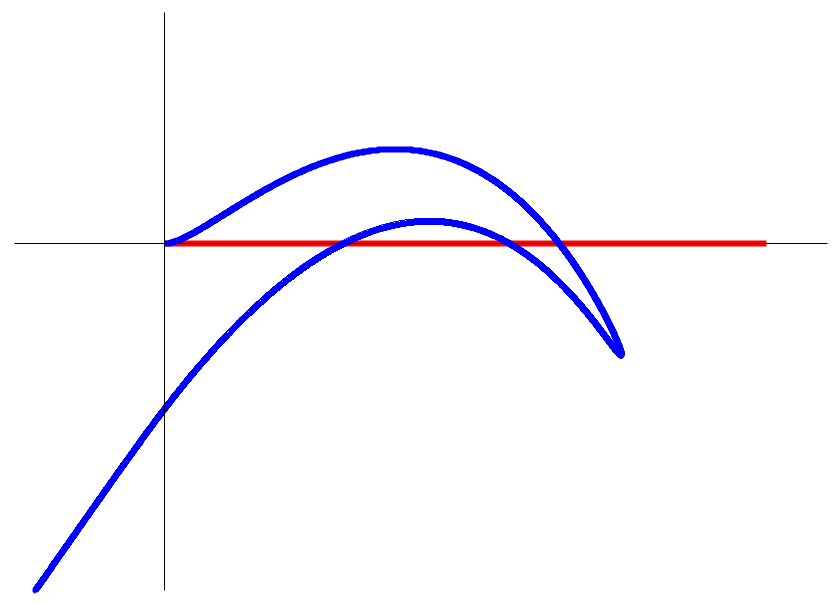}\qquad
\includegraphics[width=0.45\textwidth]{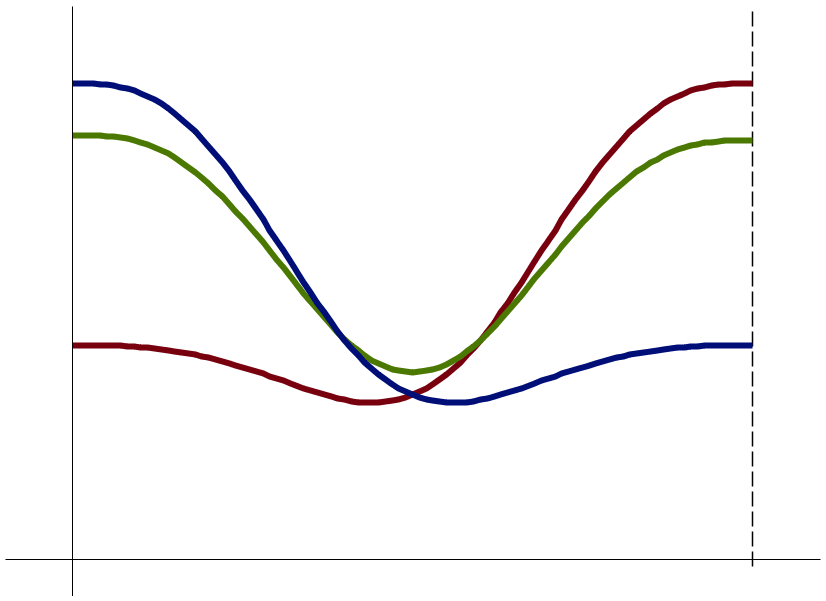}
\caption{\small{An example of three positive solutions for equation $u'' + w(t) g(u) = 0$
with Neumann boundary conditions on $\mathopen{[}0,1\mathclose{]}$.
For this numerical simulation we have chosen $w(t) = \sin^{+}(3\pi t) - 7\,\sin^{-}(3\pi t)$
and $g(s) = \max\{0,100\, s\arctan|s|\}$.
On the left we have shown the image of the segment $\mathopen{[}0,0.2\mathclose{]}\times\{0\}$
through the Poincar\'{e} map in the phase-plane $(u,u')$.
It intersects the positive part of the $u$-axis in three points.
On the right hand side of the figure, we see the graphs of the three positive solutions.}}
\label{fig-01}
\end{figure}

This kind of results lies on a line of research initiated by G\'{o}mez-Re\~{n}asco and L\'{o}pez-G\'{o}mez in \cite{GoReLoGo-2000},
where the authors gave evidence of the fact that (for the Dirichlet problem) at least $2^{m}-1$ positive solutions (for $\mu$ large) appear when
$a_{\mu}(t)$ has $m$ positive humps separated by $m-1$ negative ones.
For the Dirichlet problem, contributions in this direction have been achieved in \cite{BoGoHa-2005, FeZa-2015jde, GaHaZa-2004, GiGo-2009}.
At the best of our knowledge, the first work addressing the same questions in the periodic setting is due to
Barutello, Boscaggin and Verzini, who in \cite{BaBoVe-jde2015}, using a variational approach,
achieved multiplicity of positive periodic solutions for \eqref{eq-1.7}.
In \cite{BaBoVe-jde2015} globally bounded solutions defined on the real line and exhibiting a chaotic behavior were also produced.

\medskip

The plan of the paper is as follows.
In Section~\ref{section-2} we list the hypotheses on $a(t)$ and on $g(s)$ that we assume for the rest of the paper and we introduce an useful notation.
Section~\ref{section-3} is devoted to the application of coincidence degree theory to our problem. More in detail, we define an equivalent operator problem
and we present three technical lemmas essential for the computation of the degree in the proof of our main result (Theorem~\ref{MainTheorem}),
which is stated and proved in Section~\ref{section-4}.
In Section~\ref{section-5} we present various consequences and applications of the main theorem, as Theorem~\ref{th-1.1} and a nonexistence result (cf.~Corollary~\ref{cor-5.5}).
In Section~\ref{section-6} we deal with subharmonic solutions. In Theorem~\ref{th-6.1} we prove the existence of infinitely many subharmonic
solutions for \eqref{eq-1.1} if the negative part of the weight is sufficiently strong (i.e.~when $\mu > 0$ is large enough).
This result follows from Theorem~\ref{th-1.1} applied to an interval of the form $\mathopen{[}0,kT\mathclose{]}$ and a careful verification that
the constants needed for the proof are independent on $k$. In the same section we also discuss the number of
subharmonics of a given order and we sketch how to produce bounded solutions on the real line which are not necessarily periodic.
Even if we focus our main attention to the study of the periodic problem, in Section~\ref{section-7} we observe that
variants of our main results can be given for the Neumann problem.
Therein we also provide an application to radially symmetric solutions of PDEs on annular domains.
Finally, in Appendix~\ref{appendix-A} we discuss some basic facts about the coincidence degree
defined in open and possibly unbounded sets and we state some lemmas for the computation of the degree,
while in Appendix~\ref{appendix-B} we present a combinatorial lemma which is of crucial importance
both for the computation of the coincidence degree
and for our main multiplicity results.

\section{Setting and notation}\label{section-2}

In this section we present the main elements involved in the study of the positive $T$-periodic solutions of the equation
\begin{equation}\label{eq-2.1}
u'' + c u' + \bigl{(} a^{+}(t) - \mu a^{-}(t)\bigr{)} g(u) = 0.
\end{equation}
For $\mu >0$, we set
\begin{equation*}
a_{\mu}(t) := a^{+}(t) - \mu \, a^{-}(t).
\end{equation*}
The hypotheses that will follow will be assumed from now on in the paper.

Let $g\colon{\mathbb{R}}^{+}\to{\mathbb{R}}^{+}$ be a continuous function such that
\begin{equation*}
g(0) = 0, \qquad g(s) > 0 \quad \text{for } \; s > 0.
\leqno{\hspace*{2.2pt}(g_{*})}
\end{equation*}
Suppose also that
\begin{equation*}
g_{0} := \limsup_{s\to 0^{+}} \dfrac{g(s)}{s} < +\infty
\quad \text{ and } \quad
g_{\infty} := \liminf_{s\to +\infty} \dfrac{g(s)}{s} > 0.
\leqno{\hspace*{2.2pt}(g_{1})}
\end{equation*}

The weight coefficient $a\colon{\mathbb{R}}\to{\mathbb{R}}$ is a locally integrable $T$-periodic function such that,
in a time-interval of length $T$, there exists a finite number of closed pairwise disjoint intervals where $a(t) \succ 0$,
separated by closed intervals where $a(t) \prec 0$.
In this case, thanks to the periodicity of $a(t)$, we can suitably choose an interval $\mathopen{[}t_{0},t_{0} + T\mathclose{]}$,
which we identify with $\mathopen{[}0,T\mathclose{]}$ for notational convenience, such that the following condition $(a_{*})$ holds.

\begin{itemize}
\item[$(a_{*})$]
\textit{There exist $m \geq 2$ closed and pairwise disjoint intervals $I^{+}_{1},\ldots,I^{+}_{m}$
separated by $m$ closed intervals $I^{-}_{1},\ldots,I^{-}_{m}$ such that
\begin{equation*}
a(t)\succ 0 \; \text{ on } I^{+}_{i}, \qquad a(t)\prec 0 \; \text{ on } I^{-}_{i},
\end{equation*}
and, moreover,
\begin{equation*}
\bigcup_{i=1}^{m} I^{+}_{i} \, \cup \, \bigcup_{i=1}^{m} I^{-}_{i} = \mathopen{[}0,T\mathclose{]}.
\end{equation*}}
\end{itemize}

To explain this fact with an example, suppose that we take $a(t) = \cos (2t)$ as a $2\pi$-periodic function.
In this case, on $\mathopen{[}0,2\pi\mathclose{]}$ we have three positive humps and two negative ones.
However, in order to enter in the setting of $(a_{*})$ and hence to look at the weight as a function
with two positive humps separated by two negative ones in a time-interval of length $2\pi$,
we can choose $\mathopen{[}t_{0},t_{0}+2\pi\mathclose{]}$, for $t_{0} =3\pi/4$, as interval of periodicity.
When, for convenience in the exposition, we say that we work with the standard period interval
$\mathopen{[}0,2\pi\mathclose{]}$, we are in fact considering a shift of $t_{0}$
of the weight function, e.g.~taking $\cos(2t + 2t_{0})$ as effective coefficient.
Clearly, this does not affect our considerations as long as we are interested in $2\pi$-periodic solutions.
In the same example, let us fix an integer $k\geq 2$ and consider the coefficient $\cos(2t + 2t_{0})=\sin(2t)$ as a
$2k\pi$-periodic function. In the period interval $\mathopen{[}0,2k\pi\mathclose{]}$ the weight has $m = 2k$ intervals of positivity
separated by $2k$ intervals of negativity. We will consider again a similar example dealing with subharmonic solutions.

In the sequel, it will be not restrictive to label the intervals $I^{+}_{i}$ and $I^{-}_{i}$ following the natural order given by the
standard orientation of the real line and thus determine $2m + 1$ points
\begin{equation*}
0 = \sigma_{1} < \tau_{1} < \sigma_{2} < \tau_{2} < \ldots < \sigma_{m-1} < \tau_{m-1} < \sigma_{m} < \tau_{m} < \sigma_{m+1} = T,
\end{equation*}
so that, for $i=1,\ldots,m$,
\begin{equation*}
I^{+}_{i} := \mathopen{[}\sigma_{i},\tau_{i}\mathclose{]} \quad \text{ and } \quad
I^{-}_{i} := \mathopen{[}\tau_{i},\sigma_{i+1}\mathclose{]}.
\end{equation*}
Finally, consistently with assumption $(a_{*})$ and without loss of generality, we select the points $\sigma_{i}$
and $\tau_{i}$ in such a manner that $a(t)\not\equiv0$ on all left neighborhoods of
$\sigma_{i}$ (for $i > 1$) and on all right neighborhoods of $\tau_{i}$.
In other words, if there is an interval $J$ contained in $\mathopen{[}0,T\mathclose{]}$ where $a(t)\equiv 0$,
we choose the points $\sigma_{i}$ and $\tau_{i}$ so that $J$ is contained in one of the $I^{+}_{i}$
or $J$ is contained in the interior of one of the $I^{-}_{i}$.

We denote by $\lambda_{1}^{i}$, $i=1,\ldots,m$, the first eigenvalue of the eigenvalue problem in $I^{+}_{i}$
\begin{equation*}
\varphi'' + c \varphi' + \lambda a(t) \varphi = 0, \quad \varphi|_{\partial I^{+}_{i}} = 0.
\end{equation*}
From the assumptions on $a(t)$ in $I^{+}_{i}$, it clearly follows that $\lambda_{1}^{i} > 0$ for each $i=1,\ldots,m$.

We introduce some other useful notations.
Let $\mathcal{I}\subseteq\{1,\ldots,m\}$ be a subset of indices (possibly empty) and let
$d,D$ be two fixed positive real numbers with $d<D$.
We define two families of open unbounded sets
\begin{equation}\label{eq-2.omega}
\begin{aligned}
\Omega^{\mathcal{I}}_{d,D}:=
\biggl\{\,u\in\mathcal{C}(\mathopen{[}0,T\mathclose{]})\colon
   & \max_{t\in I^{+}_{i}}|u(t)|<D, \, i\in\mathcal{I};                             &
\\ & \max_{t\in I^{+}_{i}}|u(t)|<d, \, i\in\{1,\ldots,m\}\setminus\mathcal{I}       & \biggr\}
\end{aligned}
\end{equation}
and
\begin{equation}\label{eq-2.lambda}
\begin{aligned}
\Lambda^{\mathcal{I}}_{d,D}:=
\biggl\{\,u\in\mathcal{C}(\mathopen{[}0,T\mathclose{]})\colon
   & d < \max_{t\in I^{+}_{i}}|u(t)|<D, \, i\in\mathcal{I};                           &
\\ & \max_{t\in I^{+}_{i}}|u(t)|<d, \, i\in\{1,\ldots,m\}\setminus\mathcal{I}     & \biggr\}.
\end{aligned}
\end{equation}
We note that, for each $\mathcal{I}\subseteq \{1,\ldots,m\}$, we have
\begin{equation*}
\Omega^{\mathcal{I}}_{d,D}=\bigcup_{\mathcal{J}\subseteq\mathcal{I}}\Lambda^{\mathcal{J}}_{d,D},
\end{equation*}
and the union is disjoint, since $\Lambda^{\mathcal{J}'}_{d,D} \cap \Lambda^{\mathcal{J}''}_{d,D} = \emptyset$, for $\mathcal{J}'\neq\mathcal{J}''$.

In the sequel, once the constants $d$ and $D$ are fixed, we simply use the symbols $\Omega^{\mathcal{I}}$ and $\Lambda^{\mathcal{I}}$
to denote $\Omega^{\mathcal{I}}_{d,D}$ and $\Lambda^{\mathcal{I}}_{d,D}$, respectively.

\section{The abstract setting of the coincidence degree}\label{section-3}

In this section we apply the coincidence degree theory, that we will present in Appendix~\ref{appendix-A}, to study the periodic problem
associated with equation \eqref{eq-2.1}. We follow the same approach presented in \cite{Ma-1979}.

Let $X:= \mathcal{C}(\mathopen{[}0,T\mathclose{]})$ be the Banach space of continuous functions $u \colon \mathopen{[}0,T\mathclose{]} \to \mathbb{R}$,
endowed with the $\sup$-norm
\begin{equation*}
\|u\|_{\infty} := \max_{t\in \mathopen{[}0,T\mathclose{]}} |u(t)|,
\end{equation*}
and let $Z:=L^{1}(\mathopen{[}0,T\mathclose{]})$ be the space of integrable functions $w \colon \mathopen{[}0,T\mathclose{]} \to \mathbb{R}$,
endowed with the norm
\begin{equation*}
\|w\|_{L^{1}}:= \int_{0}^{T} |w(t)|~\!dt.
\end{equation*}
We consider the linear differential operator $L\colon \text{\rm dom}\,L \to Z$ defined as
\begin{equation*}
(Lu)(t):= - u''(t) - cu'(t), \quad t\in\mathopen{[}0,T\mathclose{]},
\end{equation*}
where $\text{\rm dom}\,L$ is determined by the functions of $X$
which are continuously differentiable with absolutely continuous derivative
and satisfying the periodic boundary condition
\begin{equation}\label{per-cond}
u(0) = u(T), \qquad u'(0) = u'(T).
\end{equation}
Therefore, $L$ is a Fredholm map of index zero, $\ker L$ and $\text{\rm coker}\,L$ are made by the constant functions and
\begin{equation*}
\text{\rm Im}\,L = \biggl\{ w\in Z \colon \int_{0}^{T} w(t)~\!dt = 0 \biggr\}.
\end{equation*}
As projectors $P \colon X \to \ker L$ and $Q \colon Z \to \text{\rm coker}\,L$ associated with $L$ we choose the average operators
\begin{equation*}
Pu = Qu := \dfrac{1}{T}\int_{0}^{T} u(t)~\!dt.
\end{equation*}
Notice that $\ker P$ is given by the continuous functions with mean value zero.
Finally, let $K_{P} \colon \text{\rm Im}\,L \to \text{\rm dom}\,L \cap \ker P$ be the right inverse of $L$,
which is the operator that at any function $w\in Z$ with $\int_{0}^{T} w(t)~\!dt =0$ associates the unique solution $u$ of
\begin{equation*}
u'' + c u' + w(t) =0, \quad \text{ with } \; \int_{0}^{T} u(t)~\!dt = 0,
\end{equation*}
and satisfying the boundary condition \eqref{per-cond}.

Thereafter, on $\mathbb{R}^{2}$ we define the $L^{1}$-Carath\'{e}odory function
\begin{equation*}
\tilde{f}(t,s) :=
\begin{cases}
\, a_{\mu}(t)g(s), & \text{if } s \geq 0;\\
\, -s,  & \text{if } s \leq 0;
\end{cases}
\end{equation*}
and observe that $\tilde{f}(t + T,s) = \tilde{f}(t,s)$ for a.e.~$t\in {\mathbb{R}}$ and for all $s\in\mathbb{R}$.
Let $N \colon X \to Z$ be the Nemytskii operator induced by $\tilde{f}$, that is
\begin{equation*}
(N u)(t):= \tilde{f}(t,u(t)), \quad t\in\mathopen{[}0,T\mathclose{]}.
\end{equation*}
According to the above positions, if $u$ is a $T$-periodic solution of
\begin{equation}\label{eq-3.tilde}
u'' + c u' + \tilde{f}(t,u) = 0,
\end{equation}
then $u|_{\mathopen{[}0,T\mathclose{]}}$ is a solution of the coincidence equation
\begin{equation}\label{eq-coincidence-eq}
Lu = Nu,\quad u\in \text{\rm dom}\,L.
\end{equation}
Conversely, any solution $u$ of \eqref{eq-coincidence-eq} can be extended by $T$-periodicity to a $T$-periodic solution
of \eqref{eq-3.tilde}. Moreover, from the definition of $\tilde{f}$ and conditions $(g_{*})$ and $(g_{1})$,
one can easily verify by a maximum principle argument (cf.~\cite[Lemma~6.1]{FeZa-2015ade})
that if $u\not\equiv0$ is a solution of \eqref{eq-coincidence-eq}, then $u(t)$
is strictly positive and hence a positive $T$-periodic solution of \eqref{eq-2.1} (once extended by $T$-periodicity to the whole real line).

As remarked in Appendix~\ref{appendix-A}, the operator equation \eqref{eq-coincidence-eq} is equivalent to the fixed point problem
\begin{equation*}
u = \Phi u:= Pu + QN u + K_{P}(Id-Q)N u, \quad u \in X,
\end{equation*}
where we have chosen the identity on $\mathbb{R}$ as linear orientation-preserving
isomorphism $J$ from $\text{\rm coker}\,L$ to $\ker L$ (both identified with $\mathbb{R}$).

\medskip

Now we are interested in computing the coincidence degree of $L$ and $N$ in some open domains.
For this purpose, we will consider some modifications of \eqref{eq-3.tilde} which correspond to operator equations of the
form \eqref{eq-coincidence-eq} for the associated Nemytskii operators $N$.
In the sequel we will also identify the $T$-periodic solutions with solutions defined on $\mathopen{[}0,T\mathclose{]}$ and satisfying the boundary condition
\eqref{per-cond}. We also denote by $L^{1}_{T}$ the space of locally integrable and $T$-periodic functions $w \colon {\mathbb{R}} \to {\mathbb{R}}$
(which can be identified with $Z$).

The subsequent two lemmas are direct applications of Lemma~\ref{lemma_Mawhin} and Lemma~\ref{lem-abs-deg0}, respectively.
They give conditions for the computation of the degree on some open balls.
The standard proofs are omitted (see \cite[Lemma~2.2]{BoFeZa-pp2015} and \cite[Theorem~2.1]{FeZa-2015ade} for the details).

\begin{lemma}\label{lem-deg1-ball}
Let $\mu > 0$ be such that $\int_{0}^{T} a_{\mu}(t)~\!dt < 0$.
Assume that there exists a constant $d > 0$ such that the following property holds.
\begin{itemize}
\item[$(H_{d})$]
If $\vartheta\in \mathopen{]}0,1\mathclose{]}$ and $u(t)$ is any non-negative $T$-periodic solution of
\begin{equation}\label{BVP-3.4}
u'' + c u' + \vartheta a_{\mu}(t) g(u) = 0, \\
\end{equation}
then $\,\max_{t\in \mathopen{[}0,T\mathclose{]}} u(t) \neq d$.
\end{itemize}
Then
\begin{equation*}
D_{L}(L-N,B(0,d)) = 1.
\end{equation*}
\end{lemma}

\begin{lemma}\label{lem-deg0-ball}
Assume that there exists a constant $D > 0$ such that the following property holds.
\begin{itemize}
\item [$(H^{D})$]
There exist a non-negative function $v\in L_{T}^{1}$ with $v\not\equiv 0$
and a constant $\alpha_{0} > 0$, such that every $T$-periodic solution $u(t)\geq0$ of the boundary value problem
\begin{equation}\label{BVP-3.5}
u'' + c u' + a_{\mu}(t) g(u) + \alpha v(t) = 0,
\end{equation}
for $\alpha \in \mathopen{[}0,\alpha_{0}\mathclose{]}$, satisfies $\|u\|_{\infty} \neq D$.
Moreover, there are no solutions $u(t)$ of \eqref{BVP-3.5} for $\alpha = \alpha_{0}$ with $0 \leq u(t) \leq D$,
for all $t\in \mathbb{R}$.
\end{itemize}
Then
\begin{equation*}
D_{L}(L-N,B(0,D)) = 0.
\end{equation*}
\end{lemma}

\medskip

In order to achieve our multiplicity result, in Section~\ref{section-4} we will fix $d,D>0$ satisfying $(H_{d})$ and $(H^{D})$, respectively,
and compute the coincidence degree in the open and unbounded sets $\Lambda^{\mathcal{I}}_{d,D}$, for $\mathcal{I}\subseteq\{1,\ldots,m\}$.
To this aim the following lemma is of utmost importance (see \cite[Lemma~2.1]{BoFeZa-pp2015} for a similar statement). In the next result
we consider again equation \eqref{BVP-3.5} of the previous lemma.

\begin{lemma}\label{lem-deg0}
Let $\mathcal{J}\subseteq\{1,\ldots,m\}$ be a nonempty subset of indices, let $d > 0$ be a constant
and $v\in L_{T}^{1}$ a non-negative nontrivial function, such that the following properties hold.
\begin{itemize}
\item[$(A_{d,\mathcal{J}})$]
If $\alpha \geq 0$, then any non-negative $T$-periodic solution $u(t)$ of \eqref{BVP-3.5}
satisfies $\,\max_{t\in I^{+}_{j}} u(t) \neq d$, for all $j\in\mathcal{J}$.
\item[$(B_{d,\mathcal{J}})$]
For every $\beta \geq 0$ there exists a constant $D_{\beta} > d$ such that if $\alpha\in \mathopen{[}0,\beta\mathclose{]}$ and
$u(t)$ is any non-negative $T$-periodic solution of \eqref{BVP-3.5} with $\,\max_{t\in I^{+}_{j}} u(t) \leq d$, for all $j\in\mathcal{J}$,
then $\,\max_{t\in \mathopen{[}0,T\mathclose{]}} u(t) \leq D_{\beta}$.
\item[$(C_{d,\mathcal{J}})$]
There exists $\alpha_{0} > 0$ such that equation \eqref{BVP-3.5}, with $\alpha=\alpha_{0}$, does not possess any
non-negative $T$-periodic solution $u(t)$ with $\,\max_{t\in I^{+}_{j}} u(t) \leq d$, for all $j\in\mathcal{J}$.
\end{itemize}
Then
\begin{equation*}
D_{L}(L-N,\Gamma_{d,\mathcal{J}}) = 0,
\end{equation*}
where
\begin{equation}\label{eq-3.gamma}
\Gamma_{d,\mathcal{J}}:=
\biggl\{\,u\in\mathcal{C}(\mathopen{[}0,T\mathclose{]}) \colon \max_{t\in I^{+}_{j}}|u(t)| < d, \, j\in\mathcal{J} \biggr\}.
\end{equation}
\end{lemma}

\begin{proof}
According to the setting presented in this section, conditions $(A_{d,\mathcal{J}})$, $(B_{d,\mathcal{J}})$ and $(C_{d,\mathcal{J}})$
are equivalent to conditions $(a)$, $(b)$ and $(c)$ of Lemma~\ref{lem-deg0-deFigueiredo} with respect to the
open set $\Omega: = \Gamma_{d,\mathcal{J}}$. Therefore, the thesis of Lemma~\ref{lem-deg0} follows.
\end{proof}

\begin{remark}\label{rem-3.1}
From a theoretical point of view, the choice of the set of indices $\mathcal{J}$ with $\emptyset\neq\mathcal{J}\subseteq\{1,\ldots,m\}$
is arbitrary. However, as we will see in the next section,
in the actual applications of Lemma~\ref{lem-deg0} we shall take $\mathcal{J}\subsetneq\{1,\ldots,m\}$
because, in our setting, the case $\mathcal{J}=\{1,\ldots,m\}$ will be discussed in the frame of Lemma~\ref{lem-deg1-ball}.
$\hfill\lhd$
\end{remark}

\section{The main multiplicity result}\label{section-4}

In this section we use all the tools just presented in the previous sections to prove the following main result.

\begin{theorem}\label{MainTheorem}
Let $g \colon {\mathbb{R}}^{+} \to {\mathbb{R}}^{+}$ be a continuous function satisfying $(g_{*})$,
\begin{equation*}
g_{0} = 0 \quad \text{ and } \quad g_{\infty} > \max_{i=1,\ldots,m} \lambda_{1}^{i}.
\end{equation*}
Let $a \colon \mathbb{R} \to \mathbb{R}$ be a $T$-periodic locally integrable function satisfying $(a_{*})$.
Then there exists $\mu^{*}>0$ such that for all $\mu>\mu^{*}$ equation \eqref{eq-2.1} has at least $2^{m}-1$ positive $T$-periodic solutions.
\end{theorem}

\begin{remark}\label{rem-4.1}
The $2^{m}-1$ positive $T$-periodic solutions are obtained as follows. Along the proof we provide two constants $0 < r < R$
(with $r$ small and $R$ large) such that if $\mu>\mu^{*}$, given any nonempty set of indices $\mathcal{I} \subseteq \{1,\ldots,m\}$,
there exists at least one positive $T$-periodic solution $u_{\mathcal{I}} \in \Lambda^{\mathcal{I}}_{r,R}$.
Namely, $u_{\mathcal{I}}(t)$ is
small for all $t\in I^{+}_{i}$ when $i\notin \mathcal{I}$, and, on the other hand, $r < u_{\mathcal{I}}(t) < R$ for some $t\in I^{+}_{i}$ when $i\in \mathcal{I}$.
We will also prove that, when $\mu$ is sufficiently large, all these solutions are small in the $I^{-}_{i}$ intervals (see Section~\ref{section-4.5}).
$\hfill\lhd$
\end{remark}

\begin{remark}\label{rem-4.2}
The assumption $g_{0}=0$ in Theorem~\ref{MainTheorem} can be slightly improved to a condition of the form
\begin{equation*}
g_{0} < \lambda_{*},
\end{equation*}
where $\lambda_{*}$ is a positive constant which satisfies
\begin{equation*}
0 < \lambda_{*} < \min_{i=1,\ldots,m} \lambda_{1}^{i}.
\end{equation*}
A lower bound for $\lambda_{*}$ (although not sharp) is explicitly given by the constant $1/K_{0}$ provided in \eqref{cond-etar}
in Section~\ref{section-4.2} (see also Remark~\ref{rem-4.5} for more details).
When $c=0$ it is easy to check that $1/K_{0}$ is strictly less than
$4/\bigl{(}|I^{+}_{i}|\int_{I^{+}_{i}} a^{+}(t)~\!dt\bigr{)}$ (for all $i=1,\ldots,m$),
which are the constants corresponding to the application of Lyapunov
inequality to each of the intervals of positivity (cf.~\cite[ch.~XI]{Ha-1982}).
$\hfill\lhd$
\end{remark}

\subsection{General strategy and proof of Theorem~\ref{MainTheorem}}\label{section-4.1}

In this section, we describe the main steps that define the proof of Theorem~\ref{MainTheorem}.
The details can be found in the following three sections.

\medskip

First of all, in Section~\ref{section-4.2}, from $g_{0}=0$ we fix a (small) constant
$r>0$ such that
\begin{equation}\label{eq-etar}
\eta(r) := \sup_{0 < s \leq r} \dfrac{g(s)}{s}
\end{equation}
is sufficiently small (cf.~condition \eqref{cond-etar}). For this fixed $r$, we determine a constant $\mu_{r}$, with
\begin{equation}\label{eq-mud}
\mu_{r} > \mu^{\#} := \dfrac{\int_{0}^{T}a^{+}(t)~\!dt}{\int_{0}^{T}a^{-}(t)~\!dt},
\end{equation}
such that condition $(H_{r})$ of Lemma~\ref{lem-deg1-ball} is satisfied \textit{for every} $\mu \geq \mu_{r}$ and therefore
\begin{equation}\label{eq-degr}
D_{L}(L-N,B(0,r)) = 1.
\end{equation}
It is important to notice that, for the validity of \eqref{eq-degr}, it is necessary to take $\mu > \mu^{\#}$
in order to have $\int_{0}^{T} a_{\mu}(t)~\!dt < 0$.

\medskip

As a second step, in Section~\ref{section-4.3}, we show that there exists a constant $R^{*}$, with $0<r<R^{*}$, such that,
for any nontrivial function $v\in L^{1}(\mathopen{[}0,T\mathclose{]})$ satisfying
\begin{equation}\label{eq-v}
v(t)\geq 0 \quad \text{on } \bigcup_{i=1}^{m} I^{+}_{i}, \qquad v(t) = 0 \quad \text{on } \bigcup_{i=1}^{m} I^{-}_{i},
\end{equation}
and for all $\alpha \geq 0$, it holds that any non-negative solution $u(t)$ of \eqref{BVP-3.5} is bounded by $R^{*}$, namely
\begin{equation}\label{upperbound}
\max_{t\in\mathopen{[}0,T\mathclose{]}}u(t)<R^{*}.
\end{equation}
This result is proved using the lower bound of $g_{\infty}$ and the constant $R^{*}$
can be chosen independently on the functions $v(t)$ satisfying \eqref{eq-v}.

In this manner (for $\alpha=0$) we obtain also a priori bound for all non-negative $T$-periodic solutions of \eqref{eq-2.1}.
Then, we verify that condition $(H^{R})$ of Lemma~\ref{lem-deg0-ball} is satisfied for all $R \geq R^{*}$. Hence, we have
\begin{equation}\label{eq-degR}
D_{L}(L-N,B(0,R)) = 0, \quad \forall \, R \geq R^{*}.
\end{equation}
It is important to notice that, in order to prove \eqref{upperbound} and consequently \eqref{eq-degR},
we only use information about $a^{+}(t)$. Hence $R^{*}$ can be chosen independently on $\mu > 0$.

\begin{remark}\label{rem-4.3}
Using the additivity property of the coincidence degree, from \eqref{eq-degr} and \eqref{eq-degR}, we reach the following equality
\begin{equation*}
D_{L}(L-N,B(0,R^{*})\setminus B[0,r]) = -1.
\end{equation*}
Then, we obtain the existence of at least a nontrivial solution $u$ of \eqref{eq-coincidence-eq},
provided that $\mu>\mu_{r}$. Using a standard maximum principle argument,
it is easy to prove that $u$ is a positive $T$-periodic solution of \eqref{eq-2.1}
(cf.~\cite[Theorem~3.1 and Theorem~3.2]{FeZa-2015ade} and see also Remark~\ref{rem-4.6}).
$\hfill\lhd$
\end{remark}

\medskip

At this point, we fix a constant $R$ with
\begin{equation*}
0 < r < R^{*} \leq R
\end{equation*}
and, for all sets of indices $\mathcal{I}\subseteq\{1,\ldots,m\}$, we consider the open and unbounded sets
\begin{equation*}
\Omega^{\mathcal{I}}:=\Omega^{\mathcal{I}}_{r,R} \quad \text{ and } \quad \Lambda^{\mathcal{I}}:=\Lambda^{\mathcal{I}}_{r,R}
\end{equation*}
introduced in \eqref{eq-2.omega} and in \eqref{eq-2.lambda}, respectively.

As a third step, we will prove that
\begin{equation}\label{eq-DL}
D_{L}(L-N,\Lambda^{\mathcal{I}}) \neq 0, \quad \text{ for all } \, \mathcal{I}\subseteq\{1,\ldots,m\}.
\end{equation}

Before the proof of \eqref{eq-DL}, we make the following observation which plays a crucial role in
various subsequent steps.

\begin{remark}\label{rem-4.4}
Writing equation \eqref{eq-2.1} as
\begin{equation*}
\bigl{(}e^{ct} u'\bigr{)}' + e^{ct} a_{\mu}(t)g(u) = 0,
\end{equation*}
we find that $(e^{ct} u'(t))' \leq 0$ for almost every $t\in I^{+}_{i}$ and $(e^{ct} u'(t))' \geq 0$ for almost every $t\in I^{-}_{i}$
(where $u(t)\geq 0$ is any solution). Then, the map
\begin{equation*}
t \mapsto e^{ct} u'(t)
\end{equation*}
is non-increasing on each $I^{+}_{i}$ and non-decreasing on each $I^{-}_{i}$.
This property replaces the convexity of $u(t)$ on $I^{-}_{i}$, which is an obvious fact when $c=0$. For an arbitrary $c\in \mathbb{R}$ we can still
preserve some convexity type properties. In particular, for all $i=1,\ldots,m$ we have that
\begin{equation}\label{eq-max}
\max_{t\in I^{-}_{i}} u(t) = \max_{t\in \partial I^{-}_{i}} u(t) = \max \bigl{\{}u(\tau_{i}),u(\sigma_{i+1})\bigr{\}},
\end{equation}
which is nothing but a one-dimensional form of a maximum principle for the differential operator $L$.
We verify now this fact since this property, although elementary, will be used several times in the sequel.
Indeed, observe that if $u'(t^{*}) \geq 0$, for some $t^{*} \in \mathopen{[}\tau_{i},\sigma_{i+1}\mathclose{[}$,
then $u'(t) \geq 0$ for all $t\in \mathopen{[}t^{*},\sigma_{i+1}\mathclose{]}$,
hence $u(t^{*}) \leq u(\sigma_{i+1})$.
Similarly, if $u'(t^{*}) \leq 0$, for some $t^{*} \in \mathopen{]}\tau_{i},\sigma_{i+1}\mathclose{]}$, then $u'(t) \leq 0$
for all $t\in \mathopen{[}\tau_{i},t^{*}\mathclose{]}$, hence $u(t^{*}) \leq u(\tau_{i})$.
From these remarks, \eqref{eq-max} follows immediately.
$\hfill\lhd$
\end{remark}

In order to prove \eqref{eq-DL}, first of all we consider $\mathcal{I}=\emptyset$. Accordingly, we have that
\begin{equation}\label{deg-emptyset}
D_{L}(L-N,\Omega^{\emptyset}) = D_{L}(L-N,\Lambda^{\emptyset}) = D_{L}(L-N,B(0,r)) = 1.
\end{equation}
The first identity in \eqref{deg-emptyset} is trivial from the definitions of the sets, since $\Omega^{\emptyset} = \Lambda^{\emptyset}$.
It is also obvious that $B(0,r) \subseteq \Omega^{\emptyset}$. Conversely, let $u$ be a $T$-periodic solution
of \eqref{eq-3.tilde} belonging to $\Omega^{\emptyset}$. By the maximum principle, we know that $u$
is a (non-negative) $T$-periodic solution of \eqref{eq-2.1}. Moreover, $u(t) < r$ for all $t\in I^{+}_{i}$, $i=1,\ldots,m$. Then, from
\eqref{eq-max} we have that $u(t) < r$ for all $t\in \mathopen{[}0,T\mathclose{]}$.
(In the application of formula \eqref{eq-max} we have considered
the interval $I^{-}_{m}$, as an interval between $I^{+}_{m}$ and $I^{+}_{1} + T$, by virtue of the $T$-periodicity of the solution.)
Finally, by the excision property of the coincidence degree and \eqref{eq-degr}, formula \eqref{deg-emptyset} follows.

Next, we consider a nonempty subset of indices $\mathcal{I}\subsetneq\{1,\ldots,m\}$.
In Section~\ref{section-4.4}, choosing $d=r$, $\mathcal{J}:=\{1,\ldots,m\}\setminus\mathcal{I} \neq \emptyset$
and a nontrivial function $v\in L^{1}(\mathopen{[}0,T\mathclose{]})$ such that
\begin{equation}\label{eq-v-I}
v(t) \succ 0 \quad \text{on } \bigcup_{i\in\mathcal{I}} I^{+}_{i}, \qquad v(t) = 0 \quad \text{otherwise},
\end{equation}
we verify that the three conditions of Lemma~\ref{lem-deg0} hold, for $\mu$ sufficiently large.
More in detail, we provide a lower bound $\mu^{*}_{\mathcal{I}} > 0$,
with $\mu^{*}_{\mathcal{I}}$ independent on $\alpha$, such that condition $(A_{r,\mathcal{J}})$ is satisfied for all $\mu > \mu^{*}_{\mathcal{I}}$.
Then, we fix an arbitrary $\mu > \mu^{*}_{\mathcal{I}}$ and show that conditions $(B_{r,\mathcal{J}})$ and $(C_{r,\mathcal{J}})$ are satisfied as well.

Since $R$ is an upper bound for all the solutions of \eqref{BVP-3.5} (cf.~\eqref{upperbound}),
comparing the definitions \eqref{eq-2.omega} and \eqref{eq-3.gamma}, we see that $u \in \Omega^{\mathcal{I}}$ if and only if
$u \in \Gamma_{r,\mathcal{J}}$, for each solution $u$. Hence,
applying the excision property of the coincidence degree and Lemma~\ref{lem-deg0}, we obtain
\begin{equation}\label{deg-Ja}
D_{L}(L-N,\Omega^{\mathcal{I}}) = D_{L}(L-N,\Gamma_{r,\mathcal{J}}) = 0,
\quad \text{ for all } \, \emptyset\neq\mathcal{I}\subsetneq\{1,\ldots,m\}.
\end{equation}

Using again \eqref{eq-max} in Remark~\ref{rem-4.4} and arguing as above for $r$,
we can check that $R$ is an a priori bound for the solutions on the whole domain.
In this manner, by \eqref{eq-degR}, if $\mathcal{I}=\{1,\ldots,m\}$ we obtain
\begin{equation*}
D_{L}(L-N,\Omega^{\mathcal{I}}) = D_{L}(L-N,B(0,R)) = 0.
\end{equation*}

In conclusion, putting together this latter relation with \eqref{deg-Ja}, we find that
\begin{equation}\label{deg-J}
D_{L}(L-N,\Omega^{\mathcal{I}}) = 0,
\quad \text{ for all } \, \emptyset\neq\mathcal{I}\subseteq\{1,\ldots,m\}.
\end{equation}

\medskip

Finally, we define
\begin{equation*}
\mu^{*} := \mu_{r} \vee \max \bigl{\{} \mu^{*}_{\mathcal{I}} \colon \emptyset \neq \mathcal{I} \subseteq \{1,\ldots,m\} \bigr{\}},
\end{equation*}
where, as usual, ``$\vee$'' denotes the maximum between two numbers. As a byproduct of the proof of
$(A_{{\mathcal J},r})$ in Section~\ref{section-4.4} (for $\alpha = 0$) we also have that for each $\mu > \mu^{*}$
the degree $D_{L}(L-N,\Lambda^{\mathcal{I}})$ is well defined for all $\mathcal{I}\subseteq\{1,\ldots,m\}$
(technically, the matter is to observe that for $\mu$ sufficiently large the are no $T$-periodic solutions touching the
level $r$ on some intervals $I^{+}_{i}$). At this point, following the same inductive argument
as in \cite[Lemma~4.1]{FeZa-2015jde} and using \eqref{deg-emptyset} and \eqref{deg-J},
it is possible to prove that
\begin{equation}\label{eq-deg-Lambda}
D_{L}(L-N,\Lambda^{\mathcal{I}})=(-1)^{\#\mathcal{I}}, \quad \text{ for all } \, \mathcal{I}\subseteq\{1,\ldots,m\},
\end{equation}
holds for each $\mu > \mu^{*}$. In this manner, \eqref{eq-DL} is verified.
Since formula \eqref{eq-deg-Lambda} is crucial to prove our multiplicity result, we give the details of the proof in Appendix~\ref{appendix-B}.

In conclusion, since the coincidence degree is nonzero in each $\Lambda^{\mathcal{I}}$, there exists a solution
$u\in\Lambda^{\mathcal{I}}$ of \eqref{eq-coincidence-eq}, for all $\mathcal{I}\subseteq\{1,\ldots,m\}$.
Notice that $0\notin\Lambda^{\mathcal{I}}$ for all $\emptyset\neq\mathcal{I}\subseteq\{1,\ldots,m\}$.
As remarked in Section~\ref{section-3}, by a maximum principle argument, for $\mathcal{I}\neq\emptyset$ the solution $u\in\Lambda^{\mathcal{I}}$
of \eqref{eq-coincidence-eq} is a positive $T$-periodic solution of \eqref{eq-2.1}. Moreover, by \eqref{eq-max}, we also deduce that $\|u\|_{\infty} < R$.
At this moment, we can summarize what we have proved as follows.
\begin{quote}
\textit{For each nonempty set of indices $\mathcal{I}\subseteq\{1,\ldots,m\}$, there exists at least one $T$-periodic solution
$u_{\mathcal{I}}$ of \eqref{eq-2.1} with $u_{\mathcal{I}}\in \Lambda^{\mathcal{I}}$ and such that
$0 < u_{\mathcal{I}}(t) < R$ for all $t\in\mathbb{R}$.}
\end{quote}
Finally, since the number of nonempty subsets of a set with $m$ elements is $2^{m}-1$ and the sets $\Lambda^{\mathcal{I}}$
are pairwise disjoint, we conclude that there are at least $2^{m}-1$ positive $T$-periodic solutions of \eqref{eq-2.1}.
The thesis of Theorem~\ref{MainTheorem} follows.
\qed

\medskip

Having already outlined the scheme of the proof, we provide now all the missing technical details.

\subsection{Proof of $(H_{r})$ for $r$ small}\label{section-4.2}

In this section we find a sufficiently small real number $r > 0$ such that $(H_{r})$ is satisfied for all $\mu$ large enough.

Let us start by introducing some constants that are crucial for our next estimates. Define
\begin{equation}\label{eq-Ki}
K_{i} := \|a^{+}\|_{L^{1}(I^{+}_{i})} e^{|c||I^{+}_{i}|}, \quad i = 1\ldots,m,
\end{equation}
and
\begin{equation}\label{eq-K0}
K_{0}:= 2 \max_{i=1,\ldots,m} K_{i} \bigl{(} |I^{+}_{i}|+ e^{|c||I^{-}_{i}|}|I^{-}_{i}| \bigr{)}.
\end{equation}
By $(g_{*})$ and $g_{0}=0$, we know that $\eta(s) \to 0^{+}$ as $s\to 0^{+}$ (where $\eta(s)$ is defined in \eqref{eq-etar}).
So, we fix $r>0$ such that
\begin{equation}\label{cond-etar}
\eta(r) < \dfrac{1}{K_{0}}.
\end{equation}
Then, we fix a positive constant $\mu_{r} > \mu^{\#}$ such that
\begin{equation}\label{eq-mur}
\mu_{r} > \dfrac{K_{0} \, e^{|c||I^{-}_{i}|}} {\int_{\tau_{i}}^{\sigma_{i+1}} \int_{\tau_{i}}^{t} a^{-}(\xi) ~\!d\xi~\!dt} \dfrac{\eta(r)}{\gamma(r)},
\quad \text{ for all } \, i =1,\ldots,m,
\end{equation}
where we have set
\begin{equation*}
\gamma(r) := \min_{\frac{r}{2} \leq s \leq r} \dfrac{g(s)}{s}.
\end{equation*}

We verify that condition $(H_{d})$ of Lemma~\ref{lem-deg1-ball} is satisfied for $d=r$, chosen as in \eqref{cond-etar},
and for all $\mu \geq \mu_{r}$. Accordingly, we claim that there is no non-negative solution $u(t)$ of
\eqref{BVP-3.4}, for some $\vartheta \in \mathopen{]}0,1\mathclose{]}$ and $\mu \geq \mu_{r}$, with $\|u\|_{\infty} = r$.

Arguing by contradiction, let us suppose that, for some $\vartheta$ and $\mu$ with $0 < \vartheta \leq 1$ and $\mu \geq \mu_{r}$,
there exists a $T$-periodic solution $u(t)$ of
\begin{equation}\label{eq-vartheta}
u''(t) + c u'(t) + \vartheta a_{\mu}(t) g(u(t)) = 0,
\end{equation}
with $0 \leq u(t) \leq \max_{t\in\mathopen{[}0,T\mathclose{]}} u(t) = r$.
Reasoning as in Remark~\ref{rem-4.4}, we observe that the solution $u(t)$ in the interval of non-positivity attains its maximum at an end-point.
Thus, there is an index $j\in\{1,\ldots,m\}$ such that
\begin{equation*}
r = \max_{t\in\mathopen{[}0,T\mathclose{]}} u(t) = \max_{t\in I^{+}_{j}} u(t) = u(\hat{t}_{j}), \quad \text{ for some } \,
\hat{t}_{j}\in I^{+}_{j}=\mathopen{[}\sigma_{j},\tau_{j}\mathclose{]}.
\end{equation*}
Next, we notice that $u'(\hat{t}_{j}) = 0$.
Indeed, if $u'(t)\neq0$ for all $t\in I^{+}_{j}$ such that $u(t)=r$, then $t=\sigma_{j}$ or $t=\tau_{j}$.
If $t=\tau_{j}$, then $u'(\tau_{j})>0$ and, since the map $t\mapsto e^{ct} u'(t)$ is non-decreasing on $I^{-}_{j}$ (cf.~Remark~\ref{rem-4.4}),
we have $u'(t)\geq u'(\tau_{j}) e^{c(\tau_{j}-t)} > 0$, for all $t\in I^{-}_{j}$. Then we obtain $r=u(\tau_{j})<u(\sigma_{j+1})$, a contradiction with respect to $\|u\|_{\infty}=r$.
If $t=\sigma_{j}$, one can obtain an analogous contradiction considering the interval $I^{-}_{j-1}$ (if $j=1$, we deal with $I^{-}_{m}-T$, by $T$-periodicity).

Writing \eqref{eq-vartheta} on $I^{+}_{j}$ as
\begin{equation*}
\bigl{(}e^{ct} u'(t)\bigr{)} ' = -\vartheta a^{+}(t) g(u(t)) e^{ct},
\end{equation*}
integrating between $\hat{t}_{j}$ and $t$ and using $u'(\hat{t}_{j}) = 0$, we obtain
\begin{equation*}
u'(t) = -\vartheta \int_{\hat{t}_{j}}^{t} a^{+}(\xi) g(u(\xi)) e^{c(\xi-t)} ~\!d\xi, \quad \forall\, t\in I^{+}_{j}.
\end{equation*}
Hence,
\begin{equation}\label{eq-u'norm}
\|u'\|_{L^{\infty}(I^{+}_{j})} \leq \vartheta \|a^{+}\|_{L^{1}(I^{+}_{j})} \eta(r)r e^{|c||I^{+}_{j}|} = \vartheta K_{j} \eta(r) r.
\end{equation}
We conclude that
\begin{equation}\label{eq-4.u}
r \geq u(\tau_{j}) = u(\hat{t}_{j}) + \int_{\hat{t}_{j}}^{\tau_{j}} u'(\xi)~\!d\xi \geq r \bigl{(}1 - \vartheta K_{j}
\eta(r) |I^{+}_{j}|\bigr{)}.
\end{equation}

Now we consider the subsequent (adjacent) interval $I^{-}_{j}=\mathopen{[}\tau_{j},\sigma_{j+1}\mathclose{]}$
where the weight is non-positive. Since (as just remarked) the map $t\mapsto e^{ct} u'(t)$ is non-decreasing on $I^{-}_{j}$, we have
$e^{ct} u'(t) \geq e^{c\tau_{j}}u'(\tau_{j})$, for all $t\in I^{-}_{j}$.
Therefore, recalling also \eqref{eq-u'norm}, we get
\begin{equation*}
u'(t) \geq e^{c(\tau_{j}-t)}u'(\tau_{j}) \geq - e^{|c||I^{-}_{j}|} \vartheta K_{j} \eta(r) r, \quad \forall \, t\in I^{-}_{j}.
\end{equation*}
Integrating on $\mathopen{[}\tau_{j},t\mathclose{]} \subseteq I^{-}_{j}$ and using \eqref{eq-4.u}, we have that
\begin{equation}\label{below}
\begin{aligned}
r \geq u(t)
     &\geq u(\tau_{j}) - |I^{-}_{j}| e^{|c||I^{-}_{j}|} \vartheta K_{j} \eta(r) r
\\   &\geq r \Bigl{(} 1 - \vartheta K_{j} \bigl{(}|I^{+}_{j}| + |I^{-}_{j}| e^{|c||I^{-}_{j}|}\bigr{)}\eta(r) \Bigr{)}
\\   &\geq  r \biggl{(} 1 - \vartheta \dfrac{K_{0}}{2} \eta(r) \biggr{)}
\geq r\biggl{(} 1 - \dfrac{\vartheta}{2} \biggr{)} \geq \dfrac{r}{2}
\end{aligned}
\end{equation}
holds for all $t\in I^{-}_{j}$.
Writing \eqref{eq-vartheta} on $I^{-}_{j}$ as
\begin{equation*}
\bigl{(}e^{ct} u'(t)\bigr{)}' = \mu \vartheta a^{-}(t) g(u(t)) e^{ct}
\end{equation*}
and integrating on $\mathopen{[}\tau_{j},t\mathclose{]}\subseteq I^{-}_{j}$, we have
\begin{equation*}
u'(t) = e^{c(\tau_{j}-t)} u'(\tau_{j}) + \mu \vartheta \int_{\tau_{j}}^{t} a^{-}(\xi) g(u(\xi)) e^{c(\xi-t)} ~\!d\xi, \quad \forall \, t\in I^{-}_{j}.
\end{equation*}
Then, using \eqref{eq-u'norm} and recalling the definition of $\gamma(r)$, we find
\begin{equation*}
u'(t) \geq \vartheta r \biggl{(} -e^{|c||I^{-}_{j}|}K_{j}\eta(r) + \dfrac{1}{2}\mu \gamma(r) e^{-|c||I^{-}_{j}|} \int_{\tau_{j}}^{t} a^{-}(\xi) ~\!d\xi \biggr{)},
\quad \forall \, t\in I^{-}_{j}.
\end{equation*}
Finally, integrating on $I^{-}_{j}$, we obtain
\begin{equation*}
\begin{aligned}
u(\sigma_{j+1}) &= u(\tau_{j}) + \int_{\tau_{j}}^{\sigma_{j+1}} u'(\xi)~\!d\xi
\\ &\geq r \biggl{(} 1 -\vartheta K_{j} \eta(r)|I^{+}_{j}| - \vartheta |I^{-}_{j}| e^{|c||I^{-}_{j}|}K_{j}\eta(r)
\\ &\hspace*{30pt} + \mu \dfrac{\vartheta}{2}\gamma(r) e^{-|c||I^{-}_{j}|} \int_{\tau_{j}}^{\sigma_{j+1}} \int_{\tau_{j}}^{t} a^{-}(\xi) ~\!d\xi~\!dt \biggr{)}
\\ &> r,
\end{aligned}
\end{equation*}
a contradiction with respect to the choice of $\mu \geq \mu_{r}$ (cf.~\eqref{eq-mur}).
\qed

\begin{remark}\label{rem-4.5}
From the proof it is clear that we do not really need that $g_{0}=0$, but we only use the fact that
$r > 0$ can be chosen so that \eqref{cond-etar} is satisfied. Accordingly, our result is still valid if we assume that
$g_{0}$ is sufficiently small. Clearly, some smallness condition on $g_{0}$ has to be required
for the validity of our estimates. Indeed, the same argument of the proof (if applied to $g(s) = \lambda s$) shows that
$1/K_{0}$ must be strictly less to all the first eigenvalues
$\lambda_{1}^{i}$ as well as to all the first eigenvalues of the Dirichlet-Neumann
problems (or focal point problems) relative to the intervals $I^{+}_{i}$. As a consequence, we could slightly improve
condition $g_{0} = 0$ of Theorem~\ref{MainTheorem} to an assumption of the form $g_{0} < \lambda_{*}$, where
the optimal choice for $\lambda_{*}$ would be that of a suitable positive constant satisfying $\lambda_{*} < \min_{i} \lambda_{1}^{i}$
(as well as other similar conditions). The constant $1/K_{0}$ found in the proof
could be improved by choosing a smaller value for $K_{0}$ in \eqref{eq-K0}. Indeed, note that the factor $2$
in \eqref{eq-K0} corresponds to the lower bound $u(t) \geq r/2$ in \eqref{below}. We do not investigate further this aspect
as it is not prominent for our results. Technical estimates related to Lyapunov type inequalities and lower bounds for the
first eigenvalue of Dirichlet-Neumann problems with weights have been studied, for instance, in \cite{BrHi-1997,El-1974,Pi-2013}.
$\hfill\lhd$
\end{remark}

\begin{remark}\label{rem-4.6}
We stress that in the above proof we have used only the continuity of $g(s)$ (near $s=0$), condition $(g_{*})$ and the hypothesis $g_{0}=0$.
In our recent work \cite{FeZa-2015ade}, to obtain the existence of at least a $T$-periodic solution of equation \eqref{eq-2.1}, we have proved
the existence of an $r > 0$ small such that $(H_{d})$ holds for all $0<d\leq r$, provided that the mean value of the weight is negative.
In our case, such condition on the weight is equivalent to $\mu > \mu^{\#}$, which is a better condition than $\mu \geq \mu_{r}$
given here in our proof. However, in order to use a weaker assumption on the weight, in \cite{FeZa-2015ade}
we have to require a stronger hypothesis
on the nonlinearity $g(s)$ near zero. In particular, we have to suppose that
$g(s)$ is continuously differentiable on a right neighborhood of $s=0$ (cf.~\cite[Theorem~3.2]{FeZa-2015ade})
or that $g(s)$ is regularly oscillating at zero, i.e.~
\begin{equation*}
\lim_{\substack{\; s\to0^{+} \\ \omega\to1}}\dfrac{g(\omega s)}{g(s)} = 1
\end{equation*}
(cf.~\cite[Theorem~3.1]{FeZa-2015ade} and the references listed at the end of \cite[\S~1]{FeZa-2015ade}
for more information about regularly oscillating functions).

With regard to this topic, we observe that even if we have proved the existence of a sufficiently small $r$ such that $(H_{r})$ holds,
we can also verify that $(H_{d})$ holds \textit{for all} $0<d\leq r$ under supplementary assumptions on $g(s)$ near zero.
For instance,
this claim can be proved if we suppose that $g(s)$ satisfies
\begin{equation*}
\liminf_{s\to 0^{+}}\dfrac{g(\sigma s)}{g(s)} > 0, \quad \text{ for all } \, \sigma > 1,
\end{equation*}
(cf.~\eqref{eq-mur}). The above hypothesis is also called a \textit{lower $\sigma$-condition at zero} and it is dual with respect to
the more classical \textit{$\Delta_{2}$-condition at infinity} considered in the theory of Orlicz-Sobolev spaces (cf.~\cite[ch.~VIII]{Ad-1975}).
We refer to \cite{AlAr-1977, Dj-1998} for a discussion about these ones and related growth assumptions at infinity, as well as for a
comparison between different Karamata type conditions.
$\hfill\lhd$
\end{remark}

\subsection{The a priori bound $R^{*}$}\label{section-4.3}

Consider an arbitrary function $v\in L^{1}(\mathopen{[}0,T\mathclose{]})$ as in \eqref{eq-v}.
For example, as $v(t)$ we can take the characteristic function of the set
\begin{equation*}
A:=\bigcup_{i=1}^{m} I^{+}_{i}.
\end{equation*}
We will show that there exists $R^{*}>0$ such that, for each $\alpha\geq0$, every non-negative $T$-periodic solution $u(t)$ of \eqref{BVP-3.5}
satisfies $\max_{t\in A} u(t) < R^{*}$.

First of all, for all $i=1,\ldots,m$, we look for a bound $R_{i}>0$ such that any non-negative $T$-periodic solution $u(t)$
of \eqref{BVP-3.5}, with $\alpha \geq 0$, satisfies
\begin{equation}\label{eq-Ri}
\max_{t\in I^{+}_{i}} u(t) < R_{i}.
\end{equation}
Let us fix $i \in \{1,\ldots,m\}$. Let $0 < \varepsilon < (\tau_{i}-\sigma_{i})/2$ be fixed such that
\begin{equation*}
a(t)\not\equiv 0 \quad \text{on } \mathopen{[}\sigma_{i}+\varepsilon,\tau_{i}-\varepsilon\mathclose{]} \subseteq I^{+}_{i}
\end{equation*}
and the first (positive) eigenvalue $\hat{\lambda}$ of the eigenvalue problem
\begin{equation*}
\begin{cases}
\, \varphi'' + c\varphi' + \lambda \, a(t) \varphi = 0 \\
\, \varphi(\sigma_{i}+\varepsilon)=\varphi(\tau_{i}-\varepsilon)=0
\end{cases}
\end{equation*}
satisfies
\begin{equation*}
g_{\infty} > \hat{\lambda}.
\end{equation*}
For notational convenience, we set $\sigma_{\varepsilon}:=\sigma_{i}+\varepsilon$ and $\tau_{\varepsilon}:=\tau_{i}-\varepsilon$.
The existence of $\varepsilon$ is ensured by the continuity of the first eigenvalue as a function of the boundary points
and by hypothesis $g_{\infty} > \lambda_{1}^{i}$.

We fix a constant $M>0$ such that
\begin{equation*}
g_{\infty} > M > \hat{\lambda}.
\end{equation*}
It follows that there exists a constant $\tilde{R}=\tilde{R}(M)>0$ such that
\begin{equation*}
g(s) > M s, \quad \forall \, s\geq \tilde{R}.
\end{equation*}

Arguing as in \cite[\S~3.1]{BoFeZa-pp2015}, we can prove that
\begin{equation}\label{eq-epsa}
\begin{aligned}
& u'(t) \leq u(t)\dfrac{e^{|c|T}}{\varepsilon}, &\text{ for all } t\in \mathopen{[}\sigma_{\varepsilon},\tau_{i}\mathclose{]} \, \text{ such that } \, u'(t) \geq 0; \\
& |u'(t)| \leq u(t)\dfrac{e^{|c|T}}{\varepsilon}, &\text{ for all } t\in \mathopen{[}\sigma_{i},\tau_{\varepsilon}\mathclose{]} \, \text{ such that } \, u'(t) \leq 0.
\end{aligned}
\end{equation}
To understand how to get these inequalities, we note that the result is trivial if $u'(t) = 0$. Then,
we deal separately with the cases $u'(t) > 0$ and $u'(t) < 0$.
In the former case, since the map $t \mapsto e^{ct} u'(t)$ is
non-increasing on $I^{+}_{i}$ we find that $u'(\xi) \geq e^{c(t-\xi)}u'(t)$ for all $\sigma_{i} \leq \xi \leq t$. Therefore the first
inequality in \eqref{eq-epsa} is obtained after an integration on $\mathopen{[}\sigma_{i},t\mathclose{]}$ and observing that $t - \sigma_{i} \geq \varepsilon$.
A symmetric argument works if $u'(t) < 0$, integrating on $\mathopen{[}t,\tau_{i}\mathclose{]}$. This yields to the second inequality in \eqref{eq-epsa}.

\medskip

We are ready now to prove \eqref{eq-Ri}. By contradiction, suppose there is not a constant $R_{i}>0$ with the properties listed above. So, for each integer $n>0$
there exists a solution $u_{n}\geq0$ of \eqref{BVP-3.5} with $\max_{t\in I^{+}_{i}}u_{n}(t)=:\hat{R}_{n}>n$.
For each $n > \tilde{R}$ we take $\hat{t}_{n} \in I^{+}_{i}$ such that $u_{n}(\hat{t}_{n}) = \hat{R}_{n}$ and let
$\mathopen{]}\varsigma_{n},\omega_{n}\mathclose{[}\subseteq I^{+}_{i}$
be the intersection with $\mathopen{]}\sigma_{i},\tau_{i}\mathclose{[}$ of the maximal open interval containing $\hat{t}_{n}$ and
such that $u_{n}(t) > \tilde{R}$ for all $t\in\mathopen{]}\varsigma_{n},\omega_{n}\mathclose{[}$.
We fix an integer $N$ such that
\begin{equation*}
N > \tilde{R} + \dfrac{\tilde{R} \, T e^{2|c|T}}{\varepsilon}
\end{equation*}
and we claim that $\mathopen{]}\varsigma_{n},\omega_{n}\mathclose{[} \supseteq \mathopen{[}\sigma_{\varepsilon},\tau_{\varepsilon}\mathclose{]}$, for each $n \geq N$.
Suppose by contradiction that $\sigma_{\varepsilon} \leq \varsigma_{n}$.
In this case, we find that $u_{n}(\varsigma_{n}) =\tilde{R}$ and $u'_{n}(\varsigma_{n}) \geq 0$.
Moreover, $u'_{n}(\varsigma_{n}) \leq \tilde{R}e^{|c|T}/\varepsilon$. Using the monotonicity
of $t\mapsto e^{ct}u'(t)$, we get $e^{ct}u'(t) \leq e^{c\varsigma_{n}}u'(\varsigma_{n})$ for every $t \in \mathopen{[}\varsigma_{n},\hat{t}_{n}\mathclose{]}$
and therefore we find $u'(t) \leq \tilde{R}e^{2|c|T}/\varepsilon$
for every $t \in \mathopen{[}\varsigma_{n},\hat{t}_{n}\mathclose{]}$.
Finally, an integration on $\mathopen{[}\varsigma_{n},\hat{t}_{n}\mathclose{]}$ yields
\begin{equation*}
n < \hat{R}_{n} = u_{n}(\hat{t}_{n}) \leq \tilde{R} + \dfrac{\tilde{R} \, T e^{2|c|T}}{\varepsilon},
\end{equation*}
hence a contradiction, since $n\geq N$. A symmetric argument provides a contradiction if we suppose that $\omega_{n} \leq \tau_{\varepsilon}$.
This proves the claim.

So, we can fix an integer $N > \tilde{R}$ such that $u_{n}(t)>\tilde{R}$ for every $t\in \mathopen{[}\sigma_{\varepsilon},\tau_{\varepsilon}\mathclose{]}$
and $n\geq N$. The function $u_{n}(t)$, being a solution of equation \eqref{BVP-3.5} or equivalently of
\begin{equation*}
\bigl{(}e^{ct} u'\bigr{)}' + e^{ct} \bigl{(}a_{\mu}(t) g(u) +\alpha v(t)\bigr{)} = 0,
\end{equation*}
satisfies
\begin{equation*}
\begin{cases}
\, u_{n}'(t) = e^{-ct} y_{n}(t) \\
\, y_{n}'(t) = - e^{ct} \bigl{(}a_{\mu}(t)g(u_{n}(t))+\alpha v(t)\bigr{)}.
\end{cases}
\end{equation*}
Via a Pr\"{u}fer transformation, we pass to the polar coordinates
\begin{equation*}
e^{ct}u'_{n}(t) = r_{n}(t) \cos \vartheta_{n}(t), \qquad u_{n}(t) = r_{n}(t) \sin \vartheta_{n}(t),
\end{equation*}
and obtain, for every $t\in \mathopen{[}\sigma_{\varepsilon},\tau_{\varepsilon}\mathclose{]}$, that
\begin{equation*}
\begin{aligned}
\vartheta'_{n}(t) &= e^{-ct} \cos^{2} \vartheta_{n}(t) + \dfrac{e^{ct}\bigl{(}a^{+}(t)g(u_{n}(t))+\alpha v(t)\bigr{)}}{u_{n}(t)} \sin^{2}\vartheta_{n}(t)
\\                &\geq e^{-ct}\cos^{2} \vartheta_{n}(t) + e^{ct} M a^{+}(t) \sin^{2}\vartheta_{n}(t).
\end{aligned}
\end{equation*}
We also consider the linear equation
\begin{equation}\label{eq-4.23}
\bigl{(}e^{ct} u'\bigr{)}' + e^{ct} M a^{+}(t) u = 0
\end{equation}
and its associated angular coordinate $\vartheta(t)$ (via the Pr\"{u}fer transformation), which satisfies
\begin{equation*}
\vartheta'(t) = e^{-ct} \cos^{2} \vartheta(t) + e^{ct} M a^{+}(t) \sin^{2}\vartheta(t).
\end{equation*}
Note also that the angular functions $\vartheta_{n}$ and $\vartheta$ are non-decreasing in $ \mathopen{[}\sigma_{\varepsilon},\tau_{\varepsilon}\mathclose{]}$.
Using a classical comparison result in the frame of Sturm's theory (cf.~\cite[ch.~8, Theorem~1.2]{CoLe-1955}), we find that
\begin{equation}\label{eq-4.24}
\vartheta_{n}(t) \geq \vartheta(t), \quad \forall \, t\in \mathopen{[}\sigma_{\varepsilon},\tau_{\varepsilon}\mathclose{]},
\end{equation}
if we choose $\vartheta(\sigma_{\varepsilon}) = \vartheta_{n}(\sigma_{\varepsilon})$.
Consider now a fixed $n \geq N$. Since $u_{n}(t) \geq \tilde{R}$ for every $t \in \mathopen{[}\sigma_{\varepsilon},\tau_{\varepsilon}\mathclose{]}$, we must have
\begin{equation}\label{eq-4.25}
\vartheta_{n}(t)\in\mathopen{]}0,\pi\mathclose{[},\quad \forall \, t\in \mathopen{[}\sigma_{\varepsilon},\tau_{\varepsilon}\mathclose{]}.
\end{equation}
On the other hand, by the choice of $M>0$, we know that any non-negative solution $u(t)$ of \eqref{eq-4.23} with $u(\sigma_{\varepsilon}) > 0$
must vanish at some point
in $\mathopen{]}\sigma_{\varepsilon},\tau_{\varepsilon}\mathclose{[}$ (see \cite[ch.~8, Theorem~1.1]{CoLe-1955}).
Therefore, from $\vartheta(\sigma_{\varepsilon}) = \vartheta_{n}(\sigma_{\varepsilon}) \in \mathopen{]}0,\pi\mathclose{[}$,
we conclude that there exists $t^{*}\in \mathopen{]}\sigma_{\varepsilon},\tau_{\varepsilon}\mathclose{[}$ such that $\vartheta(t^{*}) = \pi$.
By \eqref{eq-4.24} we have that $\vartheta_{n}(t^{*}) \geq \pi$, which contradicts \eqref{eq-4.25}.

We conclude that for each $i=1,\ldots,m$ there is a constant $R_{i}>0$ such that any non-negative $T$-periodic solution $u(t)$
of \eqref{BVP-3.5}, with $\alpha \geq 0$, satisfies $\max_{t\in I^{+}_{i}} u(t) < R_{i}$.

Now we can take as $R^{*}$ any constant such that
$R^{*} > r$ (with $r$ as in Section~\ref{section-4.2}) and
\begin{equation}\label{boundR*}
R^{*} \geq R_{i}, \quad \text{for all } \, i=1,\ldots,m.
\end{equation}
Thus $\max_{t\in A} u(t) < R^{*}$ is proved. Notice that $R^{*}$ does not depend on $v(t)$ and on $\mu$, since for the constants $R_{i}$
we have used only information about $a^{+}(t)$.

Finally, using \eqref{eq-max} in Remark~\ref{rem-4.4} and reasoning as in the proof of \eqref{deg-emptyset}
(we just need to repeat verbatim the same argument, by replacing $r$ with $R^{*}$),
we can check that $R^{*}$ is an a priori bound for the solutions
on the whole domain. In this manner \eqref{upperbound} is proved.

\begin{remark}\label{rem-4.7}
If $c=0$ in equation \eqref{BVP-3.5}, the existence of the upper bound $R^{*}$ can be obtained in a different manner,
still using a Sturm comparison argument (see \cite[Lemma~6.2]{FeZa-2015ade} for the details).
$\hfill\lhd$
\end{remark}

\begin{remark}\label{rem-4.8}
A careful reading of the proof of the a priori bound shows that the inequality $u(t) < R^{*}$
for all $t\in \mathopen{[}0,T\mathclose{]}$ has been proved independently on the assumption of
$T$-periodicity of $u(t)$. Hence, the same a priori bound on $\mathopen{[}0,T\mathclose{]}$
is valid for any non-negative solution $u(t)$ of \eqref{eq-2.1}, with $u(t)$ defined
on an interval containing $\mathopen{[}0,T\mathclose{]}$.
We claim now that the following stronger property holds.
\begin{quote}
\textit{If $w(t)$ is a non-negative solution of \eqref{eq-2.1} (not necessarily periodic), then
\begin{equation*}
w(t) < R^{*}, \quad \forall \, t\in \mathbb{R}.
\end{equation*}
}
\end{quote}
To check this assertion, suppose by contradiction that there exists $t^{*}\in \mathbb{R}$ such that $w(t^{*})\geq R^{*}$.
Let also $\ell \in \mathbb{Z}$ be such that $t^{*}\in \mathopen{[}\ell T,(\ell+1)T\mathclose{]}$.
In this case, thanks to the $T$-periodicity of the weight coefficient $a_{\mu}(t)$,
the function $u(t):= w(t+\ell T)$ is still a (non-negative) solution of \eqref{eq-2.1} with
$\max_{t\in \mathopen{[}0,T\mathclose{]}} u(t) \geq u(t^{*}-\ell T) = w(t^{*}) \geq R^{*}$,
a contradiction with respect to the previous established a priori bound of $u(t)$ on $\mathopen{[}0,T\mathclose{]}$.
$\hfill\lhd$
\end{remark}

\medskip
\noindent
\textit{Verification of $(H^{R})$ for $R\geq R^{*}$. }
We have found a constant $R^{*} > r$ such that any non-negative solution $u(t)$ of \eqref{BVP-3.5},
with $\alpha \geq 0$, satisfies $\|u\|_{\infty} < R^{*}$.
Then, for $R\geq R^{*}$, the first part of $(H^{R})$ is valid independently of the choice of $\alpha_{0}$.

Let $\alpha_{0} > 0$ be fixed such that
\begin{equation*}
\alpha_{0}>\dfrac{\mu \|a^{-}\|_{L^{1}}\,\max_{0\leq s \leq R}g(s)}{\|v\|_{L^{1}}}.
\end{equation*}
We verify that for
$\alpha = \alpha_{0}$ there are no $T$-periodic solutions $u(t)$ of \eqref{BVP-3.5} with
$0\leq u(t) \leq R$ on $\mathopen{[}0,T\mathclose{]}$. Indeed, if there were, integrating on $\mathopen{[}0,T\mathclose{]}$ the differential equation
and using the boundary conditions, we obtain
\begin{equation*}
\alpha \|v\|_{L^{1}} = \alpha \int_{0}^{T}v(t)~\!dt = -\int_{0}^{T}a_{\mu}(t)g(u(t))~\!dt \leq \mu \|a^{-}\|_{L^{1}}\,\max_{0\leq s \leq R}g(s),
\end{equation*}
which leads to a contradiction with respect to the choice of $\alpha_{0}$. Thus $(H^{R})$ is verified for all $R\geq R^{*}$.
\qed

\subsection{Checking the assumptions of Lemma~\ref{lem-deg0} for $\mu$ large}\label{section-4.4}

Let $\mathcal{I}\subsetneq\{1,\ldots,m\}$ be a nonempty subset of indices and let $r>0$ be as in Section~\ref{section-4.2},
in particular $r$ satisfies \eqref{cond-etar}. Set $d=r$, $\mathcal{J}:=\{1,\ldots,m\}\setminus\mathcal{I}$ and
let $v\in L^{1}(\mathopen{[}0,T\mathclose{]})$ be an arbitrary nontrivial function satisfying \eqref{eq-v-I}.
For example, as $v(t)$ we can take the characteristic function of the set $\bigcup_{i\in\mathcal{I}} I^{+}_{i}$.

In this section we verify that $(A_{\mathcal{J},r})$, $(B_{\mathcal{J},r})$ and $(C_{\mathcal{J},r})$ (of Lemma~\ref{lem-deg0}) hold for $\mu$ sufficiently large.

\medskip
\noindent
\textit{Verification of $(A_{\mathcal{J},r})$. }
Let $\alpha \geq 0$. We claim that there exists $\mu^{*}_{\mathcal{I}} > 0$ such that for $\mu>\mu^{*}_{\mathcal{I}}$
any non-negative $T$-periodic solution $u$ of \eqref{BVP-3.5}, or equivalently of
\begin{equation}\label{eq-equiv}
\bigl{(}e^{ct} u'\bigr{)}' + e^{ct} \bigl{(}a_{\mu}(t)g(u)+\alpha v(t)\bigr{)} = 0,
\end{equation}
is such that $\max_{t\in I^{+}_{i}}u(t)\neq r$, for all $i\notin\mathcal{I}$.

By contradiction, suppose that there is a solution $u(t)$ of \eqref{BVP-3.5} with
\begin{equation}\label{eq-maxr}
\max_{t\in I^{+}_{j}} u(t) = r, \quad \text{ for some index } \, j\in \mathcal{J}.
\end{equation}
Let $\hat{t}_{j}\in I^{+}_{j}=\mathopen{[}\sigma_{j},\tau_{j}\mathclose{]}$ be such that $u(\hat{t}_{j})=r$.
If $\hat{t}_{j}=\tau_{j}$, then clearly $u(\tau_{j})=r$ and $u'(\tau_{j})\geq0$.
If $\hat{t}_{j}=\sigma_{j}$, then $u(\sigma_{j})=r$ and $u'(\sigma_{j})\leq0$.
Suppose now that $\sigma_{j} < \hat{t}_{j} < \tau_{j}$.
By conditions $(a_{*})$ and \eqref{eq-v-I}, the solution $u(t)$ satisfies the following
initial value problem on $I^{+}_{j}$
\begin{equation*}
\begin{cases}
\, \bigl{(}e^{ct} u'\bigr{)}' + e^{ct} a^{+}(t)g(u) = 0 \\
\, u(\hat{t}_{j})=r \\
\, u'(\hat{t}_{j})=0.
\end{cases}
\end{equation*}
Then, we have
\begin{equation*}
u'(t) = - \int_{\hat{t}_{j}}^{t} e^{c(\xi-t)} a^{+}(\xi) g(u(\xi)) ~\!d\xi, \quad \forall \, t\in I^{+}_{j},
\end{equation*}
and hence, recalling \eqref{eq-Ki} and \eqref{cond-etar}, we obtain this a priori bound for $|u'(t)|$ on $I^{+}_{j}$:
\begin{equation*}
|u'(t)| \leq e^{|c||I^{+}_{j}|} \|a^{+}\|_{L^{1}(I^{+}_{j})} \eta(r) r \leq K_{j} \dfrac{r}{2 K_{j}|I^{+}_{j}| } = \dfrac{r}{2|I^{+}_{j}|}, \quad \forall \, t\in I^{+}_{j}.
\end{equation*}
Therefore, the following inequality holds
\begin{equation*}
u(\tau_{j}) = u(\hat{t}_{j}) + \int_{\hat{t}_{j}}^{\tau_{j}} u'(t) ~\!dt \geq r - |I^{+}_{j}| \dfrac{r}{2|I^{+}_{j}|} = \dfrac{r}{2}.
\end{equation*}
Thus we have a lower bound for $u(\tau_{j})$.

As a first case, suppose that $\sigma_{j} < \hat{t}_{j} \leq \tau_{j}$.
Above, we have proved that
\begin{equation}\label{eq-ineq}
u(\tau_{j})\geq \dfrac{r}{2} \quad \text{ and } \quad u'(\tau_{j})\geq -\dfrac{r}{2|I^{+}_{j}|}
\end{equation}
(this is also true in a trivial manner if $\hat{t}_{j} = \tau_{j}$).
By the initial convention in Section~\ref{section-2} about the selection of the points $\sigma_{i}$ and $\tau_{i}$
in order to separate the intervals of positivity and negativity, we know that $a(t)\prec0$ on every right
neighborhood of $\tau_{j}$. Accordingly, we can fix $\delta^{+}_{j}>0$, with $0 < \delta^{+}_{j} < \sigma_{j+1}-\tau_{j}$, such that
\begin{equation}\label{eq-deltaj}
\delta^{+}_{j} \, e^{|c|\delta^{+}_{j}} < \dfrac{|I^{+}_{j}|}{2}
\end{equation}
and $a(t)\prec0$ on $\mathopen{[}\tau_{j},\tau_{j}+\delta^{+}_{j}\mathclose{]}\subseteq I^{-}_{j}$.
Since $t\mapsto e^{ct}u'(t)$ is non-decreasing on $I^{-}_{j}$, we have
\begin{equation*}
u'(t) \geq e^{c(\tau_{j}-t)} u'(\tau_{j}) \geq -e^{|c|\delta^{+}_{j}}\dfrac{r}{2|I^{+}_{j}|}, \quad \forall \, t \in\mathopen{[}\tau_{j},\tau_{j}+\delta^{+}_{j}\mathclose{]},
\end{equation*}
then
\begin{equation*}
u(t) = u(\tau_{j}) + \int_{\tau_{j}}^{t} u'(\xi) ~\!d\xi \geq \dfrac{r}{2} - \delta^{+}_{j} e^{|c|\delta^{+}_{j}}\dfrac{r}{2|I^{+}_{j}|} > \dfrac{r}{4},
 \quad \forall \, t \in\mathopen{[}\tau_{j},\tau_{j}+\delta^{+}_{j}\mathclose{]}.
\end{equation*}
We deduce that
\begin{equation*}
\dfrac{r}{4} < u(t) \leq R^{*} \quad \text{on } \mathopen{[}\tau_{j},\tau_{j}+\delta^{+}_{j}\mathclose{]},
\end{equation*}
where $R^{*}$ is the upper bound defined in Section~\ref{section-4.3}.

Let us fix
\begin{equation}\label{eq-gamma}
\gamma := \min_{\frac{r}{4} \leq s \leq R^{*}} g(s) > 0.
\end{equation}
We prove that for $\mu>0$ sufficiently large $\max_{t\in\mathopen{[}\tau_{j},\tau_{j}+\delta^{+}_{j}\mathclose{]}}u(t)>R^{*}$
(which is a contradiction to the upper bound for $u(t)$).

Note that for $t\in\mathopen{[}\tau_{j},\tau_{j}+\delta^{+}_{j}\mathclose{]} \subseteq I^{-}_{j}$, equation
\eqref{eq-equiv} reads as
\begin{equation*}
\bigl{(}e^{ct} u'(t)\bigr{)}' = \mu e^{ct} a^{-}(t)g(u(t)).
\end{equation*}
Hence, for all $t\in \mathopen{[}\tau_{j},\tau_{j}+\delta^{+}_{j}\mathclose{]}$ we have
\begin{equation*}
\begin{aligned}
u'(t) & = e^{c(\tau_{j}-t)} u'(\tau_{j})+\int_{\tau_{j}}^{t}\mu e^{c(\xi-t)} a^{-}(\xi)g(u(\xi))~\!d\xi
\\    & \geq -e^{|c|\delta^{+}_{j}}\dfrac{r}{2|I^{+}_{j}|} + \mu e^{-|c|\delta^{+}_{j}} \gamma \int_{\tau_{j}}^{t} a^{-}(\xi)~\!d\xi,
\end{aligned}
\end{equation*}
then
\begin{equation*}
\begin{aligned}
u(t) & = u(\tau_{j})+\int_{\tau_{j}}^{t}u'(\xi)~\!d\xi
\\ &  \geq \dfrac{r}{2}-\delta^{+}_{j}e^{|c|\delta^{+}_{j}}\dfrac{r}{2|I^{+}_{j}|}+ \mu e^{-|c|\delta^{+}_{j}} \gamma \int_{\tau_{j}}^{t}\biggr(\int_{\tau_{j}}^{s} a^{-}(\xi)~\!d\xi\biggr) ~\!ds
\\ & > \dfrac{r}{4}+ \mu e^{-|c|\delta^{+}_{j}} \gamma \int_{\tau_{j}}^{t}\biggr(\int_{\tau_{j}}^{s} a^{-}(\xi)~\!d\xi\biggr) ~\!ds.
\end{aligned}
\end{equation*}
Therefore, for $t=\tau_{j}+\delta^{+}_{j}$ we get
\begin{equation*}
R^{*} \geq u(\tau_{j}+\delta^{+}_{j} ) \geq \mu e^{-|c|\delta^{+}_{j}} \gamma \int_{\tau_{j}}^{\tau_{j}+\delta^{+}_{j} }\biggl(\int_{\tau_{j}}^{s} a^{-}(\xi)~\!d\xi \biggr) ~\!ds.
\end{equation*}
This gives a contradiction if $\mu$ is sufficiently large, say
\begin{equation}\label{eq-mu+}
\mu > \mu_{j}^{+} := \dfrac{R^{*} \, e^{|c|\delta^{+}_{j}}}{\gamma \int_{\tau_{j}}^{\tau_{j}+\delta^{+}_{j} }\bigl{(}\int_{\tau_{j}}^{s} a^{-}(\xi)~\!d\xi \bigr{)} ~\!ds},
\end{equation}
recalling that $\int_{\tau_{j}}^{t} a^{-}(\xi) ~\!d\xi > 0$ for each $t\in \mathopen{]}\tau_{j},\sigma_{j+1}\mathclose{]}$.

As a second case, if $\hat{t}_{j} = \sigma_{j}$, we consider the interval $I^{-}_{j-1}$ (if $j=1$, we deal with $I^{-}_{m}-T$, by $T$-periodicity).
We define $\delta^{-}_{j}$ in a similar manner, using the fact that $a(t)$ is not identically zero on all left neighborhoods of $\sigma_{j}$.
We obtain a contradiction for
\begin{equation}\label{eq-mu-}
\mu > \mu_{j}^{-} := \dfrac{R^{*} \, e^{|c|\delta^{-}_{j}}}{\gamma \int_{\sigma_{j}-\delta^{-}_{j}}^{\sigma_{j}}\bigl{(}\int_{s}^{\sigma_{j}} a^{-}(\xi)~\!d\xi \bigr{)} ~\!ds}.
\end{equation}

At the end, we define
\begin{equation*}
\mu^{*}_{\mathcal{I}} := \max_{j\notin \mathcal{I}} \mu^{\pm}_{j}
\end{equation*}
and we obtain a contradiction if $\mu > \mu^{*}_{\mathcal{I}}$.
Hence condition $(A_{\mathcal{J},r})$ is verified.

\begin{remark}\label{rem-4.9}
We emphasize the fact that the constant $\mu^{*}_{\mathcal{I}}$ is chosen independently on the solution $u(t)$
for which we have made the estimates. In fact, the numbers $\mu^{\pm}_{j}$ depend only on absolute constants, like $R^{*}$, $\gamma$
(depending only on $g(s)$), the constants $\delta^{\pm}_{j}$ defined as in \eqref{eq-deltaj} and, finally,
the integrals of the negative part of the weight function. Observe also that, in order to have non-vanishing denominators
in the definition of the $\mu^{\pm}_{j}$, we had to be sure that $a(t) \prec 0$ on each right neighborhood of $\tau_{j}$
as well as on each left neighborhood of $\sigma_{j}$, consistently with the choice we made at the beginning.
$\hfill\lhd$
\end{remark}

\begin{remark}\label{rem-4.10}
A careful reading of the above proof shows that the result does not involve the periodicity of the function $u(t)$,
since we have only analyzed the behavior of the solution on an interval of positivity of the weight and
on the adjacent intervals of negativity. Indeed, we claim that the following result holds.
\begin{quote}
\textit{There exists a constant $\mu^{*}>0$ such that, for every $\mu > \mu^{*}$, any non-negative solution $w(t)$ of
\eqref{eq-2.1} (not necessarily periodic), with $w(t) < R^{*}$ for all $t\in \mathbb{R}$, is such that
\begin{equation*}
\max \bigl{\{} w(t) \colon t\in I^{+}_{i} + \ell T \bigr{\}} \neq r, \quad \forall \, i\in\{1,\ldots,m\}, \; \forall \, \ell\in \mathbb{Z}.
\end{equation*}
}
\end{quote}
To check this assertion, suppose by contradiction that there exist $i\in\{1,\ldots,m\}$ and $\ell\in \mathbb{Z}$
such that $\max \bigl{\{} w(t) \colon t\in I^{+}_{i} + \ell T \bigr{\}} = r$.
Thanks to the $T$-periodicity of the weight coefficient $a_{\mu}(t)$, the function $u(t):= w(t+\ell T)$ is still a (non-negative) solution of \eqref{eq-2.1} with
$\max_{t\in I^{+}_{i}} u(t) = r$. So, we are in the same situation like at the beginning of the \textit{Verification of $(A_{\mathcal{J},r})$}
(cf.~\eqref{eq-maxr}). From now on, we proceed exactly the same as in that proof and obtain a contradiction
with respect to the bound $u(t) < R^{*}$ (for all $t$) taking $\mu > \max \{\mu^{+}_{i},\mu^{-}_{i}\}$.
Hence the result is proved for $\mu$ sufficiently large, namely $\mu > \max_{i} \mu^{\pm}_{i}$.
$\hfill\lhd$
\end{remark}

\medskip
\noindent
Now, we fix $\mu>\mu^{*}_{\mathcal{I}}$ and we prove that conditions  $(B_{\mathcal{J},r})$ and $(C_{\mathcal{J},r})$ hold
(independently of the coefficient $\mu$ previously fixed).

\medskip
\noindent
\textit{Verification of $(B_{\mathcal{J},r})$. }
Let $u(t)$ be any non-negative $T$-periodic solution of \eqref{BVP-3.5} with $\max_{t\in I^{+}_{j}} u(t) \leq r$, for all $j\in\mathcal{J}$.
Notice that $R^{*}$ is an upper bound for all the solutions of \eqref{BVP-3.5} and $R^{*}$ is independent on the functions
$v\in L^{1}(\mathopen{[}0,T\mathclose{]})$ satisfying \eqref{eq-v} (and hence \eqref{eq-v-I}).
So condition $(B_{\mathcal{J},r})$ is verified with $D_{\beta}=R^{*}$, for every $\beta \geq 0$.

\medskip
\noindent
\textit{Verification of $(C_{\mathcal{J},r})$. }
Recalling the choice of $v(t)$ in \eqref{eq-v-I}, we take an index $i\in \mathcal{I} = \{1,\ldots,m\} \setminus {\mathcal{J}}$
such that $v\not\equiv 0$ on $I^{+}_{i}$ and we also fix $\varepsilon>0$ such that
\begin{equation*}
v(t)\not\equiv 0 \quad \text{on } \mathopen{[}\sigma_{i}+\varepsilon,\tau_{i}- \varepsilon\mathclose{]}.
\end{equation*}
We claim that $(C_{\mathcal{J},r})$ is satisfied for $\alpha_{0}$ such that
\begin{equation*}
\alpha_{0} > \dfrac{R^{*}}{\int_{\sigma_{i}+\varepsilon}^{\tau_{i}+\varepsilon}v(t) ~\!dt} \biggl{(}\dfrac{2 e^{|c|T}}{\varepsilon} + |c| \biggr{)}.
\end{equation*}

To prove our assertion, first of all we observe that if $u(t)$ is any non-negative solution of \eqref{BVP-3.5}, then
\begin{equation}\label{eq-eps}
|u'(t)| \leq u(t) \dfrac{e^{|c|T}}{\varepsilon}, \quad \forall\, t\in \mathopen{[}\sigma_{i}+\varepsilon,\tau_{i} - \varepsilon\mathclose{]}.
\end{equation}
Such inequality can be found in \cite[\S~3.1]{BoFeZa-pp2015} and has been already proved and used in Section~\ref{section-4.3}
(see the inequalities in \eqref{eq-epsa}).

Let $u(t)$ be a non-negative $T$-periodic solution of \eqref{BVP-3.5}, which reads as
\begin{equation*}
\alpha v(t) = - u'' - c u' - a^{+}(t) g(u)
\end{equation*}
on the interval $I^{+}_{i}$. Recall also that $\|u\|_{\infty} < R^{*}$ (cf.~\eqref{upperbound}).
Integrating the equation on $\mathopen{[}\sigma_{i}+\varepsilon,\tau_{i} - \varepsilon\mathclose{]}$, for $\alpha=\alpha_{0}$, we obtain
\begin{equation*}
\begin{aligned}
&\alpha_{0} \int_{\sigma_{i}+\varepsilon}^{\tau_{i}- \varepsilon}v(t) ~\!dt =
\\ &= u'(\sigma_{i}+\varepsilon) - u'(\tau_{i}- \varepsilon) + c \bigl{(} u(\sigma_{i}+\varepsilon)
    - u(\tau_{i}- \varepsilon) \bigr{)} - \int_{\sigma_{i}+\varepsilon}^{\tau_{i} - \varepsilon} a^{+}(t)g(u(t))~\!dt
\\ &\leq 2 \dfrac{R^{*}}{\varepsilon} e^{|c|T} + |c|R^{*} < \alpha_{0} \int_{\sigma_{i}+\varepsilon}^{\tau_{i}- \varepsilon}v(t) ~\!dt,
\end{aligned}
\end{equation*}
a contradiction. Hence $(C_{\mathcal{J},r})$ is verified.

\begin{remark}\label{rem-4.11}
Note that for the verification of $(B_{\mathcal{J},r})$ and $(C_{\mathcal{J},r})$ the small constant $r$ has no played any relevant role.
In fact, we used only the information about the existence of the a priori bound $R^{*}$ obtained in Section~\ref{section-4.3}.
$\hfill\lhd$
\end{remark}

\medskip
\noindent
In conclusion, all the assumptions of Lemma~\ref{lem-deg0} have been verified for a fixed $r$ and for
$\mu > \mu^{*}_{\mathcal{I}}$.
\qed

\subsection{A posteriori bounds}\label{section-4.5}

Let $\mathcal{I}\subseteq\{1,\ldots,m\}$ be a nonempty subset of indices and let $r$ and $R$ be fixed as explained in Section~\ref{section-4.1}.
Theorem~\ref{MainTheorem} ensures the existence of at least a $T$-periodic positive solution $u(t)$ of \eqref{eq-2.1} with $u\in\Lambda^{\mathcal{I}}_{r,R}$.
More in details, the solution $u(t)$ is such that $r<u(\hat{t}_{i})<R$ for some $\hat{t}_{i}\in I^{+}_{i}$, if $i\in\mathcal{I}$, and
$0<u(t)<r$ for all $t\in I^{+}_{i}$, if $i\notin\mathcal{I}$.

As premised in Remark~\ref{rem-4.1}, in this section we prove that, for $\mu$ sufficiently large, it holds that $0<u(t)<r$ also on the non-positivity intervals $I^{-}_{i}$.
First of all, by \eqref{eq-max}, we observe that the solution $u(t)$ in the interval of non-positivity attains its maximum at an end-point.
Therefore it is sufficient to show that
\begin{equation*}
u(\sigma_{i}) < r \quad \text{ and } \quad u(\tau_{i}) < r, \quad \text{ for all } \, i=1,\ldots,m.
\end{equation*}
If $i\notin\mathcal{I}$, there is nothing to prove, because $u(t) < r$ on $I^{+}_{i}=\mathopen{[}\sigma_{i},\tau_{i}\mathclose{]}$.
Let us deal with the case $i\in\mathcal{I}$ and, by contradiction, suppose that
\begin{equation*}
u(\tau_{i}) \geq r.
\end{equation*}
Proceeding as in Section~\ref{section-4.4}, one can prove that
\begin{equation*}
u'(\tau_{i}) \geq - K_{i} \max_{0\leq s \leq R^{*}} g(s)
\end{equation*}
and hence (using estimates analogous to those following after \eqref{eq-ineq}) that there exists $\delta^{+}_{i}>0$ such that, for $t=\tau_{i}+\delta^{+}_{i}\in I^{-}_{i}$ and $\mu$ sufficiently large, we obtain
\begin{equation*}
u(\tau_{i}+\delta^{+}_{i}) \geq R^{*},
\end{equation*}
a contradiction.

A similar argument generates a contradiction also assuming $u(\sigma_{i}) \geq r$, for $i\in\mathcal{I}$.

\medskip

Finally, repeating again the argument in Section~\ref{section-4.4}, we can also check that if $u(t) = u_{\mu}(t)$
is a positive $T$-periodic solution of \eqref{eq-2.1} (belonging to a set of the form $\Lambda^{\mathcal{I}}_{r,R}$), then, for $\mu\to +\infty$,
$u_{\mu}$ tends uniformly to zero on the intervals $I^{-}_{i}$.

\begin{remark}\label{rem-4.12}
Notice that the same arguments work for any \textit{arbitrary} non-negative solution which is upper bounded by $R^{*}$
(at any effect this observation is analogous to Remark~\ref{rem-4.10}, since it only involves the behavior of the solution
in the intervals where the weight is negative, without requiring the periodicity of the solution).
Indeed, the following result holds.
\begin{quote}
\textit{There exists a constant $\mu^{**}>0$ such that, for every $\mu > \mu^{**}$, any non-negative solution $w(t)$ of
\eqref{eq-2.1} (not necessarily periodic), with $w(t) < R^{*}$ for all $t\in \mathbb{R}$, is such that
\begin{equation*}
\max \bigl{\{} w(t) \colon t\in I^{-}_{i} + \ell T \bigr{\}} < r, \quad \forall \, i\in\{1,\ldots,m\}, \; \forall \, \ell\in \mathbb{Z}.
\end{equation*}
}
\end{quote}
To check this assertion, suppose by contradiction that there exist $i\in\{1,\ldots,m\}$ and $\ell\in \mathbb{Z}$
such that $\max \bigl{\{} w(t) \colon t\in I^{-}_{i} + \ell T \bigr{\}} \geq r$.
Thanks to the $T$-periodicity of the weight coefficient $a_{\mu}(t)$, the function $u(t):= w(t+\ell T)$ is still a (non-negative) solution of \eqref{eq-2.1} with
$\max_{t\in I^{-}_{i}} u(t) \geq r$. This means that $u(\sigma_{i}) \geq r$ or $u(\tau_{i})\geq r$.
At this point we achieve a contradiction exactly as above.
$\hfill\lhd$
\end{remark}

\section{Related results}\label{section-5}

In this section we deal with corollaries, variants and applications of Theorem~\ref{MainTheorem}.
We also analyze the case of a nonlinearity $g(s)$ which is smooth in order to give a nonexistence result, too.

The following corollaries are obtained as direct applications of Theorem~\ref{MainTheorem}.

\begin{corollary}\label{cor-5.1}
Let $g \colon {\mathbb{R}}^{+} \to {\mathbb{R}}^{+}$ be a continuous function satisfying $(g_{*})$,
\begin{equation*}
g_{0} = 0 \quad \text{ and } \quad g_{\infty} > 0.
\end{equation*}
Let $a \colon \mathbb{R} \to \mathbb{R}$ be a $T$-periodic locally integrable function satisfying $(a_{*})$.
Then there exists $\nu^{*}>0$ such that for all $\nu>\nu^{*}$ there exists $\mu^{*}=\mu^{*}(\nu)$
such that for $\mu>\mu^{*}$ there exist at least $2^{m}-1$ positive $T$-periodic solutions of
\begin{equation}\label{eq-5.1}
u'' + c u' + \bigl{(} \nu a^{+}(t) - \mu a^{-}(t)\bigr{)} g(u) = 0.
\end{equation}
\end{corollary}

The constant $\nu^{*}$ will be chosen so that $\nu^{*} g_{\infty} > \max_{i} \lambda_{1}^{i}$.
The lower bound for $g_{\infty}$ in the main theorem is automatically satisfied when
$g_{\infty} = +\infty$. Accordingly, we have.

\begin{corollary}\label{cor-5.2}
Let $g \colon {\mathbb{R}}^{+} \to {\mathbb{R}}^{+}$ be a continuous function satisfying $(g_{*})$,
\begin{equation*}
g_{0} = 0 \quad \text{ and } \quad g_{\infty} = +\infty.
\end{equation*}
Let $a \colon \mathbb{R} \to \mathbb{R}$ be a $T$-periodic locally integrable function satisfying $(a_{*})$.
Then there exists $\mu^{*}>0$ such that for all $\mu>\mu^{*}$ equation \eqref{eq-2.1} has at least $2^{m}-1$ positive $T$-periodic solutions.
\end{corollary}

A typical case in which the above corollary applies is for the power nonlinearity $g(s) = s^{p}$ (for $p > 1$), so that the next result holds.

\begin{corollary}\label{cor-5.3}
Let $a \colon \mathbb{R} \to \mathbb{R}$ be a $T$-periodic locally integrable function satisfying $(a_{*})$.
Then there exists $\mu^{*}>0$ such that for all $\mu>\mu^{*}$ there exist at least $2^{m}-1$ positive $T$-periodic solutions of
\begin{equation*}
u'' + c u' + \bigl{(} a^{+}(t) - \mu a^{-}(t)\bigr{)} u^{p} = 0, \quad p > 1.
\end{equation*}
\end{corollary}

Using Remark~\ref{rem-4.2}, we can also obtain the following result which, in some sense, is dual with respect to
Corollary~\ref{cor-5.1}.

\begin{corollary}\label{cor-5.4}
Let $g \colon {\mathbb{R}}^{+} \to {\mathbb{R}}^{+}$ be a continuous function satisfying $(g_{*})$,
\begin{equation*}
g_{0} > 0 \quad \text{ and } \quad g_{\infty} = +\infty.
\end{equation*}
Let $a \colon \mathbb{R} \to \mathbb{R}$ be a $T$-periodic locally integrable function satisfying $(a_{*})$.
Then there exists $\nu_{*}>0$ such that for all $0 < \nu <\nu_{*}$ there exists $\mu^{*}=\mu^{*}(\nu)$
such that for $\mu>\mu^{*}$ there exist at least $2^{m}-1$ positive $T$-periodic solutions of
equation \eqref{eq-5.1}.
\end{corollary}

Combining Theorem~\ref{MainTheorem} with \cite[Lemma~4.1]{BoFeZa-pp2015} (or \cite[Proposition~3.1]{FeZa-2015ade}),
the following result can be obtained.

\begin{corollary}\label{cor-5.5}
Let $g \colon {\mathbb{R}}^{+} \to {\mathbb{R}}^{+}$ be a continuously differentiable function satisfying $(g_{*})$,
\begin{equation*}
g_{0} = 0 \quad \text{ and } \quad g_{\infty} > 0.
\end{equation*}
Let $a \colon \mathbb{R} \to \mathbb{R}$ be a $T$-periodic locally integrable function satisfying $(a_{*})$.
Then for all $\mu>0$ such that $\int_{0}^{T} a_{\mu}(t)~\!dt < 0$ there exist two constants $0 < \omega_{*} \leq \omega^{*}$
(depending on $\mu$) such that equation
\begin{equation}\label{eq-5.2}
u'' + c u' + \nu \bigl{(} a^{+}(t) - \mu a^{-}(t)\bigr{)} g(u) = 0
\end{equation}
has no positive $T$-periodic solutions for $0 < \nu < \omega_{*}$ and at least one positive $T$-periodic solution
for $\nu > \omega^{*}$.
Moreover there exists $\nu^{*} > 0$ such that for all $\nu >\nu^{*}$ there exists $\mu^{*}=\mu^{*}(\nu)$
such that for $\mu>\mu^{*}$ equation \eqref{eq-5.2} has at least $2^{m}-1$ positive $T$-periodic solutions.
\end{corollary}

Equation \eqref{eq-5.2} is substantially equivalent to \eqref{eq-5.1}. We have preferred to write it in a slightly
different form for sake of convenience in stating Corollary~\ref{cor-5.5}.

\section{Subharmonic solutions and complex dynamics}\label{section-6}

Our next goal is to apply the preceding results concerning the existence and multiplicity of periodic solutions to the search of
subharmonic solutions. Then, we shall use the information obtained on the subharmonics to produce bounded positive solutions
which are not necessarily periodic and can reproduce an arbitrary coin-tossing sequence.

\subsection{Subharmonic solutions}\label{section-6.1}

In the previous sections we have studied the existence and multiplicity of $T$-periodic solutions, assuming that the
weight coefficient is a $T$-period function. Since any $T$-periodic coefficient can be though as a $kT$-periodic
function, with the same technique we can look for the existence of $kT$-periodic solutions (with $k$ an integer).
In this context, a typical problem which occurs is that of proving the minimality of the period, that is, to ensure the presence of
subharmonic solutions of order $k$, according to the standard definition that we recall now for the reader's convenience.
A \textit{subharmonic solution of order $k$}, with $k\geq2$ an integer, is a $kT$-periodic solution which is not $lT$-periodic
for any integer $l=1,\ldots,k-1$. Throughout this section, for the sake of simplicity in the exposition, when not explicitly
stated we assume that \textit{$k$ is an integer such that $k\geq 2$}.

Generally speaking, if $x(t)$ is a $kT$-periodic solution of a differential system $x' = f(t,x)$ in $\mathbb{R}^{N}$,
with $f(t + T,x) = f(t,x)$ for all $t\in \mathbb{R}$ and $x\in {\mathbb{R}}^{N}$, the information that $x(t)$
is not $lT$-periodic, for any integer $l=1,\ldots,k-1$, is not enough to conclude that $kT$ is actually the minimal positive
period of the solution. However, in many significant situations, it is possible to derive such a conclusion, under suitable conditions
on the vector field $f(t,x)$. For instance, in case of \eqref{eq-1.1} and for $g(s)$ satisfying $(g_{*})$, it is easy to check
that any positive subharmonic solution of order $k$ is a solution of minimal period $kT$ provided that $T$ is the minimal period
of the weight function. The problem of minimality of the period in the study of subharmonic solutions is a topic of considerable importance
in this area of research and different approaches have been proposed depending also on the nature of the techniques adopted to obtain the solutions.
See for instance \cite{BoZa-2013, CaLa-1996, DiZa-1993, MiTa-1988,Pl-1966} for some pertinent remarks. It may be also interesting to
observe that equations of the form \eqref{eq-1.1}, with $w(t)$ a non-constant $T$-periodic coefficient, do not possess
\textit{exceptional solutions}, i.e.~solutions having a minimal period which has an irrational ratio with $T$ (cf.~\cite[ch.~I, \S~4]{SaCo-1964}).
In view of all these premises, throughout the section we suppose that the function $a(t)$ is a periodic function having $T > 0$
as a minimal period.

As a final remark, we observe that if $u(t)$ is a $kT$-periodic solution of \eqref{eq-1.1}
then, for any integer $\ell$ with $1 \leq \ell \leq k-1$, also the function $v_{\ell}(t):= u(t + \ell T)$ is a $kT$-periodic solution
and it has $kT$ as a minimal period if and only if $kT$ is the minimal period for $u(t)$. Accordingly, whenever it happens that we find a subharmonic solution
of order $k$, we also find other $k-1$ subharmonic solutions (of the same order). These solutions, even if formally distinct,
will be considered as belonging to the same periodicity class and for the purposes of counting the number of solutions will count only once.

In order to present in a simplified manner our main multiplicity results for subharmonic solutions, we
first take a class of weights of special form, namely we suppose that
\begin{quote}
\textit{$a \colon \mathbb{R} \to \mathbb{R}$ is a continuous periodic sign-changing function
with simple zeros and with minimal period $T$, such that there exist two consecutive zeros $\alpha < \beta$
so that $a(t) > 0$ for all $t \in \mathopen{]}\alpha,\beta\mathclose{[}$ and $a(t) < 0$ for all $t\in\mathopen{]}\beta,\alpha + T\mathclose{[}$.}
\end{quote}
That is $a(t)$ has only one positive hump and one negative one in a period interval. In such a simplified situation, the following result holds.

\begin{theorem}\label{th-6.1}
Let $g \colon {\mathbb{R}}^{+} \to {\mathbb{R}}^{+}$ be a continuous function satisfying $(g_{*})$,
\begin{equation*}
g_{0} = 0 \quad \text{ and } \quad g_{\infty} = +\infty.
\end{equation*}
Then there exists $\mu^{*}>0$ such that, for all $\mu>\mu^{*}$ and for every integer $k\geq 2$,
equation \eqref{eq-2.1} has a subharmonic solution of order $k$.
\end{theorem}

\begin{proof}
Without loss of generality (if necessary, we can make a shift by $\alpha$ in the time variable), we suppose that
\begin{equation*}
a(t) > 0, \, \text{ for all } \; 0 < t < \tau:= \beta-\alpha, \quad \text{ and } \quad a(t) < 0, \, \text{ for all } \; \tau < t < T.
\end{equation*}
Let us fix an integer $k\geq 2$ and consider the $T$-periodic function $a(t)$ as a $kT$-periodic weight
on the interval $\mathopen{[}0,kT\mathclose{]}$. In such an interval we have condition $(a_{*})$ satisfied with
\begin{equation*}
I^{+}_{i} = \mathopen{[}(i-1)T,\tau + (i-1)T\mathclose{]}
\quad \text{ and } \quad
I^{-}_{i} = \mathopen{[}\tau + (i-1)T,iT\mathclose{]},
\end{equation*}
for $i=1,\ldots,k$.
With respect to the notation introduced in Section~\ref{section-2}, we also have $m = k$,
$\sigma_{1} = 0$, $\sigma_{k+1} = kT$ and
\begin{equation*}
0 < \tau_{1} = \tau < \sigma_{2} = T < \ldots < \sigma_{k} = (k-1)T < \tau_{k} = \tau + (k-1)T < kT.
\end{equation*}
In this setting we can apply Corollary~\ref{cor-5.2}, which ensures the existence of $2^{k} -1$
positive solutions which are also $kT$-periodic, provided that $\mu$ is sufficiently large.

Even if we have found $kT$-periodic solutions, our proof is not yet complete. In fact we still have to verify that $\mu^{*}$
(found in the proof of Theorem~\ref{MainTheorem}) is independent on $k$ and,
moreover, that among the $2^{k}-1$ periodic solutions there is at least one subharmonic of order $k$.

For the first question, we need to check how the bounds obtained in the proof of Theorem~\ref{MainTheorem} depend on the weight function.
First of all we underline that, by the $T$-periodicity of $a(t)$, the constants $K_{i}$ defined in \eqref{eq-Ki} are
all equal for $i=1,\ldots,k$, then $K_{0}$ does not depend on $k$ (cf.~\eqref{eq-K0}).
Consequently condition \eqref{cond-etar} reads as
\begin{equation*}
\eta(r) \, 2 \|a^{+}\|_{L^{1}(\mathopen{[}0,\tau\mathclose{]})} e^{|c|\tau} \bigl{(}\tau + e^{|c|(T-\tau)}(T-\tau)\bigr{)} < 1
\end{equation*}
and so the small constant $r>0$ is absolute and depends only on $c$, $\|a^{+}\|_{L^{1}(\mathopen{[}0,\tau\mathclose{]})}$, $T$ and $\tau$,
but it does not depend on $k$.

Once that we have fixed $r > 0$, using again the $T$-periodicity of the weight, we notice also that the lower bounds $\mu^{\#}$ and $\mu_{r}$
do not depend on $k$ (cf.~\eqref{eq-mud} and \eqref{eq-mur}).

The constant $R^{*}$ is chosen in \eqref{boundR*} and depends on the a priori bounds $R_{i}$, which in turn depend
on the properties of $a(t)$ restricted to the interval $I^{+}_{i}$. In our case, by the $T$-periodicity of the coefficient $a(t)$,
we can choose $R_{i}$ as constant with respect to $i$. Therefore, $R^{*}$ is independent on $k$ and then also the constant $\gamma$
defined in \eqref{eq-gamma} does not depend on $k$.
By the periodicity of $a(t)$, the constants $\delta^{+}_{j}$ introduced in Section~\ref{section-4.4} (see \eqref{eq-deltaj})
can be also taken all equal to a common value $\delta^{+} = \delta$ such that $0 < \delta < T- \tau$ and $\delta e^{|c|\delta} < \tau/2$.
The same choice can be made for $\delta^{-}_{j}$ in order to have $\delta^{-}_{j} = \delta$ for all $j$.
From these choices of the constants $R^{*}$, $\gamma$ and $\delta$,
for all $j=1,\ldots,k$ we take $\mu^{\pm}_{j}$, according to \eqref{eq-mu+} and \eqref{eq-mu-}, as
\begin{equation*}
\mu_{j}^{+} = \mu^{+}:=
\dfrac{R^{*} \, e^{|c|\delta}}{\gamma \int_{\tau}^{\tau+\delta }
\bigl{(}\int_{\tau}^{s} a^{-}(\xi)~\!d\xi \bigr{)} ~\!ds}
\end{equation*}
and
\begin{equation*}
\mu_{j}^{-} = \mu^{-} :=
\dfrac{R^{*} \, e^{|c|\delta}}{\gamma \int_{T-\delta}^{T}\bigl{(}\int_{s}^{T} a^{-}(\xi)~\!d\xi \bigr{)} ~\!ds},
\end{equation*}
respectively. Therefore, setting
\begin{equation*}
\mu^{*} := \mu_{r} \vee \max\bigl{\{}\mu^{+},\mu^{-}\bigr{\}},
\end{equation*}
we have found an absolute constant which is independent on $k$ and also does not depend on the set
of indices $\mathcal{I}$.
This solves the first question.

To complete the proof, we show how to produce at least one subharmonic solution. It is sufficient to take $\mathcal{I}:= \{1\}$.
As explained in Remark~\ref{rem-4.1} and also at the end of Section~\ref{section-4.1},
there exists a positive $kT$-periodic solution $u(t)$ for \eqref{eq-2.1} such that $u\in\Lambda^{\{1\}}_{r,R}$.
This implies that there exists $\hat{t}_{1}\in I^{+}_{1}=\mathopen{[}0,\tau\mathclose{]}$ such that $r < u(\hat{t}_{1}) < R$ and,
if $i\neq 1$, $0<u(t)<r$ for all $t\in I^{+}_{i}$. Then $u(\hat{t}_{1})\neq u(t)$ for all $t\in I_{i}$ with $i\neq 1$, and hence
$u$ is not $lT$-periodic for all $l=1,\ldots,k-1$. We conclude that $u$ is a subharmonic solution of order $k$.
\end{proof}

\begin{remark}\label{rem-6.1}
The fact that the weight coefficient has simple zeros has been assumed only for convenience in the exposition.
The same result holds true if we suppose that there are $\alpha < \beta$ such that $a(t)\succ 0$ on $\mathopen{[}\alpha,\beta\mathclose{]}$,
$a(t) \prec 0$ on $\mathopen{[}\beta, \alpha + T\mathclose{]}$ and $a(t)$ is not identically zero on all left neighborhoods of $\alpha$
and on all right neighborhoods of $\beta$. The possibility of more changes of sign of $a(t)$ in a period can be considered as well.
$\hfill\lhd$
\end{remark}

\begin{remark}\label{rem-6.2}
We stress the fact that $\mu^{*}$ is chosen independent on $k$ and also independent on the set of indices $\mathcal{I}$.
This is a crucial observation if one wants to prove the existence of bounded solutions defined on the whole real line
and with \textit{any prescribed behavior} as a limit of subharmonic solutions (see Section~\ref{section-6.3} and \cite{BaBoVe-jde2015}).
$\hfill\lhd$
\end{remark}

\subsection{Counting the subharmonic solutions}\label{section-6.2}

Theorem~\ref{th-6.1} guarantees the existence of at least a subharmonic solution of order $k$ for \eqref{eq-2.1},
but, in general, there are many solutions of this kind.
Even if in the statement we have not described the number of subharmonics and their behavior, this can be achieved
(with the same proof) just exploiting more deeply the content of Theorem~\ref{MainTheorem}.
In this section, given an integer $k\geq 2$, we look for an estimate on the number of subharmonic solutions of order $k$.
To this purpose, we adapt to our setting some considerations which are typical in the area of dynamical systems, combinatorics and graph theory.

First of all, we need to introduce a notation, which is borrowed from \cite{BaBoVe-jde2015}.
We start with an alphabet of two symbols, conventionally indicated as $\{0,1\}$, and denote by $\{0,1\}^{k}$
the set of the $k$-tuples of $\{0,1\}$, that is the set of finite words of length $k$.
We also denote by $0^{[k]}$ the $0$-string in $\{0,1\}^{k}$.

For simplicity, we still consider the special weight coefficient as in the setting of Theorem~\ref{th-6.1}.
Recalling the definitions of $I^{\pm}_{i}$, for $i=1,\ldots,k$, given by
\begin{equation*}
I^{+}_{i} = \mathopen{[}(i-1)T,\tau + (i-1)T\mathclose{]}
\quad \text{ and } \quad
I^{-}_{i} = \mathopen{[}\tau + (i-1)T,iT\mathclose{]},
\end{equation*}
and reworking as in the proof of Theorem~\ref{th-6.1}, we have the following result.

\begin{theorem}\label{th-6.2}
Let $g \colon {\mathbb{R}}^{+} \to {\mathbb{R}}^{+}$ be a continuous function satisfying $(g_{*})$,
\begin{equation*}
g_{0} = 0 \quad \text{ and } \quad g_{\infty} = +\infty.
\end{equation*}
Then there exist $0 < r < R$ and $\mu^{**} > 0$ such that, for all $\mu > \mu^{**}$ and for every integer $k\geq 2$,
given any $k$-tuple $\mathcal{L}^{[k]} = (s_{i})_{i=1,\ldots,k}$ with $\mathcal{L}^{[k]}\neq 0^{[k]}$,
there exists at least one $kT$-periodic positive solution of equation \eqref{eq-2.1} such that $\|u\|_{\infty} < R$ and
\begin{itemize}
\item $0<u(t)<r$ on $I^{+}_{i}$, if ${s}_{i}=0$;
\item $r<u(\hat{t}_{i})<R$ for some $\hat{t}_{i}\in I^{+}_{i}$, if ${s}_{i}=1$;
\item $0<u(t)<r$ on $I^{-}_{i}$, for all $i=1,\ldots,k$.
\end{itemize}
\end{theorem}

\begin{proof}
We proceed exactly as in the proof of Theorem~\ref{th-6.1} till to the final step where we chose the set
of indices $\mathcal{I}$. At this moment $r$, $R$ and $\mu^{*}$ are determined and we are free to take any $\mu > \mu^{*}$.
Let us consider an arbitrary integer $k\geq 2$. Observe that we took $\mathcal{I} = \{1\}$ in order to be sure to have a subharmonic,
however, Theorem~\ref{MainTheorem} provides the existence of a positive $kT$-periodic solution in $\Lambda^{\mathcal{I}}_{r,R}$
for any nonempty subset $\mathcal{I}$ of $\{1,\ldots,k\}$.

Given an arbitrary $k$-tuple ${\mathcal{L}}^{[k]} = (s_{i})_{i=1,\ldots,k}$ with ${\mathcal{L}}^{[k]}\neq 0^{[k]}$,
using a typical bijection between $\{0,1\}^{k}$ and the power set $\mathscr{P}(\{1,\ldots,k\})$,
we associate to $\mathcal{L}^{[k]}$ the set
\begin{equation*}
\mathcal{I}_{\mathcal{L}^{[k]}} := \bigl{\{}i\in\{1,\ldots,k\}\colon {s}_{i}=1 \bigr{\}}.
\end{equation*}
Now, applying Theorem~\ref{MainTheorem}, we have guaranteed the existence of at least one $kT$-periodic solution
$u(t)$ which is positive and belongs to the set $\Lambda^{\mathcal{I}_{\mathcal{L}^{[k]}}}_{r,R}$.
Recalling the definition of $\Lambda^{\mathcal{I}}_{r,R}$ in \eqref{eq-2.lambda}, we find that
$u(t)$ satisfies the first two conditions in the statement of the theorem. The latter condition, concerning
the smallness of $u(t)$ on the intervals $I^{-}_{i}$, follows from the result in Section~\ref{section-4.5}
provided that $\mu$ is sufficiently large, say $\mu > \mu^{**}$.
Arguing as in the proof of Theorem~\ref{th-6.1}, it is easy to note that $\mu^{**}$ does not depend on $k$.
\end{proof}

The above theorem provides the existence of $2^{k}-1$ distinct $kT$-periodic solutions of \eqref{eq-2.1} which
are positive and uniformly bounded in $\mathbb{R}$. Our goal now is to detect among these solutions the ``true''
subharmonics of order $k$ which do not belong to the same periodicity class. Figure~\ref{fig-02} gives
an explanation of what we are looking for.

\begin{figure}[h!]
\centering
\includegraphics[width=0.77\textwidth]{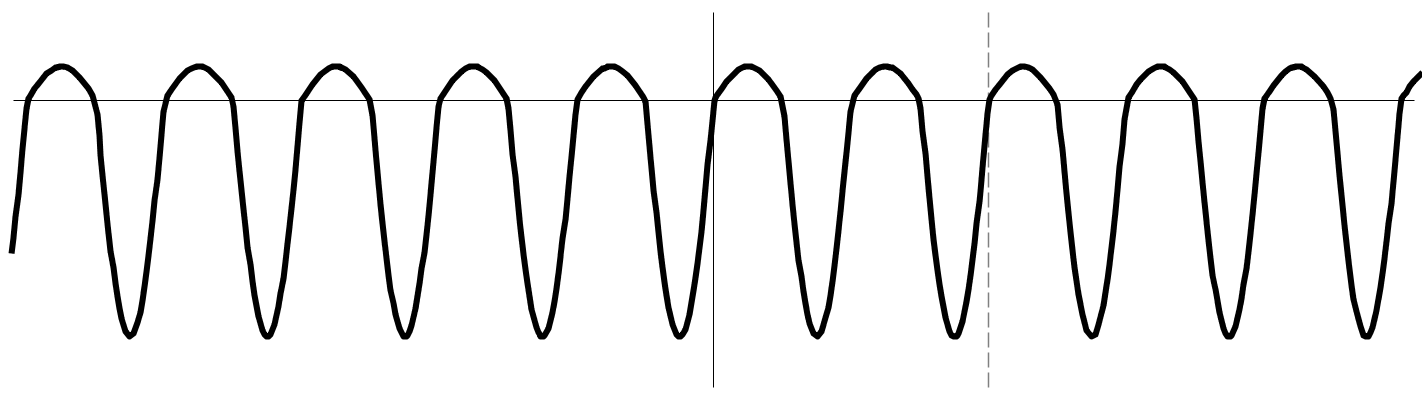}
\\\vspace*{10pt}
\includegraphics[width=0.26\textwidth]{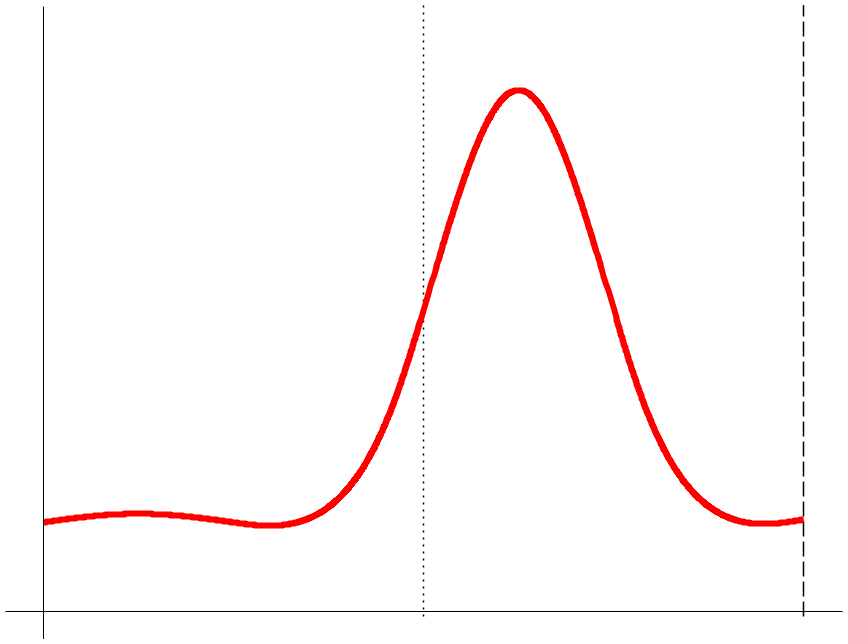}\qquad
\includegraphics[width=0.26\textwidth]{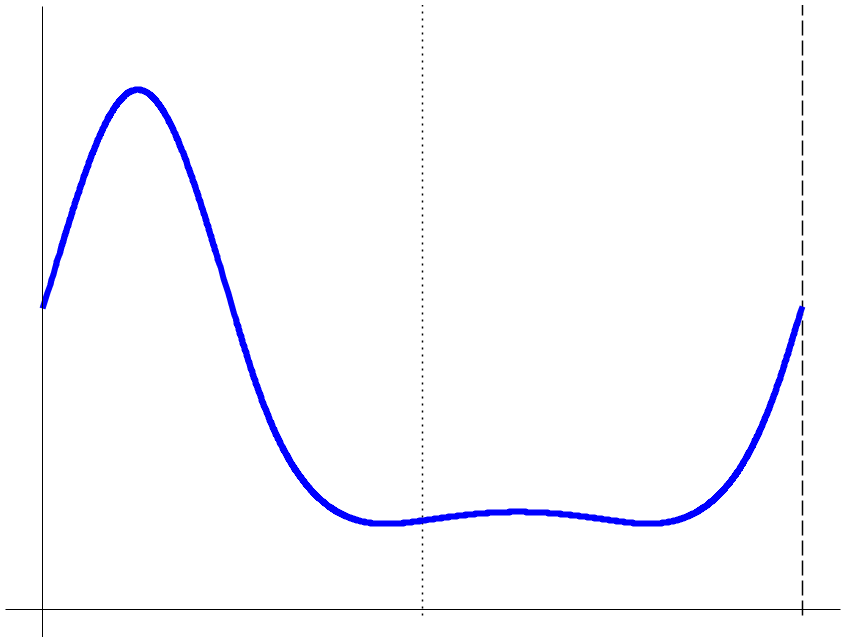}\qquad
\includegraphics[width=0.26\textwidth]{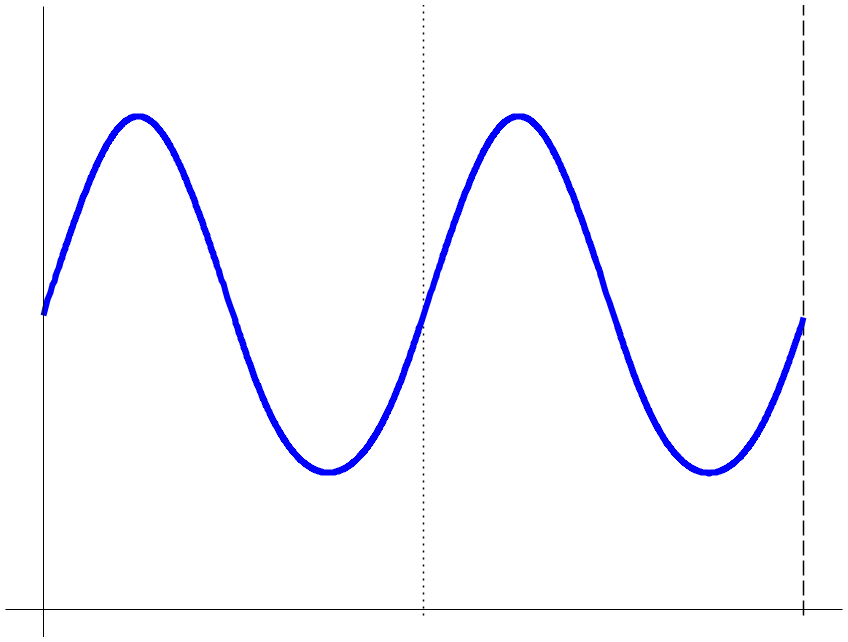}
\caption{\small{At the top we have shown the graph of the function $a_{\mu}(t)$, where $a(t) = \sin(2\pi t)$ and $\mu = 7$.
Using a numerical simulation we have studied the subharmonic solutions of order $k=2$ of equation \eqref{eq-2.1}
with $g(s) = \max\{0,100\, s\arctan|s|\}$. Clearly, $T=1$.
In the lower part, the figure shows the graphs of three $2$-periodic positive solutions,
whose existence is consistent with Theorem~\ref{th-6.2}.
The first two solutions are subharmonic solutions of order $2$ and the third one is a $1$-periodic solution.
As subharmonic solutions of order $2$, we consider only the first one, since the second one is a translation by $1$ of the first solution.
}}
\label{fig-02}
\end{figure}

In order to count the $k$-tuples corresponding to subharmonic solutions of order $k$ which are not equal up to translation (geometrically distinct),
we notice that the number we are looking for coincides with the number of \textit{binary Lyndon words of length $k$},
that is the number of binary strings inequivalent modulo rotation (cyclic permutation) of the digits and not having
a period smaller than $k$. Usually, in each equivalent class one selects the minimal element in the lexicographic ordering.
For instance, for the alphabet $\mathscr{A}= \{a,b\}$ and $k=4$, the corresponding binary Lyndon words of length $4$
are $aaab$, $aabb$, $abbb$. Note that the string $abab$ is not acceptable as it represents a sequence of period $2$ and
the string $bbaa$ is already counted as $aabb$. To give a formal definition,
consider an alphabet $\mathscr{A}$ which, in our context, is a nonempty totally ordered set of $n\geq 2$ symbols.
A \textit{$n$-ary Lyndon word of length $k$} is a string of $k$ digits of $\mathscr{A}$
which is strictly smaller in the lexicographic ordering than all of its nontrivial rotations.

The number of $n$-ary Lyndon words of length $k$ is given by \textit{Witt's formula}
\begin{equation}\label{eq-Lyndon}
\mathcal{S}_{n}(k) = \dfrac{1}{k} \sum_{l|k} \mu(l) \, n^{\frac{k}{l}},
\end{equation}
where $\mu(\cdot)$ is the M\"{o}bius function, defined on $\mathbb{N}\setminus\{0\}$ by $\mu(1) = 1$, $\mu(l) = (-1)^{s}$ if $l$ is the
product of $s$ distinct primes and $\mu(l) = 0$ otherwise (cf.~\cite[\S~5.1]{Lo-1997}).
Formula \eqref{eq-Lyndon} can be obtained by the \textit{M\"{o}bius inversion formula},
which is strictly related with the classical \textit{inclusion-exclusion principle}.

For instance, the values of $\mathcal{S}_{2}(k)$ (number of binary Lyndon words of length $k$) for $k=2,\ldots,10$ are $1$, $2$, $3$, $6$, $9$, $18$, $30$, $56$, $99$.

The following proposition provides an explicit formula of $\mathcal{S}_{n}(k)$ (for arbitrary integers $n,k \geq 2$), depending on the
\textit{prime factorization} of $k$.

\begin{proposition}\label{prop-6.1}
Let $n,k \geq 2$ be two integers. If the prime factorization of $k$ is
\begin{equation*}
k = p_{1}^{\alpha_{1}} \cdot p_{2}^{\alpha_{2}} \cdot p_{3}^{\alpha_{3}} \cdot \ldots \cdot p_{s}^{\alpha_{s}} = \prod_{i=1}^{s} p_{i}^{\alpha_{i}},
\end{equation*}
where $s$ is the number of distinct prime factors of $k$,
then the following formula holds
\begin{equation*}
\mathcal{S}_{n}(k) = \dfrac{1}{k} n^{k} + \dfrac{1}{k} \sum_{i=1}^{s}(-1)^{i}\sum_{\substack{j_{d}\in\{1,\ldots,s\} \\ j_{1}<j_{2} < \ldots < j_{i}}}
n^{\frac{k}{p_{j_{1}}\cdot \ldots \cdot p_{j_{i}}}}.
\end{equation*}
\end{proposition}

\begin{proof}
First of all, we observe that the divisors $l$ of the integer $k$ such that $\mu(l) \neq 0$ are the square-free factors of $k$,
hence $l=1$ (with $\mu(1)=1$) and the integers of the form $l=p_{j_{1}}\cdot \ldots \cdot p_{j_{i}}$
for $j_{d}\in\{1,\ldots,s\}$ (with $\mu(l)=(-1)^{i}$). The above formula immediately follows from \eqref{eq-Lyndon}.
\end{proof}

\begin{remark}\label{rem-6.3}
Although in this context formula \eqref{eq-Lyndon} and the more explicit one in Proposition~\ref{prop-6.1}
are related to the number of Lyndon words of length $k$ in an alphabet of size $n$, these formulas come out in different areas of mathematics.
Now we provide an overview of the several meanings of \eqref{eq-Lyndon}.

Still in combinatorics, it is not difficult to see that $\mathcal{S}_{n}(k)$ is also the number of \textit{aperiodic necklaces}
that can be made by arranging $k$ beads whose color is chosen from a list of $n$ colors.
The notions of \textit{Lyndon words} and \textit{necklaces} are also strictly related to \textit{de Bruijn sequences}.
We recall that a \textit{$n$-ary de Bruijn sequence of order $k$} is a circular string of characters
chosen in an alphabet of size $n$,
for which every possible subsequence of length $k$ appears as a substring of consecutive characters exactly once.
For more details about these concepts and other aspects of the formula in the context of combinatorics on words, we refer to \cite{Lo-1997, Ma-1891}
and the very interesting historical survey \cite[\S~4]{BePe-2007}.

The number $\mathcal{S}_{n}(k)$ has several meanings even outside combinatorics.
For instance, the integer $\mathcal{S}_{2}(k)$ (of binary Lyndon words of length $k$) corresponds to
the number of periodic points with minimal period $k$ in the iteration of the tent map $f(x) := 2 \min\{x,1-x\}$ on the unit interval (cf.~\cite{Du-1995},
also for more general formulas) and to the number of distinct cycles of minimal period $k$
in a shift dynamical system associated with a totally disconnected hyperbolic iterated function system (cf.~\cite[Lemma~1, p.~171]{Ba-1988}).
Concerning the more general formula for $\mathcal{S}_{n}(k)$, we just mention two other meanings.
The classical \textit{Witt's formula} (proved in 1937), which is still widely studied in algebra,
gives the dimensions of the homogeneous components of degree $k$ of the free
Lie algebra over a finite set with $n$ elements (cf.~\cite[Corollary~5.3.5]{Lo-1997}).
Moreover, in Galois theory, $\mathcal{S}_{n}(k)$ is also the number of
monic irreducible polynomials of degree $k$ over the finite field $\mathbb{F}_{n}$,
when $n$ is a prime power (in this context \eqref{eq-Lyndon} is also known as \textit{Gauss formula};
we refer to \cite[ch.~14, p.~588]{DuFo-2004} for a possible proof).

It is not possible to mention here all the other several implications of formula~\eqref{eq-Lyndon},
for example in symbolic dynamics, algebra, number theory and chaos theory. For this latter topic,
we only recall the recent paper \cite{JoSaYo-2010} where such numbers appear in connection with the study of
period-doubling cascades.

Further information and references can be found in \cite{GiRi-1961, KoRaWo-2014, Sl-oeis}.
$\hfill\lhd$
\end{remark}

\medskip

Using the above discussion, we achieve the following result.

\begin{theorem}\label{th-6.3}
Let $g \colon {\mathbb{R}}^{+} \to {\mathbb{R}}^{+}$ be a continuous function satisfying $(g_{*})$,
\begin{equation*}
g_{0} = 0 \quad \text{ and } \quad g_{\infty} = +\infty.
\end{equation*}
Let $a \colon \mathbb{R} \to \mathbb{R}$ be a $T$-periodic continuous function with minimal period $T$ such that there exist two consecutive zeros $\alpha < \beta$
so that $a(t) > 0$ for all $t \in \mathopen{]}\alpha,\beta\mathclose{[}$ and $a(t) < 0$ for all $t\in\mathopen{]}\beta,\alpha + T\mathclose{[}$.
Then there exists $\mu^{*}>0$ such that, for all $\mu>\mu^{*}$ and for every $k\geq 2$, equation \eqref{eq-2.1} has at least $\mathcal{S}_{2}(k)$ positive subharmonic solutions of order $k$.
\end{theorem}

\begin{proof}
We have to detect the subharmonic solutions of order $k$ among the $2^{k}-1$ distinct $kT$-periodic positive solutions of \eqref{eq-2.1}
provided by Theorem~\ref{th-6.2}. As remarked above, the number we are looking for is $\mathcal{S}_{2}(k)$. Therefore the thesis immediately follows.
\end{proof}

For the sake of simplicity, above we have considered only the particular case of a continuous periodic
sign-changing function $a(t)$ with minimal period $T$ and such that it has only one positive hump and one negative one
in a period interval. Moreover, we have taken a superlinear function $g(s)$.
We conclude this section by stating the analogous result for more general functions $a(t)$ and $g(s)$.

\begin{theorem}\label{th-6.4}
Let $g \colon {\mathbb{R}}^{+} \to {\mathbb{R}}^{+}$ be a continuous function satisfying $(g_{*})$,
\begin{equation*}
g_{0} = 0 \quad \text{ and } \quad g_{\infty} > \max_{i=1,\ldots,m} \lambda_{1}^{i}.
\end{equation*}
Let $a \colon \mathbb{R} \to \mathbb{R}$ be a $T$-periodic locally integrable function satisfying $(a_{*})$ with minimal period $T$.
Then there exists $\mu^{*}>0$ such that, for all $\mu>\mu^{*}$ and for every $k\geq 2$, equation \eqref{eq-2.1} has
at least $\mathcal{S}_{2^{m}}(k)$ positive subharmonic solutions of order $k$.
\end{theorem}

\begin{proof}
We only sketch the proof which is mimicked from those of Theorem~\ref{th-6.1} and of Theorem~\ref{th-6.2}, using Theorem~\ref{MainTheorem}.
To start, we need to be careful with the notation. For this reason, we call $J^{+}_{1},\ldots,J^{+}_{m}$ the $m$ intervals of positivity for $a(t)$
in the interval $\mathopen{[}0,T\mathclose{]}$ and $J^{-}_{1},\ldots,J^{-}_{m}$ the $m$ intervals of negativity for $a(t)$,
according to assumption $(a^{*})$. Consider an arbitrary integer $k\geq 2$.
The function $a(t)$ restricted to the interval $\mathopen{[}0,kT\mathclose{]}$ satisfies again
an assumption of the form $(a^{*})$, with respect to $mk$ intervals of positivity/negativity that we denote now with
$I^{\pm}_{1},\ldots,I^{\pm}_{mk}$, defined as
\begin{equation*}
I^{\pm}_{j + \ell m} = J^{\pm}_{j} + \ell T, \qquad j=1,\ldots,m, \quad \ell = 0,\ldots, k-1.
\end{equation*}
In other terms, in the interval $\mathopen{[}0,kT\mathclose{]}$ there are $mk$ closed subintervals
where $a(t) \succ 0$, separated by closed subintervals where $a(t) \prec 0$.
Then we can apply Theorem~\ref{MainTheorem}, looking for $kT$-periodic solutions.
In fact, by our main result, we have at least $2^{mk} -1$ positive periodic solutions of period $kT$
(which up to now is not necessarily the minimal period for the solutions).
More precisely, as in Theorem~\ref{th-6.2}, there exist $0 < r < R$ and $\mu^{**} > 0$ (depending on $m$ but not on $k$) such that,
for all $\mu > \mu^{**}$, given any nontrivial $k$-tuple $\mathcal{L}^{[k]} = (s_{\ell})_{\ell=0,\ldots,k-1}$ in the alphabet $\mathscr{A}:=\{0,1\}^{m}$ of size $2^{m}$
(hence, for $\ell=0,\ldots,k-1$, $s_{\ell}=(s_{\ell}^{j})_{j=1,\ldots,m}$), there exists at least one $kT$-periodic positive solution
\begin{equation*}
u(t) = u_{\mathcal{L}^{[k]}}(t)
\end{equation*}
of equation \eqref{eq-2.1} such that $\|u\|_{\infty} < R$ and
\begin{itemize}
\item $0<u(t)<r$ on $I^{+}_{j + \ell m}$, if $s_{\ell}^{j}=0$;
\item $r < u(\hat{t}) < R$ for some $\hat{t}\in I^{+}_{j + \ell m}$, if $s_{\ell}^{j}=1$;
\item $0<u(t)<r$ on $I^{-}_{i}$, for all $i=1,\ldots,mk$.
\end{itemize}
It remains to see whether, on the basis of the information we have on $u(t)$, we are able
first to determine the minimality of the period and next to distinguish among solutions do not belonging to
the same periodicity class. In view of the above listed properties of the solution $u(t)$,
our first problem is equivalent to choosing a string $\mathcal{L}^{[k]}$ having $k$ as a minimal period (when repeated cyclically).
For the second question, given any string of this kind, we count as the same all those strings (of length $k$) which are equivalent by cyclic permutations.
To choose exactly one string in each of these equivalent classes, we can take the minimal one in the lexicographic order.
As a consequence, we can conclude that there are so many nonequivalent $kT$-periodic solutions which are not $pT$-periodic for every $p= 1,\ldots, k-1$,
how many $2^{m}$-ary Lyndon words of length $k$.
Since we know that the equation does not possess exceptional solutions, we find that
for these subharmonic solutions $kT$ is precisely the minimal period.
\end{proof}

We have listed before some values of $\mathcal{S}_{2}(k)$ which give the number of subharmonic solutions in the setting of Theorem~\ref{th-6.3}.
Concerning the general case addressed in Theorem~\ref{th-6.4}, we observe that the number $\mathcal{S}_{2^{m}}(k)$, with $m\geq 2$, grows very fast with $k$.
For instance, the values of $\mathcal{S}_{2^{2}}(k)$ (number of quaternary Lyndon words of length $k$)
for $k=2,\ldots,10$ are $6$, $20$, $60$, $204$, $670$, $2340$, $8160$, $29120$, $104754$.

\subsection{Positive solutions with complex behavior}\label{section-6.3}

In this section we just outline a possible procedure in order to obtain the existence of
solutions which follow any preassigned coding described by two symbols, say $0$ and $1$,
that in our context will be interpreted as ``small'' and, respectively, ``large'' in the intervals where the weight is positive.
In other terms we are looking for the presence of a Bernoulli shift as a factor within the set of positive and bounded solutions.
Results in this direction are classical in the theory of dynamical systems (cf.~\cite{De-1989, Mo-1973, Sm-1967}) and have been achieved
in the variational setting as well (see, for instance, \cite{BoSeTe-2002, ByRa-2013, Se-1993}).
Even if the obtention of chaotic dynamics using topological degree or index theories is an established technique
(see \cite{CaKwMi-2000, SrWoZg-2005} and the references therein),
the achievement of similar results with our approach seems new in the literature.

Our proof is based on the above results about subharmonic solutions and on the following diagonal lemma, which is typical
in this context. Lemma~\ref{lem-6.1} is adapted from \cite[Lemma~8.1]{Kr-1968} and \cite[Lemma~4]{Ma-1996}.

\begin{lemma}\label{lem-6.1}
Let $f\colon\mathbb{R}\times\mathbb{R}^{d}\to\mathbb{R}^{d}$ be an $L^{1}$-Carath\'{e}odory function.
Let $(t_{n})_{n\in\mathbb{N}}$ be an increasing sequence of positive numbers and $(x_{n})_{n\in\mathbb{N}}$ be a sequence
of functions from $\mathbb{R}$ to $\mathbb{R}^{d}$ with the following properties:
\begin{itemize}
\item[$(i)$] $t_{n}\to+\infty$ as $n\to\infty$;
\item[$(ii)$] for each $n\in\mathbb{N}$, $x_{n}(t)$ is a solution of
\begin{equation}\label{eq-lem-6.1}
x'=f(t,x)
\end{equation}
defined on $\mathopen{[}-t_{n},t_{n}\mathclose{]}$;
\item[$(iii)$] for every $N\in\mathbb{N}$ there exists a bounded set $B_{N}\subseteq\mathbb{R}^{d}$ such that,
for each $n\geq N$, it holds that $x_{n}(t)\in B_{N}$ for every $t\in\mathopen{[}-t_{N},t_{N}\mathclose{]}$.
\end{itemize}
Then there exists a subsequence $(\tilde{x}_{n})_{n\in\mathbb{N}}$ of $(x_{n})_{n\in\mathbb{N}}$ which converges uniformly
on the compact subsets of $\mathbb{R}$ to a solution $\tilde{x}(t)$ of system \eqref{eq-lem-6.1};
in particular $\tilde{x}(t)$ is defined on $\mathbb{R}$ and, for each $N\in\mathbb{N}$,
it holds that $\tilde{x}(t)\in\overline{B_{N}}$ for all $t\in\mathopen{[}-t_{N},t_{N}\mathclose{]}$.
\end{lemma}

\begin{proof}
This result is classical and perhaps a proof is not needed.
We give a sketch of the proof for the reader's convenience, following \cite[Lemma~4]{Ma-1996}.

First of all we observe that, by the Carath\'{e}odory assumption, for each $N\in\mathbb{N}$ there exists
a measurable function $\rho_{N}\in L^{1}(\mathopen{[}-t_{N},t_{N}\mathclose{]},{\mathbb{R}}^{+})$ such that
\begin{equation*}
\|f(t,x)\|\leq \rho_{N}(t), \quad\text{ for a.e. } t\in [-t_{N},t_{N}] \text{ and for all } x\in B_{N}.
\end{equation*}
For every $N\in\mathbb{N}$ we also introduce the absolutely continuous function
\begin{equation*}
\mathcal{M}_{N}(t):= \int_{0}^{t} \rho_{N}(\xi) ~\!d\xi, \quad t\in \mathopen{[}-t_{N},t_{N}\mathclose{]}.
\end{equation*}
By hypothesis $(ii)$, we have that
\begin{equation*}
x_{n}(t) = x_{n}(0) + \int_{0}^{t} f(\xi,x_{n}(\xi)) ~\!d\xi, \quad \forall \, t\in\mathopen{[}-t_{n},t_{n}\mathclose{]},\; \forall \, n\in\mathbb{N},
\end{equation*}
and, by hypothesis $(iii)$, for every $N\in\mathbb{N}$ it follows that
\begin{equation*}
|x_{n}(t')-x_{n}(t'')| \leq |\mathcal{M}_{N}(t') - \mathcal{M}_{N}(t'')|, \quad \forall \, t',t''\in\mathopen{[}-t_{N},t_{N}\mathclose{]}, \; \forall \, n\geq N,
\end{equation*}
(cf.~\cite[p.~29]{Ha-1980}).
Consequently, the sequence $(x_{n})_{n\in\mathbb{N}}$ restricted to the interval $\mathopen{[}-t_{0},t_{0}\mathclose{]}$ is uniformly bounded
(by any constant which bounds in the Euclidean norm the set $B_{0}$) and equicontinuous.
By Ascoli-Arzel\`{a} theorem, it has a subsequence $(x_{n}^{0})_{n\in\mathbb{N}}$ which converges uniformly on $\mathopen{[}-t_{0},t_{0}\mathclose{]}$
to a continuous function named $\hat{x}_{0}$.
Similarly, the sequence $(x_{n}^{0})_{n\geq1}$ restricted to $\mathopen{[}-t_{1},t_{1}\mathclose{]}$ is a uniformly bounded and equicontinuous sequence
and has a subsequence $(x_{n}^{1})_{n\geq1}$ which converges uniformly on $\mathopen{[}-t_{1},t_{1}\mathclose{]}$
to a continuous function $\hat{x}_{1}$ such that $\hat{x}_{1}(t) = \hat{x}_{0}(t)$ for all $t\in \mathopen{[}-t_{0},t_{0}\mathclose{]}$.
Proceeding inductively in this way, we construct a sequence of sequences $(x_{n}^{N})_{n\geq N}$ so that
$(x_{n}^{N})_{n\geq N}$ is a subsequence of $(x_{n}^{N-1})_{n\geq N-1}$
and converges uniformly on $\mathopen{[}-t_{N},t_{N}\mathclose{]}$ to a continuous function $\hat{x}_{N}$ such that
$\hat{x}_{N}(t) = \hat{x}_{N-1}(t)$ for all $t\in \mathopen{[}-t_{N-1},t_{N-1}\mathclose{]}$.
By construction, we have that $\hat{x}_{N}(t)\in\overline{B_{N}}$ for all $t\in\mathopen{[}-t_{N},t_{N}\mathclose{]}$.
The diagonal sequence $(\tilde{x}_{n})_{n\in\mathbb{N}}:=(x_{n}^{n})_{n\in\mathbb{N}}$ converges uniformly on every compact interval to a function $\tilde{x}$ defined on $\mathbb{R}$
and such that $\tilde{x}(t) = \hat{x}_{N}(t)$ for all $t\in\mathopen{[}-t_{N},t_{N}\mathclose{]}$ and therefore,
$\tilde{x}(t)\in\overline{B_{N}}$ for all $t\in\mathopen{[}-t_{N},t_{N}\mathclose{]}$.
It remains to prove that $\tilde{x}(t)$ is a solution of \eqref{eq-lem-6.1} on $\mathbb{R}$.
Indeed, let $t\in \mathbb{R}$ be arbitrary but fixed and let us fix $N\in \mathbb{N}$ such that $t\in\mathopen{[}-t_{N},t_{N}\mathclose{]}$.
Passing to the limit as $n\to\infty$ in the identity
\begin{equation*}
\tilde{x}_{n}(t) = \tilde{x}_{n}(0) + \int_{0}^{t} f(\xi,\tilde{x}_{n}(\xi)) ~\!d\xi, \quad \forall \, n\geq N,
\end{equation*}
via the Lebesgue dominated convergence theorem, we obtain
\begin{equation*}
\tilde{x}(t) = \tilde{x}(0) + \int_{0}^{t} f(\xi,\tilde{x}(\xi)) ~\!d\xi.
\end{equation*}
For the arbitrariness of $t\in \mathbb{R}$ and the above integral relation, we conclude that
$\tilde{x}(t)$ is absolutely continuous and a solution of \eqref{eq-lem-6.1} (in the Carath\'{e}odory sense).
\end{proof}

If there exists a bounded set $B$ such that $B_{N} \subseteq B$ for all $N\in \mathbb{N}$,
then we have the stronger conclusion that $\tilde{x}(t)\in\overline{B}$ for all $t\in\mathbb{R}$
(which is precisely the result of \cite[Lemma~8.1]{Kr-1968} and \cite[Lemma~4]{Ma-1996}).

\medskip

An application of Lemma~\ref{lem-6.1} to the planar system
\begin{equation}\label{planar-system}
\begin{cases}
\, u' = y \\
\, y' = - cy - \bigl{(}a^{+}(t)-\mu a^{-}(t)\bigr{)}g(u)
\end{cases}
\end{equation}
will produce bounded solutions with any prescribed complex behavior.
In order to simplify the exposition, we suppose that the coefficient $a(t)$ is a continuous $T$-periodic function
of minimal period $T$ having a positive hump followed by a negative one in a period interval
(these are the same assumptions for the weight coefficient as in Theorem~\ref{th-6.1}). In this framework, the next result follows.

\begin{theorem}\label{th-6.5}
Let $g \colon {\mathbb{R}}^{+} \to {\mathbb{R}}^{+}$ be a continuous function satisfying $(g_{*})$,
\begin{equation*}
g_{0} = 0 \quad \text{ and } \quad g_{\infty} = +\infty.
\end{equation*}
Let $a \colon \mathbb{R} \to \mathbb{R}$ be a $T$-periodic continuous function with minimal period $T$ such that there exist two consecutive zeros $\alpha < \beta$
so that $a(t) > 0$ for all $t \in \mathopen{]}\alpha,\beta\mathclose{[}$ and $a(t) < 0$ for all $t\in\mathopen{]}\beta,\alpha + T\mathclose{[}$.
Then there exist $0 < r < R$ and $\mu^{**} > 0$ such that, for all $\mu > \mu^{**}$, given any two-sided sequence
$\mathcal{L} = (s_{i})_{i\in \mathbb{Z}}\in \{0,1\}^{\mathbb{Z}}$ which is not identically zero,
there exists at least one positive solution $u(t) = u_{\mathcal{L}}(t)$
of equation \eqref{eq-2.1} such that $\|u\|_{\infty} < R$ and
\begin{itemize}
\item $0<u(t)<r$ on $\mathopen{[}\alpha + iT,\beta+iT\mathclose{]}$, if ${s}_{i}=0$;
\item $r<u(\hat{t}_{i})<R$ for some $\hat{t}_{i}\in \mathopen{[}\alpha + iT,\beta+iT\mathclose{]}$, if ${s}_{i}=1$;
\item $0<u(t)<r$ on $\mathopen{[}\beta+iT,\alpha + (i+1)T\mathclose{]}$, for all $i\in \mathbb{Z}$.
\end{itemize}
\end{theorem}

\begin{proof}
Without loss of generality, we suppose that $\alpha = 0$ and set $\tau:=\beta - \alpha$, so that $a(t) > 0$ on $\mathopen{]}0,\tau\mathclose{[}$ and
$a(t) < 0$ on $\mathopen{]}\tau,T\mathclose{[}$. We also introduce the intervals
\begin{equation}\label{eq-intI}
I^{+}_{i}:= \mathopen{[}iT,\tau + iT\mathclose{]}, \quad I^{-}_{i}:= \mathopen{[}\tau + iT,(i+1)T\mathclose{]}, \quad i\in \mathbb{Z}.
\end{equation}
Let $0 < r < R$ and $\mu^{**}>0$ as in Theorem~\ref{th-6.1} and Theorem~\ref{th-6.2}.
One more time, we wish to emphasize the fact that, once we have fixed $r$, $R$ and $\mu > \mu^{**}$,
we can produce $kT$-periodic solutions following any $k$-periodic sequence of two symbols, \textit{independently on $k$}.
Accordingly, from this moment to the end of the proof, $r$, $R$ and $\mu > \mu^{**}$ are fixed.

Consider now an arbitrary sequence $\mathcal{L} = (s_{i})_{i\in\mathbb{Z}}\in \{0,1\}^{\mathbb{Z}}$ which is not identically zero.
We fix a positive integer $n_{0}$ such that there is at least an index
$i\in\{-n_{0},\ldots,n_{0}\}$ such that $s_{i}=1$. Then, for each $n \geq n_{0}$
we consider the $(2n+1)$-periodic sequence ${\mathcal{L}}^{n} =(s'_{i})_{i}\in \{0,1\}^{\mathbb{Z}}$
which is obtained by truncating $\mathcal{L}$ between $-n$ and $n$, and then repeating that string by periodicity.
An application of Theorem~\ref{th-6.2} on the periodicity interval $\mathopen{[}-nT,(n+1)T\mathclose{]}$ ensures the existence of
a positive periodic solution $u_{n}(t)$ such that $u_{n}(t + (2n+1)T) = u_{n}(t)$ for all $t\in\mathbb{R}$ and $\|u_{n}\|_{\infty} < R$.
According to Theorem~\ref{th-6.2}, we also know that $u_{n}(t) < r$ for all $t\in I^{+}_{i}$, if $s'_{i} = 0$,
$u_{n}(\hat{t}) > r$ for some $\hat{t}\in I^{+}_{i}$, if $s'_{i} = 1$, and $\max_{t\in I^{-}_{i}} u_{n}(t) < r$ (for each $i\in\mathbb{Z}$).

Notice that, for $C:= \max_{0\leq s\leq R} g(s)$, we have that
\begin{equation*}
|u''_{n}(t)|\leq |c| |u'_{n}(t)| + \bigl{(}|a^{+}(t)| + \mu |a^{-}(t)| \bigr{)} C, \quad \forall \, t\in\mathbb{R},
\end{equation*}
and hence,
\begin{equation}\label{eq-6.5}
\dfrac{|u''_{n}(t)|}{1 + |u'_{n}(t)|} \leq \psi_{\mu}(t), \quad \forall \, t \in \mathbb{R},
\end{equation}
where we have set $\psi_{\mu}(t): = |c| + (|a^{+}(t)| + \mu |a^{-}(t)|) C$.

Since the truncated string ${\mathcal{L}}^{n}$ contains at least one $s'_{i} = s_{i} = 1$, with $i\in\{-n_{0},\ldots,n_{0}\}$,
we know that each periodic function $u_{n}(t)$ has at least a local maximum point $\hat{t}_{n}\in \mathopen{]}-n_{0}T,n_{0}T + \tau\mathclose{[}$
and then $u'_{n}(\hat{t}_{n}) = 0$. Suppose now that $N \geq n_{0}$ is fixed and define the constant
\begin{equation*}
K_{N}:= \exp\biggl{(} (2N+1) \int_{0}^{T}\psi_{\mu}(t)~\!dt \biggr{)}.
\end{equation*}
We claim that
\begin{equation}\label{eq-6.6}
|u'_{n}(t)|\leq K_{N}, \quad \forall \, t\in \mathopen{[}-NT,(N+1)T\mathclose{]}, \; \forall \, n \geq N.
\end{equation}
Our claim follows from a Nagumo type argument as in \cite[ch.~I, \S~4]{DCHa-2006}.
Suppose, by contradiction, that \eqref{eq-6.6} is not true.
Hence, there exist some $n \geq N$ and a point $t^{*}_{n} \in \mathopen{[}-NT,(N+1)T\mathclose{]}$
such that $u'_{n}(t^{*}_{n}) > K_{N}$ or $u'_{n}(t^{*}_{n}) < -K_{N}$. In the first case there exists a maximal
interval $J \subseteq \mathopen{[}-NT,(N+1)T\mathclose{]}$ such that one of the following two possibilities occurs:
\begin{itemize}
\item $J = \mathopen{[}\xi_{0},\xi_{1}\mathclose{]}$ and
$u'_{n}(\xi_{0}) = 0$, $u'_{n}(\xi_{1}) > K_{N}$ with $u'_{n}(t) > 0$ for all $t\in \mathopen{]}\xi_{0},\xi_{1}\mathclose{]}$;
\item $J = \mathopen{[}\xi_{1},\xi_{0}\mathclose{]}$ and
$u'_{n}(\xi_{0}) = 0$, $u'_{n}(\xi_{1}) > K_{N}$ with $u'_{n}(t) > 0$ for all $t\in \mathopen{[}\xi_{1},\xi_{0}\mathclose{[}$.
\end{itemize}
Integrating $u''_{n}/(1 + |u'_{n}|)$ on $J$ and using \eqref{eq-6.5}, we obtain
\begin{equation*}
\begin{aligned}
\log(1 + K_{N}) &< \log(1 + |u'_{n}(\xi_{1})|) \leq \int_{J} \psi_{\mu}(t)~\!dt \\
&\leq \int_{-NT}^{(N+1)T} \psi_{\mu}(t)~\!dt = (2N+1) \int_{0}^{T}\psi_{\mu}(t)~\!dt = \log(K_{N}),
\end{aligned}
\end{equation*}
a contradiction. We have achieved a contradiction by assuming $u'_{n}(t^{*}_{n}) > K_{N}$.
A similar argument gives a contradiction if $u'_{n}(t^{*}_{n}) < -K_{N}$.

Now we write equation \eqref{eq-2.1} as a planar system \eqref{planar-system}. From the above remarks,
one can see that (up to a reparametrization of indices, counting from $n_{0}$)
assumptions $(i)$, $(ii)$ and $(iii)$ of Lemma~\ref{lem-6.1} are satisfied,
taking $t_{n}:=nT$, $f(t,x)=(y,-cy-(a^{+}(t)-\mu a^{-}(t))g(u))$, with $x=(u,y)$, and
\begin{equation*}
B_{N} := \bigl{\{} x\in\mathbb{R}^{2} \colon 0 < x_{1} < R, \; |x_{2}| \leq K_{N} \bigr{\}}, \quad N\in\mathbb{N},
\end{equation*}
as bounded set in $\mathbb{R}^{2}$.
By Lemma~\ref{lem-6.1}, there is a solution $\tilde{u}(t)$ of equation \eqref{eq-2.1} which is defined on $\mathbb{R}$ and such that
$0 \leq \tilde{u}(t) \leq R$ for all $t\in\mathopen{[}-NT,NT\mathclose{]}$, for each $N\in\mathbb{N}$. Then $\|\tilde{u}\|_{\infty} \leq R$.
Moreover, such a solution $\tilde{u}(t)$ is the limit of a subsequence $(\tilde{u}_{n})_{n}$
of the sequence of the periodic solutions $u_{n}(t)$.

We claim that
\begin{itemize}
\item $0<\tilde{u}(t)<r$ on $I^{+}_{i}$, if $s_{i}=0$;
\item $r<\tilde{u}(\hat{t}_{i})<R$ for some $\hat{t}_{i}\in I^{+}_{i}$, if $s_{i}=1$;
\item $0<\tilde{u}(t)<r$ on $I^{-}_{i}$, for all $i\in \mathbb{Z}$.
\end{itemize}
To prove our claim, let us fix $i\in \mathbb{Z}$ and consider the interval $I^{+}_{i}$ introduced in \eqref{eq-intI}.
For each $n \geq |i|$ (and $n\geq n_{0}$) the periodic solution $u_{n}(t)$ is defined on $\mathbb{R}$ and
such that $0 < u_{n}(t) < r$ for all $t\in I^{+}_{i}$, if ${s}_{i}=0$, or $\max_{t\in I^{+}_{i}} u_{n}(t) > r$, if ${s}_{i}=1$.
Passing to the limit on the subsequence $(\tilde{u}_{n})_{n}$, we obtain that
\begin{equation*}
0 \leq \tilde{u}(t) \leq r, \quad \forall \, t\in I^{+}_{i}, \; \text{ if } s_{i}=0,
\end{equation*}
or
\begin{equation*}
\max_{t\in I^{+}_{i}} \tilde{u}(t) \geq r, \quad \text{if } s_{i}=1,
\end{equation*}
respectively. With the same argument we also prove that
\begin{equation*}
0 \leq \tilde{u}(t)\leq r, \quad \forall \, t\in I^{-}_{i}, \; \forall \, i\in \mathbb{Z}.
\end{equation*}
By Remark~\ref{rem-4.8} we get that $\tilde{u}(t) < R^{*} \leq R$, for all $t\in \mathbb{R}$. Moreover,
since there exists at least one index $i\in \mathbb{Z}$ such that $s_{i} = 1$, we know that $\tilde{u}$ is not identically zero.
Hence, a maximum principle argument shows that $\tilde{u}(t)$ never vanishes. In conclusion, we have proved that
\begin{equation*}
0 < \tilde{u}(t) < R, \quad \forall \, t\in \mathbb{R}.
\end{equation*}
Next, we observe that
\begin{equation*}
\max_{t\in I^{+}_{i}} \tilde{u}(t) \neq r, \quad \forall \, i\in \mathbb{Z}.
\end{equation*}
Indeed, this is a consequence of Remark~\ref{rem-4.10}, using the fact that the solution $\tilde{u}(t)$ is
upper bounded by $R^{*}$ and, at the beginning, $\mu$ has been chosen large enough (note also that we apply that result in the case
$m =1$ and therefore the sets $I^{+}_{i} + \ell T$ of Remark~\ref{rem-4.10} reduce, in our case,
to the intervals $\mathopen{[}0,\tau\mathclose{]} + \ell T$).
Finally, using Remark~\ref{rem-4.12} we also deduce that
\begin{equation*}
\tilde{u}(t)<r, \quad \forall \, t\in I^{-}_{i}, \; \forall \, i\in \mathbb{Z}.
\end{equation*}
Our claim is thus verified and this completes the proof of the theorem.
\end{proof}

For the equation
\begin{equation*}
u'' + \bigl{(} a^{+}(t) - \mu a^{-}(t)\bigr{)} u^{3} = 0,
\end{equation*}
a version of Theorem~\ref{th-6.5} has been recently obtained in \cite{BaBoVe-jde2015}, under the supplementary condition that
in the strings of symbols the consecutive sequences of zeros are bounded in length. The proof of \cite[Theorem~2.1]{BaBoVe-jde2015}
and ours are completely different (the former one relies on variational techniques, ours on degree theory).
Our new contribution is twofold: on one side, we can deal with non Hamiltonian systems (indeed we can consider also a term of the form
$c u'$) and with a nonlinearity $g(s)$ which is not positively homogeneous; on the other hand, our approach allows to remove the
condition on bounded sequences of consecutive zeros. In any case, the two results are not completely comparable since
the way to associate a solution to a given string of symbols is different: the symbols $0$ and $1$ in our case are associated
to the maximum of a solution on $I^{+}_{i}$, while in \cite[Theorem~2.1]{BaBoVe-jde2015} are associated to an integral norm
on the same interval.

\medskip

We remark that Theorem~\ref{th-6.5} can be generalized at the same extent like Theorem~\ref{th-6.4} generalizes Theorem~\ref{th-6.3}.
Indeed, combining the proofs of Theorem~\ref{th-6.4} and Theorem~\ref{th-6.5}, we can obtain the following result (the proof is omitted).

\begin{theorem}\label{th-6.6}
Let $g \colon {\mathbb{R}}^{+} \to {\mathbb{R}}^{+}$ be a continuous function satisfying $(g_{*})$,
\begin{equation*}
g_{0} = 0 \quad \text{ and } \quad g_{\infty} > \max_{i=1,\ldots,m} \lambda_{1}^{i}.
\end{equation*}
Let $a \colon \mathbb{R} \to \mathbb{R}$ be a $T$-periodic locally integrable function satisfying $(a_{*})$ with minimal period $T$.
Then there exist $0 < r < R$ and $\mu^{**} > 0$ such that, for all $\mu > \mu^{**}$, given any two-sided sequence
$\mathcal{L} = (s_{\ell})_{\ell\in\mathbb{Z}}$ in the alphabet $\mathscr{A}:=\{0,1\}^{m}$
and not identically zero,
there exists at least one positive solution $u(t) = u_{\mathcal{L}}(t)$
of equation \eqref{eq-2.1} such that $\|u\|_{\infty} < R$ and the following properties hold
(where we set $s_{\ell}=(s_{\ell}^{i})_{i=1,\ldots,m}$, for each $\ell\in\mathbb{Z}$):
\begin{itemize}
\item $0<u(t)<r$ on $I^{+}_{i} + \ell T$, if $s_{\ell}^{i}=0$;
\item $r < u(\hat{t}) < R$ for some $\hat{t}\in I^{+}_{i} + \ell T$, if $s_{\ell}^{i}=1$;
\item $0<u(t)<r$ on $I^{-}_{i} + \ell T$, for all $i\in\{1,\ldots,m\}$ and for all $\ell \in \mathbb{Z}$.
\end{itemize}
\end{theorem}

\section{The Neumann boundary value problem}\label{section-7}

In this section we briefly describe how to obtain the results of Section~\ref{section-4} and Section~\ref{section-5}
for the Neumann boundary value problem.
For the sake of simplicity, we deal with the case $c=0$.
If $c\neq0$, we can produce analogous results writing equation \eqref{eq-2.1} as
\begin{equation*}
\bigl{(}u' e^{ct} \bigr{)}' + \tilde{a}_{\mu}(t)g(u) = 0, \quad \text{ with }\; \tilde{a}_{\mu}(t):= a_{\mu}(t)e^{ct},
\end{equation*}
and entering in the setting of coincidence degree theory for the linear operator $L \colon u\mapsto - (u' e^{ct})'$.
For the abstract framework, we refer to \cite{FeZa-2015ade}, where the existence of positive
solutions is analyzed.
Accordingly, we consider the BVP
\begin{equation}\label{NeumannBVP}
\begin{cases}
\, u'' + a_{\mu}(t) g(u) = 0 \\
\, u'(0)=u'(T)=0,
\end{cases}
\end{equation}
where $a \colon \mathopen{[}0,T\mathclose{]}\to \mathbb{R}$ is an integrable function satisfying condition $(a_{*})$
and $g(s)$ fulfils the same conditions as in the previous sections.
In particular, when we assume $(a_{*})$ we suppose that there exist $m\geq 2$ subintervals of
$\mathopen{[}0,T\mathclose{]}$ where the weight is non-negative
separated by $m-1$ subintervals where the weight is non-positive, namely there are $2m+2$ points
\begin{equation*}
0 = \tau_{0}\leq  \sigma_{1} < \tau_{1} < \ldots < \sigma_{i} < \tau_{i} < \ldots < \sigma_{m} < \tau_{m} \leq \sigma_{m+1}=T
\end{equation*}
such that $a(t)\succ 0$ on $\mathopen{[}\sigma_{i},\tau_{i}\mathclose{]}$ and
$a(t)\prec 0$ on $\mathopen{[}\tau_{i},\sigma_{i+1}\mathclose{]}$.

In this case, the abstract setting of Section~\ref{section-3} can be reproduced almost verbatim with $X:= \mathcal{C}(\mathopen{[}0,T\mathclose{]})$,
$Z:=L^{1}(\mathopen{[}0,T\mathclose{]})$ and $L \colon u\mapsto - u''$, by taking in $\text{\rm dom}\,L$
the functions of $X$ which are continuously differentiable with absolutely continuous derivative
and such that $u'(0) = u'(T) = 0$.
With the above positions $\ker L \cong {\mathbb{R}}$, $\text{\rm Im}\,L$, as well as the projectors $P$
and $Q$ are exactly the same as in Section~\ref{section-3}.
Then Theorem~\ref{MainTheorem} can be restated as follows.

\begin{theorem}\label{MainTheorem-Neumann}
Let $g \colon {\mathbb{R}}^{+} \to {\mathbb{R}}^{+}$ be a continuous function satisfying $(g_{*})$,
\begin{equation*}
g_{0} = 0 \quad \text{ and } \quad g_{\infty} > \max_{i=1,\ldots,m} \lambda_{1}^{i}.
\end{equation*}
Let $a \colon  \mathopen{[}0,T\mathclose{]} \to \mathbb{R}$ be an integrable function satisfying $(a_{*})$.
Then there exists $\mu^{*}>0$ such that for all $\mu>\mu^{*}$ problem \eqref{NeumannBVP} has at least $2^{m}-1$ positive solutions.
\end{theorem}

As in Theorem~\ref{MainTheorem}, the $2^{m}-1$ positive solutions are discriminated by the fact that $\max_{t\in I^{+}_{i}}u(t) < r$
or $r < \max_{t\in I^{+}_{i}}u(t) < R$, where $I^{+}_{i} = \mathopen{[}\sigma_{i},\tau_{i}\mathclose{]}$ is the $i$-th interval where the weight is non-negative (cf.~Remark~\ref{rem-4.1}).
The constants $\lambda_{1}^{i}$ (for $i=1,\ldots,m$) are the first eigenvalues of the eigenvalue problems in $I^{+}_{i}$
\begin{equation*}
\varphi'' + \lambda a(t) \varphi = 0, \quad \varphi|_{\partial I^{+}_{i}} = 0.
\end{equation*}
If $\sigma_{1} = \tau_{0} = 0$ (that is $a(t)$ starts with a first interval of non-negativity), we can take $\lambda^{1}_{1}$
as the first eigenvalue of the eigenvalue problem
\begin{equation*}
\varphi'' + \lambda a(t) \varphi = 0, \quad \varphi'(0) = \varphi(\tau_{1}) = 0,
\end{equation*}
while if
$\tau_{m} = \sigma_{m+1}=T$ (that is $a(t)$ ends with a last interval of non-negativity), we can take $\lambda^{m}_{1}$
as the first eigenvalue of the eigenvalue problem
\begin{equation*}
\varphi'' + \lambda a(t) \varphi = 0, \quad \varphi(\sigma_{m}) = \varphi'(T) = 0.
\end{equation*}

Clearly, for the Neumann problem \eqref{NeumannBVP} we can also reestablish
the corollaries in Section~\ref{section-5}. In particular, Corollary~\ref{cor-5.2} reads as follows.

\begin{corollary}\label{cor-7.1}
Let $g \colon {\mathbb{R}}^{+} \to {\mathbb{R}}^{+}$ be a continuous function satisfying $(g_{*})$,
\begin{equation*}
g_{0} = 0 \quad \text{ and } \quad g_{\infty} = +\infty.
\end{equation*}
Let $a \colon  \mathopen{[}0,T\mathclose{]} \to \mathbb{R}$ be an  integrable function satisfying $(a_{*})$.
Then there exists $\mu^{*}>0$ such that for all $\mu>\mu^{*}$ problem \eqref{NeumannBVP} has at least $2^{m}-1$ positive solutions.
\end{corollary}

In the sequel we are going to use also a variant for the Neumann problem of Corollary~\ref{cor-5.5} that we do not state here explicitly.

\subsection{Radially symmetric solutions}\label{section-7.1}

We show now a consequence of the above results to the study of a PDE in an annular domain.
In order to simplify the exposition, we assume the continuity of the weight function.
In this manner, the solutions we find are the ``classical'' ones (at least two times continuously differentiable).
The first part of the following presentation
is essentially borrowed from \cite[\S~4.1]{FeZa-2015ade}; however, the multiplicity result is a new
contribution.

Let $\|\cdot\|$ be the Euclidean norm in ${\mathbb{R}}^{N}$ (for $N \geq 2$) and let
\begin{equation*}
\Omega:= B(0,R_{2})\setminus B[0,R_{1}] = \bigl{\{}x\in {\mathbb{R}}^{N} \colon R_{1} < \|x\| < R_{2}\bigr{\}}
\end{equation*}
be an open annular domain, with $0 < R_{1} < R_{2}$.

We deal with the Neumann boundary value problem
\begin{equation}\label{eq-pde-rad}
\begin{cases}
\, -\Delta \,u = q_{\mu}(x)\,g(u) & \text{ in } \Omega \\ \vspace*{2pt}
\, \dfrac{\partial u}{\partial {\bf n}} = 0 & \text{ on } \partial\Omega,
\end{cases}
\end{equation}
where $q \colon \overline{\Omega}\to {\mathbb{R}}$ is a continuous function which is radially symmetric, namely
there exists a continuous scalar function ${\mathcal{Q}} \colon \mathopen{[}R_{1},R_{2}\mathclose{]}\to {\mathbb{R}}$
such that
\begin{equation*}
q(x) = {\mathcal{Q}}(\|x\|), \quad \forall \, x\in \overline{\Omega},
\end{equation*}
and
\begin{equation*}
q_{\mu}(x) := q^{+}(x)- \mu q^{-}(x), \quad {\mathcal{Q}}_{\mu}(r) := {\mathcal{Q}}^{+}(r) - \mu {\mathcal{Q}}^{-}(r).
\end{equation*}
We look for existence/nonexistence and multiplicity of radially symmetric positive solutions of \eqref{eq-pde-rad},
that are classical solutions such that $u(x) > 0$ for all $x\in \Omega$ and also $u(x) = {\mathcal{U}}(\|x\|)$,
where ${\mathcal{U}}$ is a scalar function defined on $\mathopen{[}R_{1},R_{2}\mathclose{]}$.

Accordingly, our study can be reduced to the search of positive solutions
of the Neumann boundary value problem
\begin{equation}\label{eq-rad}
{\mathcal{U}}''(r) + \dfrac{N-1}{r} \, {\mathcal{U}}'(r) + {\mathcal{Q}}_{\mu}(r) g({\mathcal{U}}(r)) = 0,
\quad {\mathcal{U}}'(R_{1}) = {\mathcal{U}}'(R_{2}) = 0.
\end{equation}
Using the standard change of variable
\begin{equation*}
t = h(r):= \int_{R_{1}}^{r} \xi^{1-N} ~\!d\xi
\end{equation*}
and defining
\begin{equation*}
T:= \int_{R_{1}}^{R_{2}} \xi^{1-N} ~\!d\xi, \quad r(t):= h^{-1}(t) \quad \text{and} \quad v(t)={\mathcal{U}}(r(t)),
\end{equation*}
we transform \eqref{eq-rad} into the equivalent problem
\begin{equation}\label{eq-rad1}
v'' + a_{\mu}(t) g(v) = 0, \quad v'(0) = v'(T) = 0,
\end{equation}
with
\begin{equation*}
a(t):= r(t)^{2(N-1)}{\mathcal{Q}}(r(t)).
\end{equation*}
Consequently, the Neumann boundary value problem \eqref{eq-rad1} is of the same form of \eqref{NeumannBVP}
and we can apply the previous results.

Accordingly, suppose also that
\begin{itemize}
\item [$(q_{*})$]
\textit{
there exist $2m+2$ points $R_{1} = \tau_{0} \leq \sigma_{1} < \tau_{1} < \ldots < \sigma_{m} < \tau_{m} \leq \sigma_{m+1} = R_{2}$ such that
\begin{equation*}
\begin{aligned}
& {\mathcal{Q}}(r) > 0 \; \text{ on } \; \mathopen{]}\sigma_{i},\tau_{i}\mathclose{[}, \; i=1,\ldots, m;\\
& {\mathcal{Q}}(r) < 0 \; \text{ on } \; \mathopen{]}\tau_{i},\sigma_{i+1}\mathclose{[}, \; i=0,\ldots, m.
\end{aligned}
\end{equation*}
}
\end{itemize}

Notice that condition
\begin{equation}\label{eq-7.5}
\int_{0}^{T} a_{\mu}(t)~\!dt < 0
\end{equation}
reads as
\begin{equation*}
0 > \int_{0}^{T}r(t)^{2(N-1)}{\mathcal{Q}}_{\mu}(r(t))~\!dt = \int_{R_{1}}^{R_{2}}r^{N-1}{\mathcal{Q}}_{\mu}(r)~\!dr.
\end{equation*}
Up to a multiplicative constant, the latter integral is the integral of $q(x)$ on $\Omega$,
using the change of variable formula for radially symmetric functions. Thus, $\mu>0$ satisfies \eqref{eq-7.5} if and only if $\mu$ satisfies
\begin{equation*}
\int_{\Omega}^{} q_{\mu}(x)~\!dx < 0.
\leqno{(q_{**})}
\end{equation*}
Similarly, the integral in \eqref{eq-7.5} is sufficiently negative (depending on $\mu$) if and only if the integral in $(q_{**})$
is negative enough (depending on $\mu$). With these premises, Corollary~\ref{cor-7.1} yields to the following result.

\begin{theorem}\label{th-7.1}
Let $g \colon {\mathbb{R}}^{+} \to {\mathbb{R}}^{+}$ be a continuous function satisfying $(g_{*})$,
\begin{equation*}
g_{0} = 0 \quad \text{ and } \quad g_{\infty} = +\infty.
\end{equation*}
Let $q(x)$ be a continuous (radial) weight function as above.
Then there exists $\mu^{*}>0$ such that for each $\mu > \mu^{*}$ problem \eqref{eq-pde-rad}
has at least $2^{m}-1$ positive radially symmetric solutions.
\end{theorem}

Corollary~\ref{cor-7.1} and Theorem~\ref{th-7.1} represent an extension of \cite{Bo-2011}, where the same result was obtained
(with a shooting type approach) for $m=2$. Another extension of \cite{Bo-2011}, for an arbitrary $m\geq 2$, has been recently
achieved in \cite{BaBoVe-jde2015} (using a variational approach) for a power type nonlinearity $g(s)$.

Adding the smoothness of $g(s)$, from Corollary~\ref{cor-5.5} we obtain the next result.

\begin{theorem}\label{th-7.2}
Let $g \colon {\mathbb{R}}^{+} \to {\mathbb{R}}^{+}$ be a continuously differentiable function satisfying $(g_{*})$,
\begin{equation*}
g_{0} = 0 \quad \text{ and } \quad g_{\infty} = +\infty.
\end{equation*}
Let $q(x)$ be a continuous (radial) weight function as above.
Then, for all $\mu>0$ such that $(q_{**})$ holds, there exist two constants $0 < \omega_{*} \leq \omega^{*}$
(depending on $\mu$) such that the Neumann boundary value problem
\begin{equation}\label{eq-pde-rad-2}
\begin{cases}
\, -\Delta \,u = \nu \, q_{\mu}(x)\,g(u) & \text{ in } \Omega \\ \vspace*{2pt}
\, \dfrac{\partial u}{\partial {\bf n}} = 0 & \text{ on } \partial\Omega
\end{cases}
\end{equation}
has no positive solutions for $0 < \nu < \omega_{*}$ and at least one positive solution for $\nu > \omega^{*}$.
Moreover there exists $\nu^{*} > 0$ such that for all $\nu >\nu^{*}$ there exists $\mu^{*}=\mu^{*}(\nu)$
such that for $\mu>\mu^{*}$ problem \eqref{eq-pde-rad-2} has at least $2^{m}-1$ positive solutions.
\end{theorem}

\appendix
\section{Mawhin's coincidence degree}\label{appendix-A}

This appendix is devoted to recalling some basic facts about a version
of Mawhin's coincidence degree for open and possibly unbounded sets that is used in the present paper.
For more details about the coincidence degree, proofs and applications, we refer to \cite{GaMa-1977,Ma-1979,Ma-1993} and the references therein.

Let $X$ and $Z$ be real Banach spaces and let
\begin{equation*}
L \colon \text{\rm dom}\,L (\subseteq X) \to Z
\end{equation*}
be a linear Fredholm mapping of index zero, i.e.~$\text{\rm Im}\,L$ is a closed subspace of $Z$ and
$\text{\rm dim}(\ker L) = \text{\rm codim}(\text{\rm Im}\,L)$ are finite.
We denote by $\ker L = L^{-1}(0)$ the kernel of $L$, by $\text{\rm Im}\,L\subseteq Z$ the image of $L$
and by $\text{\rm coker}\,L \cong Z/\text{\rm Im}\,L$ the complementary subspace of $\text{\rm Im}\,L$ in $Z$.
Consider the linear continuous projections
\begin{equation*}
P \colon X \to \ker L, \qquad Q \colon Z \to \text{\rm coker}\,L.
\end{equation*}
so that
\begin{equation*}
X = \ker L \oplus \ker P, \qquad Z = \text{\rm Im}\,L \oplus \text{\rm Im}\,Q.
\end{equation*}
We denote by
\begin{equation*}
K_{P} \colon \text{\rm Im}\,L \to \text{\rm dom}\,L \cap \ker P
\end{equation*}
the right inverse of $L$, i.e.~$L K_{P}(w) = w$ for each $w\in \text{\rm Im}\,L$.
Since $\ker L$ and $\text{\rm coker}\,L$ are finite dimensional vector spaces of the same dimension,
once an orientation on both spaces is fixed, we choose a linear orientation-preserving isomorphism $J \colon \text{\rm coker}\,L \to \ker L$.

Let
\begin{equation*}
N \colon X \to Z
\end{equation*}
be a nonlinear \textit{$L$-completely continuous} operator, namely $N$ is continuous and
$QN(B)$ and $K_{P}(Id-Q)N(B)$ are relatively compact sets, for each bounded set $B\subseteq X$.
For example, $N$ is $L$-completely continuous when $N$ is continuous, maps bounded sets to bounded sets
and $K_{P}$ is a compact linear operator.
Consider the \textit{coincidence equation}
\begin{equation}\label{eq-A.1}
Lu = Nu,\quad u\in \text{\rm dom}\,L.
\end{equation}
One can easily prove that equation \eqref{eq-A.1} is equivalent to the fixed point problem
\begin{equation}\label{eq-A.2}
u = \Phi(u):= Pu + JQNu + K_{P}(Id-Q)Nu, \quad u\in X.
\end{equation}
Notice that, under the above assumptions, $\Phi \colon X \to X$ is a completely continuous operator.
Therefore, applying Leray-Schauder degree theory to the operator equation \eqref{eq-A.2}, it is possible to solve equation \eqref{eq-A.1} when $L$ is not invertible.

\medskip

Let $\mathcal{O}\subseteq X$ be an open and \textit{bounded} set such that
\begin{equation*}
Lu \neq Nu, \quad \forall \, u\in \text{\rm dom}\,L \cap \partial\mathcal{O}.
\end{equation*}
In this case, the \textit{coincidence degree of $L$ and $N$ in $\mathcal{O}$} is defined as
\begin{equation*}
D_{L}(L-N,\mathcal{O}):= \text{deg}_{LS}(Id-\Phi,\mathcal{O},0),
\end{equation*}
where ``$\text{deg}_{LS}$'' denotes the Leray-Schauder degree. We underline that $D_{L}$ is independent
on the choice of the linear orientation-preserving isomorphism $J$ and of the projectors $P$ and $Q$.

Now we present an extension of the coincidence degree to open (possibly unbounded) sets, following the
standard approach used in the theory of fixed point index
to define the Leray-Schauder degree for locally compact maps on arbitrary open sets (cf.~\cite{Nu-1985,Nu-1993}).
Extensions of coincidence degree to the case of general open sets have been already considered in
previous articles (see for instance \cite{CaHeMaZa-1994, MaReZa-2000, Mo-1996}).

Consider an open set $\Omega\subseteq X$ and suppose that the solution set
\begin{equation*}
\text{\rm Fix}\,(\Phi,\Omega):= \bigl{\{}u\in {\Omega} \colon u =
\Phi u\bigr{\}} = \bigl{\{}u\in {\Omega}\cap \text{\rm dom}\,L \colon Lu = N u\bigr{\}}
\end{equation*}
is compact. According to the extension of Leray-Schauder degree, we can define
\begin{equation*}
\text{deg}_{LS}(Id - \Phi,\Omega,0):=\text{deg}_{LS}(Id - \Phi,\mathcal{V},0),
\end{equation*}
where $\mathcal{V}$ is an open and bounded set with $\text{\rm Fix}\,(\Phi,\Omega) \subseteq \mathcal{V} \subseteq \overline{\mathcal{V}} \subseteq \Omega$.
The definition is independent of the choice of $\mathcal{V}$.
In this case the \textit{coincidence degree} \textit{of $L$ and $N$ in $\Omega$} is defined as
\begin{equation*}
D_{L}(L-N,\Omega):= D_{L}(L-N,{\mathcal{V}}) =\text{deg}_{LS}(Id - \Phi,\mathcal{V},0),
\end{equation*}
with $\mathcal{V}$ as above.
Using the excision property of the Leray-Schauder degree, it is easy to check that
if $\Omega$ is an open and bounded set satisfying $Lu \neq Nu$, for all $u\in \partial\Omega\cap \text{\rm dom}\,L$,
this definition is the usual definition of coincidence degree described above.

\medskip

The main properties of the coincidence degree are the following.
\begin{itemize}
\item \textit{Additivity. }
Let $\Omega_{1}$, $\Omega_{2}$ be open and disjoint subsets of $\Omega$ such that $\text{\rm Fix}\,(\Phi,\Omega)\subseteq \Omega_{1}\cup\Omega_{2}$.
Then
\begin{equation*}
D_{L}(L-N,\Omega) = D_{L}(L-N,\Omega_{1})+ D_{L}(L-N,\Omega_{2}).
\end{equation*}
\item \textit{Excision. }
Let $\Omega_{0}$ be an open subset of $\Omega$ such that $\text{\rm Fix}\,(\Phi,\Omega)\subseteq \Omega_{0}$.
Then
\begin{equation*}
D_{L}(L-N,\Omega)=D_{L}(L-N,\Omega_{0}).
\end{equation*}
\item \textit{Existence theorem. }
If $D_{L}(L-N,\Omega)\neq0$, then $\text{\rm Fix}\,(\Phi,\Omega)\neq\emptyset$,
hence there exists $u\in {\Omega}\cap \text{\rm dom}\,L$ such that $Lu = Nu$.
\item \textit{Homotopic invariance. }
Let $H\colon\mathopen{[}0,1\mathclose{]}\times \Omega \to X$, $H_{\vartheta}(u) := H(\vartheta,u)$, be a continuous homotopy such that
\begin{equation*}
\mathcal{S}:=\bigcup_{\vartheta\in\mathopen{[}0,1\mathclose{]}} \bigl{\{}u\in \Omega\cap \text{\rm dom}\,L \colon Lu=H_{\vartheta}u\bigr{\}}
\end{equation*}
is a compact set and there exists an open neighborhood $\mathcal{W}$ of $\mathcal{S}$ such that $\overline{\mathcal{W}}\subseteq \Omega$ and
$(K_{P}(Id-Q)H)|_{\mathopen{[}0,1\mathclose{]}\times\overline{\mathcal{W}}}$ is a compact map.
Then the map $\vartheta\mapsto D_{L}(L-H_{\vartheta},\Omega)$ is constant on $\mathopen{[}0,1\mathclose{]}$.
\end{itemize}

The following two results are of crucial importance in the computation of the coincidence degree in open and bounded sets.
The proofs are omitted.
Lemma~\ref{lemma_Mawhin} is taken from \cite[Theorem~IV.1]{GaMa-1977} and \cite[Theorem~2.4]{Ma-1993}, while
Lemma~\ref{lem-abs-deg0} is a classical result (see \cite{Nu-1973}) adapted to the present setting
as in \cite[Lemma~2.2, Lemma~2.3]{FeZa-2015ade}. By ``$\text{deg}_{B}$'' we denote the Brouwer degree.

\begin{lemma}\label{lemma_Mawhin}
Let $L$ and $N$ be as above and let $\Omega\subseteq X$ be an open and bounded set.
Suppose that
\begin{equation*}
Lu \neq \vartheta Nu, \quad \forall \, u\in \text{\rm dom}\,L \cap \partial\Omega, \; \forall \, \vartheta\in\mathopen{]}0,1\mathclose{]},
\end{equation*}
and
\begin{equation*}
QN(u)\neq0, \quad \forall\, u\in \partial\Omega \cap \ker L.
\end{equation*}
Then
\begin{equation*}
D_{L}(L-N,\Omega) = \text{deg}_{B}(-JQN|_{\ker L},\Omega \cap \ker L,0).
\end{equation*}
\end{lemma}

\begin{lemma}\label{lem-abs-deg0}
Let $L$ and $N$ be as above and let $\Omega\subseteq X$ be an open and bounded set.
Suppose that there exist a vector $v\neq0$ and a constant $\alpha_{0} > 0$ such that
\begin{equation*}
Lu \neq Nu + \alpha v,
\quad \forall \, u\in \text{\rm dom}\,L \cap \partial\Omega, \; \forall\, \alpha\in \mathopen{[}0,\alpha_{0}\mathclose{]},
\end{equation*}
and
\begin{equation*}
Lu \neq Nu + \alpha_{0} v, \quad \forall \, u\in \text{\rm dom}\,L \cap \Omega.
\end{equation*}
Then
\begin{equation*}
D_{L}(L-N,\Omega) = 0.
\end{equation*}
\end{lemma}

Finally we state and prove a key lemma for the computation of the degree in open and unbounded sets.
This result is a more general version of Lemma~\ref{lem-abs-deg0}.

\begin{lemma}\label{lem-deg0-deFigueiredo}
Let $L$ and $N$ be as above and let $\Omega\subseteq X$ be an open set.
Suppose that there exist a vector $v\neq0$ and a constant $\alpha_{0} > 0$ such that
\begin{itemize}
 \item [$(a)$] $Lu \neq Nu + \alpha v$, for all $u\in \text{\rm dom}\,L \cap \partial\Omega$ and for all $\alpha\geq0$;
 \item [$(b)$] for all $\beta\geq0$ there exists $R_{\beta}>0$ such that if
 there exist $u\in \overline{\Omega}\cap \text{\rm dom}\,L$ and $\alpha\in\mathopen{[}0,\beta\mathclose{]}$ with
             $Lu = Nu + \alpha v$, then $\|u\|_{X}\leq R_{\beta}$;
 \item [$(c)$] there exists $\alpha_{0}>0$ such that $Lu \neq Nu + \alpha v$, for all $u\in \text{\rm dom}\,L \cap \Omega$ and $\alpha\geq\alpha_{0}$.
\end{itemize}
Then
\begin{equation*}
D_{L}(L-N,\Omega) = 0.
\end{equation*}
\end{lemma}

\begin{proof}
For $\alpha\geq 0$, let us consider the set
\begin{equation*}
\mathcal{R}_{\alpha}
:=\bigl{\{} u\in \overline{\Omega}\cap \text{\rm dom}\,L \colon Lu = Nu + \alpha v\bigr{\}}
= \bigl{\{} u\in \overline{\Omega} \colon u = \Phi u + \alpha v^{*}\bigr{\}},
\end{equation*}
where $v^{*} := JQv + K_{P}(Id-Q)v$.
Without loss of generality, we assume that $R_{\alpha'}<R_{\alpha''}$ for $\alpha'<\alpha''$.
By conditions $(a)$, for all $\alpha\geq0$, the solution set $\mathcal{R}_{\alpha}$ is disjoint from $\partial\Omega$.
Moreover, by conditions $(b)$ and $(c)$, $\mathcal{R}_{\alpha}$ is contained in $\Omega\cap B(0,R_{\alpha_{0}+1})$.
So $\mathcal{R}_{\alpha}$ is bounded, and hence compact.
In this manner we have proved that the coincidence degree $D_{L}(L-N-\alpha v,\Omega)$ is well defined for any $\alpha\geq 0$.

Now, condition $(c)$, together with the property of existence of solutions when the degree $D_{L}$ is nonzero,
implies that there exists $\alpha_{0}\geq 0$ such that
\begin{equation*}
D_{L}(L-N- \alpha_{0} v,\Omega)=0.
\end{equation*}
On the other hand, from condition $(b)$ applied to $\beta = \alpha_{0}$,
repeating the same argument as above, we find that the set
\begin{equation*}
\mathcal{S}
:= \bigcup_{\alpha\in \mathopen{[}0,\alpha_{0}\mathclose{]}} \mathcal{R}_{\alpha} \,
= \bigcup_{\alpha\in \mathopen{[}0,\alpha_{0}\mathclose{]}} \bigl{\{}u\in \overline{\Omega}\cap \text{\rm dom}\,L \colon Lu = N u + \alpha v\bigr{\}}
\end{equation*}
is a compact subset of $\Omega$. Hence, by the homotopic invariance of the coincidence degree, we have that
\begin{equation*}
D_{L}(L-N,\Omega) = D_{L}(L-N - \alpha_{0} v,\Omega) = 0.
\end{equation*}
This concludes the proof.
\end{proof}

\section{A combinatorial lemma}\label{appendix-B}

In this final appendix we present the key lemma for the proof of our main multiplicity result.
The proof is based on the same inductive argument of \cite[Lemma~4.1]{FeZa-2015jde}.

\begin{lemma}\label{lem-B1}
Let $\mathcal{I}\subseteq\{1,\ldots,n\}$ be a set of indices.
Suppose that for all ${\mathcal{J}}\subseteq {\mathcal{I}}$
the coincidence degree is defined on the sets $\Lambda^{\mathcal{J}}$ and $\Omega^{\mathcal{J}}$, with
\begin{equation}\label{eq-B1}
D_{L}(L-N,\Omega^{\emptyset}) = D_{L}(L-N,\Lambda^{\emptyset}) = 1
\end{equation}
and
\begin{equation}\label{eq-B2}
D_{L}(L-N,\Omega^{\mathcal{J}})=0, \quad \forall \, \emptyset\neq\mathcal{J}\subseteq {\mathcal{I}}.
\end{equation}
Then
\begin{equation}\label{eq-B3}
D_{L}(L-N,\Lambda^{\mathcal{I}})=(-1)^{\#\mathcal{I}}.
\end{equation}
\end{lemma}

\begin{proof}
First of all, we notice that, in view of \eqref{eq-B1}, the conclusion is trivially satisfied when $\mathcal{I} = \emptyset$.
Suppose now that $m:=\#{\mathcal{I}} \geq 1$.
We are going to prove our claim by using an inductive argument.
More precisely, for every integer $k$ with $0\leq k\leq m$, we introduce the property ${\mathscr{P}(k)}$
which reads as follows:
\textit{the formula
\begin{equation*}
D_{L}(L-N,\Lambda^{\mathcal{J}})=(-1)^{\#\mathcal{J}}
\end{equation*}
holds for each subset ${\mathcal{J}}$ of ${\mathcal{I}}$ having at most $k$ elements.}
In this manner, if we are able to prove ${\mathscr{P}}(m)$, then \eqref{eq-B3} follows.

\medskip
\noindent
\textit{Verification of ${\mathscr{P}}(0)$. } It follows by hypothesis \eqref{eq-B1}.

\medskip
\noindent
\textit{Verification of ${\mathscr{P}}(1)$. }
Condition \eqref{eq-B1} covers the case $\mathcal{J} =\emptyset$.
For ${\mathcal{J}}=\{j\}$, with $j\in {\mathcal{I}}$, we have
\begin{equation*}
\begin{aligned}
D_{L}(L-N,\Lambda^{\mathcal{J}})
    & = D_{L}(L-N,\Lambda^{\{j\}}) = D_{L}(L-N,\Omega^{\{j\}}\setminus \Lambda^{\emptyset})
\\  & = 0-1=-1 =(-1)^{\#\mathcal{J}}.
\end{aligned}
\end{equation*}

\medskip
\noindent
\textit{Verification of ${\mathscr{P}}(k-1)\Rightarrow{\mathscr{P}}(k)$, for $1\leq k\leq m$. }
Assuming the validity of ${\mathscr{P}}(k-1)$ we have that the formula is true for every subset of ${\mathcal{I}}$
having at most $k-1$ elements. Therefore, in order to prove ${\mathscr{P}}(k)$, we have only to check that the formula is true for an
arbitrary subset $\mathcal{J}$ of $\mathcal{I}$ with $\#{\mathcal{J}} = k$.
Writing $\Omega^{\mathcal{J}}$ as the disjoint union
\begin{equation*}
\Omega^{\mathcal{J}} = \Lambda^{\mathcal{J}} \cup \bigcup_{\mathcal{K}\subsetneq\mathcal{J}}\Lambda^{\mathcal{K}},
\end{equation*}
by the inductive hypothesis and assumption \eqref{eq-B2}, we obtain
\begin{equation*}
\begin{aligned}
&       D_{L}(L-N,\Lambda^{\mathcal{J}})
      = D_{L}(L-N,\Omega^{\mathcal{J}})-\sum_{\mathcal{K}\subsetneq\mathcal{J}}D_{L}(L-N,\Lambda^{\mathcal{K}}) =
\\  & = 0-\sum_{\mathcal{K}\subsetneq\mathcal{J}}(-1)^{\#\mathcal{K}}
      = -\sum_{\mathcal{K}\subseteq\mathcal{J}}(-1)^{\#\mathcal{K}}+(-1)^{\#\mathcal{J}}.
\end{aligned}
\end{equation*}
Observe now that
\begin{equation*}
\sum_{\mathcal{K}\subseteq\mathcal{J}}(-1)^{\#\mathcal{K}} = 0,
\end{equation*}
due to the fact that in a finite set there are so many subsets of even cardinality how many subsets of odd cardinality.
Thus we conclude that
\begin{equation*}
D_{L}(L-N,\Lambda^{\mathcal{J}}) = (-1)^{\#\mathcal{J}}.
\end{equation*}
Therefore ${\mathscr{P}}(k)$ is proved.
\end{proof}

In Section~\ref{section-4.1} we have proved \eqref{deg-emptyset} and \eqref{deg-J},
which correspond to the hypotheses \eqref{eq-B1} and \eqref{eq-B2} in the above combinatorial lemma.
The conclusion given by \eqref{eq-B3} guarantees the validity of \eqref{eq-deg-Lambda}.

\section*{Acknowledgment}
We thank Alberto Boscaggin for interesting discussions on the subject of the present paper.

\bibliographystyle{elsart-num-sort}
\bibliography{Feltrin_Zanolin_biblio}

\begin{thebibliography}{10}
\expandafter\ifx\csname url\endcsname\relax
  \def\url#1{\texttt{#1}}\fi
\expandafter\ifx\csname urlprefix\endcsname\relax\def\urlprefix{URL }\fi

\bibitem{Ad-1975}
R.~A. Adams, Sobolev spaces, Academic Press, New York-London, 1975, {P}ure and
  {A}pplied {M}athematics, Vol. 65.

\bibitem{AlTa-1996}
S.~Alama, G.~Tarantello, Elliptic problems with nonlinearities indefinite in
  sign, J. Funct. Anal. 141 (1996) 159--215.

\bibitem{AlAr-1977}
S.~Aljan{\v{c}}i{\'c}, D.~Arandelovi{\'c}, {$\mathscr{O}$}-regularly varying
  functions, Publ. Inst. Math. (Beograd) (N.S.) 22 (36) (1977) 5--22.

\bibitem{Ba-1988}
M.~Barnsley, Fractals everywhere, Academic Press, Inc., Boston, MA, 1988.

\bibitem{BaBoVe-jde2015}
V.~L. Barutello, A.~Boscaggin, G.~Verzini, Positive solutions with a complex
  behavior for superlinear indefinite {ODE}s on the real line, J. Differential
  Equations 259 (2015) 3448--3489.

\bibitem{BeCaDoNi-1994}
H.~Berestycki, I.~Capuzzo-Dolcetta, L.~Nirenberg, Superlinear indefinite
  elliptic problems and nonlinear {L}iouville theorems, Topol. Methods
  Nonlinear Anal. 4 (1994) 59--78.

\bibitem{BeCaDoNi-1995}
H.~Berestycki, I.~Capuzzo-Dolcetta, L.~Nirenberg, Variational methods for
  indefinite superlinear homogeneous elliptic problems, NoDEA Nonlinear
  Differential Equations Appl. 2 (1995) 553--572.

\bibitem{BePe-2007}
J.~Berstel, D.~Perrin, The origins of combinatorics on words, European J.
  Combin. 28 (2007) 996--1022.

\bibitem{BoGoHa-2005}
D.~Bonheure, J.~M. Gomes, P.~Habets, Multiple positive solutions of superlinear
  elliptic problems with sign-changing weight, J. Differential Equations 214
  (2005) 36--64.

\bibitem{Bo-2011}
A.~Boscaggin, A note on a superlinear indefinite {N}eumann problem with
  multiple positive solutions, J. Math. Anal. Appl. 377 (2011) 259--268.

\bibitem{BoFeZa-pp2015}
A.~Boscaggin, G.~Feltrin, F.~Zanolin, Pairs of positive periodic solutions of
  nonlinear {ODEs} with indefinite weight: a topological degree approach for
  the super-sublinear case, Proc. Roy. Soc. Edinburgh Sect. A{}, to appear.

\bibitem{BoZa-2013}
A.~Boscaggin, F.~Zanolin, Subharmonic solutions for nonlinear second order
  equations in presence of lower and upper solutions, Discrete Contin. Dyn.
  Syst. 33 (2013) 89--110.

\bibitem{BoSeTe-2002}
E.~Bosetto, E.~Serra, S.~Terracini, Generic-type results for chaotic dynamics
  in equations with periodic forcing terms, J. Differential Equations 180
  (2002) 99--124.

\bibitem{BrHi-1997}
R.~C. Brown, D.~B. Hinton, Norm eigenvalue bounds for some weighted
  {S}turm-{L}iouville problems, in: General inequalities, 7 ({O}berwolfach,
  1995), vol. 123 of Internat. Ser. Numer. Math., Birkh\"auser, Basel, 1997,
  pp. 129--144.

\bibitem{Bu-1976a}
G.~J. Butler, Oscillation theorems for a nonlinear analogue of {H}ill's
  equation, Quart. J. Math. Oxford Ser. (2) 27 (1976) 159--171.

\bibitem{Bu-1976}
G.~J. Butler, Rapid oscillation, nonextendability, and the existence of
  periodic solutions to second order nonlinear ordinary differential equations,
  J. Differential Equations 22 (1976) 467--477.

\bibitem{ByRa-2013}
J.~Byeon, P.~H. Rabinowitz, On a phase transition model, Calc. Var. Partial
  Differential Equations 47 (2013) 1--23.

\bibitem{CaLa-1996}
N.~P. C{\'a}c, A.~C. Lazer, On second order, periodic, symmetric, differential
  systems having subharmonics of all sufficiently large orders, J. Differential
  Equations 127 (1996) 426--438.

\bibitem{CaDaPa-2002}
A.~Capietto, W.~Dambrosio, D.~Papini, Superlinear indefinite equations on the
  real line and chaotic dynamics, J. Differential Equations 181 (2002)
  419--438.

\bibitem{CaHeMaZa-1994}
A.~Capietto, M.~Henrard, J.~Mawhin, F.~Zanolin, A continuation approach to some
  forced superlinear {S}turm-{L}iouville boundary value problems, Topol.
  Methods Nonlinear Anal. 3 (1994) 81--100.

\bibitem{CaKwMi-2000}
M.~C. Carbinatto, J.~Kwapisz, K.~Mischaikow, Horseshoes and the {C}onley index
  spectrum, Ergodic Theory Dynam. Systems 20 (2000) 365--377.

\bibitem{CoLe-1955}
E.~A. Coddington, N.~Levinson, Theory of ordinary differential equations,
  McGraw-Hill Book Company, Inc., New York-Toronto-London, 1955.

\bibitem{DCHa-2006}
C.~De~Coster, P.~Habets, Two-point boundary value problems: lower and upper
  solutions, vol. 205 of Mathematics in Science and Engineering, Elsevier B.
  V., Amsterdam, 2006.

\bibitem{De-1989}
R.~L. Devaney, An introduction to chaotic dynamical systems, Addison-Wesley
  Studies in Nonlinearity, 2nd ed., Addison-Wesley Publishing Company, Advanced
  Book Program, Redwood City, CA, 1989.

\bibitem{DiZa-1993}
T.~R. Ding, F.~Zanolin, Subharmonic solutions of second order nonlinear
  equations: a time-map approach, Nonlinear Anal. 20 (1993) 509--532.

\bibitem{Dj-1998}
D.~Djur{\v{c}}i{\'c}, {$\mathscr{O}$}-regularly varying functions and strong
  asymptotic equivalence, J. Math. Anal. Appl. 220 (1998) 451--461.

\bibitem{Du-1995}
B.~S. Du, The minimal number of periodic orbits of periods guaranteed in
  {S}harkovski\u\i 's theorem, Bull. Austral. Math. Soc. 31 (1985) 89--103.

\bibitem{DuFo-2004}
D.~S. Dummit, R.~M. Foote, Abstract algebra, 3rd ed., John Wiley \& Sons, Inc.,
  Hoboken, NJ, 2004.

\bibitem{El-1974}
S.~B. Eliason, Lyapunov inequalities and bounds on solutions of certain second
  order equations, Canad. Math. Bull. 17 (1974) 499--504.

\bibitem{ErWa-1994}
L.~H. Erbe, H.~Wang, On the existence of positive solutions of ordinary
  differential equations, Proc. Amer. Math. Soc. 120 (1994) 743--748.

\bibitem{FeZa-2015ade}
G.~Feltrin, F.~Zanolin, Existence of positive solutions in the superlinear case
  via coincidence degree: the {N}eumann and the periodic boundary value
  problems, Adv. Differential Equations 20 (2015) 937--982.

\bibitem{FeZa-2015jde}
G.~Feltrin, F.~Zanolin, Multiple positive solutions for a superlinear problem:
  a topological approach, J. Differential Equations 259 (2015) 925--963.

\bibitem{GaMa-1977}
R.~E. Gaines, J.~L. Mawhin, Coincidence degree, and nonlinear differential
  equations, Lecture Notes in Mathematics, Vol. 568, Springer-Verlag,
  Berlin-New York, 1977.

\bibitem{GaHaZa-2003}
M.~Gaudenzi, P.~Habets, F.~Zanolin, An example of a superlinear problem with
  multiple positive solutions, Atti Sem. Mat. Fis. Univ. Modena 51 (2003)
  259--272.

\bibitem{GaHaZa-2004}
M.~Gaudenzi, P.~Habets, F.~Zanolin, A seven-positive-solutions theorem for a
  superlinear problem, Adv. Nonlinear Stud. 4 (2004) 149--164.

\bibitem{GiRi-1961}
E.~N. Gilbert, J.~Riordan, Symmetry types of periodic sequences, Illinois J.
  Math. 5 (1961) 657--665.

\bibitem{GiGo-2009}
P.~M. Gir{\~a}o, J.~M. Gomes, Multibump nodal solutions for an indefinite
  superlinear elliptic problem, J. Differential Equations 247 (2009)
  1001--1012.

\bibitem{GoReLoGo-2000}
R.~G{\'o}mez-Re{\~n}asco, J.~L{\'o}pez-G{\'o}mez, The effect of varying
  coefficients on the dynamics of a class of superlinear indefinite
  reaction-diffusion equations, J. Differential Equations 167 (2000) 36--72.

\bibitem{GrKoWa-2008}
J.~R. Graef, L.~Kong, H.~Wang, Existence, multiplicity, and dependence on a
  parameter for a periodic boundary value problem, J. Differential Equations
  245 (2008) 1185--1197.

\bibitem{Ha-1980}
J.~K. Hale, Ordinary differential equations, 2nd ed., Robert E. Krieger
  Publishing Co., Inc., Huntington, N.Y., 1980.

\bibitem{Ha-1982}
P.~Hartman, Ordinary differential equations, 2nd ed., Birkh\"auser, Boston,
  Mass., 1982.

\bibitem{JoSaYo-2010}
M.~R. Joglekar, E.~Sander, J.~A. Yorke, Fixed points indices and
  period-doubling cascades, J. Fixed Point Theory Appl. 8 (2010) 151--176.

\bibitem{KoRaWo-2014}
T.~Kociumaka, J.~Radoszewski, W.~Rytter, Computing {$k$}-th {L}yndon word and
  decoding lexicographically minimal de {B}ruijn sequence, in: Combinatorial
  pattern matching, vol. 8486 of Lecture Notes in Comput. Sci., Springer, Cham,
  2014, pp. 202--211.

\bibitem{Kr-1968}
M.~A. Krasnosel'ski{\u\i}, The operator of translation along the trajectories
  of differential equations, Translations of Mathematical Monographs, Vol. 19,
  American Mathematical Society, Providence, R.I., 1968.

\bibitem{Lo-1997}
M.~Lothaire, Combinatorics on words, Cambridge Mathematical Library, Cambridge
  University Press, Cambridge, 1997.

\bibitem{Ma-1891}
P.~A. MacMahon, Applications of a theory of permutations in circular procession
  to the theory of numbers, Proc. London Math. Soc. S1-23 (1891/92) 305--313.

\bibitem{Ma-1979}
J.~Mawhin, Topological degree methods in nonlinear boundary value problems,
  vol.~40 of CBMS Regional Conference Series in Mathematics, American
  Mathematical Society, Providence, R.I., 1979.

\bibitem{Ma-1993}
J.~Mawhin, Topological degree and boundary value problems for nonlinear
  differential equations, in: Topological methods for ordinary differential
  equations ({M}ontecatini {T}erme, 1991), vol. 1537 of Lecture Notes in Math.,
  Springer, Berlin, 1993, pp. 74--142.

\bibitem{Ma-1996}
J.~Mawhin, Bounded solutions of nonlinear ordinary differential equations, in:
  Non-linear analysis and boundary value problems for ordinary differential
  equations ({U}dine), vol. 371 of CISM Courses and Lectures, Springer, Vienna,
  1996, pp. 121--147.

\bibitem{MaReZa-2000}
J.~Mawhin, C.~Rebelo, F.~Zanolin, Continuation theorems for
  {A}mbrosetti-{P}rodi type periodic problems, Commun. Contemp. Math. 2 (2000)
  87--126.

\bibitem{MiTa-1988}
R.~Michalek, G.~Tarantello, Subharmonic solutions with prescribed minimal
  period for nonautonomous {H}amiltonian systems, J. Differential Equations 72
  (1988) 28--55.

\bibitem{Mo-1996}
P.~Morassi, A note on the construction of coincidence degree, Boll. Un. Mat.
  Ital. A (7) 10 (1996) 421--433.

\bibitem{Mo-1973}
J.~Moser, Stable and random motions in dynamical systems, Princeton University
  Press, Princeton, N.J., 1973, with special emphasis on celestial mechanics,
  Hermann Weyl Lectures, the Institute for Advanced Study, Princeton, N.J.,
  Annals of Mathematics Studies, No. 77.

\bibitem{Nu-1973}
R.~D. Nussbaum, Periodic solutions of some nonlinear, autonomous functional
  differential equations. {II}, J. Differential Equations 14 (1973) 360--394.

\bibitem{Nu-1985}
R.~D. Nussbaum, The fixed point index and some applications, vol.~94 of
  S\'eminaire de Math\'ematiques Sup\'erieures [Seminar on Higher Mathematics],
  Presses de l'Universit\'e de Montr\'eal, Montreal, QC, 1985.

\bibitem{Nu-1993}
R.~D. Nussbaum, The fixed point index and fixed point theorems, in: Topological
  methods for ordinary differential equations ({M}ontecatini {T}erme, 1991),
  vol. 1537 of Lecture Notes in Math., Springer, Berlin, 1993, pp. 143--205.

\bibitem{Pa-2003}
D.~Papini, Prescribing the nodal behaviour of periodic solutions of a
  superlinear equation with indefinite weight, Atti Sem. Mat. Fis. Univ. Modena
  51 (2003) 43--63.

\bibitem{PaZa-2004}
D.~Papini, F.~Zanolin, On the periodic boundary value problem and chaotic-like
  dynamics for nonlinear {H}ill's equations, Adv. Nonlinear Stud. 4 (2004)
  71--91.

\bibitem{Pi-2013}
J.~P. Pinasco, Lyapunov-type inequalities, Springer Briefs in Mathematics,
  Springer, New York, 2013.

\bibitem{Pl-1966}
V.~A. Pliss, Nonlocal problems of the theory of oscillations, Academic Press,
  New York-London, 1966.

\bibitem{SaCo-1964}
G.~Sansone, R.~Conti, Non-linear differential equations, Revised edition.
  Translated from the Italian by Ainsley H. Diamond. International Series of
  Monographs in Pure and Applied Mathematics, Vol. 67, A Pergamon Press Book.
  The Macmillan Co., New York, 1964.

\bibitem{Se-1993}
{\'E}.~S{\'e}r{\'e}, Looking for the {B}ernoulli shift, Ann. Inst. H.
  Poincar\'e Anal. Non Lin\'eaire 10 (1993) 561--590.

\bibitem{Sl-oeis}
N.~J.~A. Sloane, The on-line encyclopedia of integer sequences, published
  electronically at http://oeis.org, 2010, Sequence A001037.

\bibitem{Sm-1967}
S.~Smale, Differentiable dynamical systems, Bull. Amer. Math. Soc. 73 (1967)
  747--817.

\bibitem{SrWoZg-2005}
R.~Srzednicki, K.~W{\'o}jcik, P.~Zgliczy{\'n}ski, Fixed point results based on
  the {W}a\.zewski method, in: Handbook of topological fixed point theory,
  Springer, Dordrecht, 2005, pp. 905--943.

\bibitem{TeVe-2000}
S.~Terracini, G.~Verzini, Oscillating solutions to second-order {ODE}s with
  indefinite superlinear nonlinearities, Nonlinearity 13 (2000) 1501--1514.

\bibitem{Wa-1965}
P.~Waltman, An oscillation criterion for a nonlinear second order equation, J.
  Math. Anal. Appl. 10 (1965) 439--441.

\bibitem{Wa-1994}
H.~Wang, On the existence of positive solutions for semilinear elliptic
  equations in the annulus, J. Differential Equations 109 (1994) 1--7.

\end{thebibliography}

\bigskip
\begin{flushleft}

{\small{\it Preprint}}

{\small{\it July 2015}}

\end{flushleft}

\end{document}